\documentclass[12pt,reqno]{amsart}
\usepackage{hyperref}
\usepackage{geometry}
\usepackage{enumitem}
\usepackage{amsmath}
\usepackage{amssymb}
\usepackage{amsthm}
\usepackage{amsrefs}
\usepackage{amsfonts}
\usepackage[dvipsnames]{xcolor}
\usepackage{mathtools}
\usepackage{stackengine}
\usepackage{verbatim}
\usepackage{tikz-cd} 
\usepackage{xcolor}
\usepackage{longtable}
\usepackage{array}
\usepackage{booktabs}
\usepackage{tabularx}

\makeatletter
\@namedef{subjclassname@2020}{%
  \textup{2020} MSC}
\makeatother

\hypersetup{
	colorlinks   = true, 
	urlcolor     = blue, 
	linkcolor    = purple, 
	citecolor   = blue 
}

\def\XXint#1#2#3{{\setbox0=\hbox{$#1{#2#3}{\int}$ }
\vcenter{\hbox{$#2#3$ }}\kern-.6\wd0}}

\makeatletter
\newcommand*{\rom}[1]{\expandafter\@slowromancap\romannumeral #1@}

\newcommand{\ind}{\protect\raisebox{2pt}{$\chi$}}

\newcommand{\disp}{\operatorname{Error}}
\newcommand{\prim}{\mathrm{prim}}

\newcommand{\SL}{\mathrm{SL}}
\newcommand{\sspan}{\operatorname{span}}

\newcommand{\GL}{\mathrm{GL}}
\newcommand{\M}{\mathrm{M}}
\newcommand{\X}{\mathcal{X}}
\newcommand{\XA}{\mathcal{X}^{\mathbb{A}}}
\newcommand{\Y}{\mathcal{Y}}
\newcommand{\E}{\mathcal{E}}
\newcommand{\R}{\mathbb{R}}
\newcommand{\Q}{\mathbb{Q}}

\newcommand{\e}{\varepsilon}

\newcommand{\one}{\mathbf{1}}
\newcommand{\A}{\mathbb{A}}
\newcommand{\Z}{\mathbb{Z}}
\newcommand{\hZ}{\widehat{\mathbb{Z}}}
\newcommand{\hZp}{\widehat{\mathbb{Z}}_{\prim}}

\newcommand{\Sphere}{\mathbb{S}}

\newcommand{\N}{\mathbb{N}}

\newcommand{\sed}{\mathcal{S}_{\epsilon}}
\newcommand{\ssed}{{\mathcal{S}^\sharp_{\epsilon}}}
\newcommand{\ued}{{\mathcal{U}_{\epsilon}}}
\newcommand{\ueda}{{\Tilde{\mathcal{U}}_{\epsilon}}}
\newcommand{\seda}{\Tilde{\mathcal{S}}_{\epsilon}}
\newcommand{\sseda}{{\Tilde{\mathcal{S}}^\sharp_{\epsilon}}}
\newcommand{\Le}{L_{\epsilon}}

\newcommand{\tmu}{\tilde{\mu}}

\newcommand{\Cc}{{C}_{c}^{\infty}}

\newcommand{\bfe}{\mathbf{e}}

\newcommand{\el}{\infty}
\newcommand{\proj}{\operatorname{Proj}}
\newcommand{\cR}{\mathcal{R}}
\newcommand{\Primes}{\mathbb{P}}
\newcommand{\ZZ}{\mathcal{Z}}
\newcommand{\Mat}{\M_{m \times n}(\R)}
\renewcommand{\vector}{\mathfrak{v}}
\newcommand{\cL}{\mathcal{C}}
\newcommand{\cl}{\operatorname{cl}}
\newcommand{\Int}{\operatorname{int}}
\newcommand{\J}{\mathbb{J}}

\title{Counting and Joint equidistribution of approximates}

\begin{document}
\theoremstyle{plain}
\newtheorem{thm}{Theorem}[section]
\newtheorem{lem}[thm]{Lemma}
\newtheorem{prop}[thm]{Proposition}
\newtheorem{cor}[thm]{Corollary}

\theoremstyle{definition}
\newtheorem{defn}[thm]{Definition}
\newtheorem{exm}[thm]{Example}
\newtheorem{nexm}[thm]{Non Example}
\newtheorem{prob}[thm]{Problem}

\theoremstyle{remark}
\newtheorem{rem}[thm]{Remark}

\author{Gaurav Aggarwal}
\address{\textbf{Gaurav Aggarwal} \\
School of Mathematics,
Tata Institute of Fundamental Research, Mumbai, India 400005}
\email{gaurav@math.tifr.res.in}

\author{Anish Ghosh}
\address{\textbf{Anish Ghosh} \\
School of Mathematics,
Tata Institute of Fundamental Research, Mumbai, India 400005}
\email{ghosh@math.tifr.res.in}

\date{}

\thanks{ A.\ G.\ gratefully acknowledges support from a J. C. Bose grant and a grant from the Infosys foundation to the Infosys Chandrasekharan Random Geometry Centre. G. \ A.\ and  A.\ G.\ gratefully acknowledge a grant from the Department of Atomic Energy, Government of India, under project $12-R\&D-TFR-5.01-0500$. }



\begin{abstract} 
In this paper, we consider the problem of counting Diophantine inequalities with multiple natural constraints. We prove a very general result in this setting using dynamical techniques. More precisely, we consider the joint asymptotic distribution of $\varepsilon$-Diophantine approximates of matrices in several aspects. Our main results describe the resulting limiting measures for almost every matrix. Multiplicative Diophantine approximation is treated for the first time, and a number of Diophantine corollaries are derived. While we treat the general case of approximation of matrices, our results are already new for the case of simultaneous Diophantine approximation of vectors. Our approach is dynamical and is based on the construction of an appropriate Poincar\'{e} section for certain diagonal group actions on the space of unimodular lattices along with multiple mixing. The main new idea in our paper is a method that allows us to treat actions of higher rank groups.  
\end{abstract}

\maketitle

\tableofcontents

\section{Introduction}\label{Introduction}

For a real number $\alpha$, consider the inequality $$0 < q |q\alpha +p| < \varepsilon.$$

In Diophantine approximation, one seeks to obtain integer solutions to the above inequality, often with additional arithmetic information. For example, congruence conditions could be imposed, and in higher dimensions, the shape of the associated unimodular lattice could be considered, as well as the directions of the approximates. Thus, to the above Diophantine inequality, one can associate a number of arithmetic and geometric objects, and one could recast Diophantine approximation as the investigation of the distribution of these objects. All of these properties have individually been the subject of intense {activities} of late. An answer to a general form of this question was recently obtained by Shapira and Weiss \cite{SW22}. As a consequence of very general results, they showed that for almost every $\alpha$, these objects are jointly equidistributed in an appropriate phase space, equipped with a natural measure.\\  

It is natural to ask if one can obtain a similar satisfactory answer to the inequality $$0 < q|q\alpha +p_1|  |q\beta + p_2| < \varepsilon?$$

It turns out that one can indeed build a robust theory of approximates for \emph{multiplicative} inequalities like the one above, and this is the main subject of the present paper. In fact, we will deal with a very general set-up incorporating \emph{simultaneous} approximation of vectors as well as \emph{dual} approximation of linear forms. We will offer a unified treatment by studying Diophantine properties of matrices. We will further consider \emph{products} and also replace the $L^{\infty}$-norm with arbitrary norms. To a finite set of matrices, we will attach a packet of Diophantine data. That is,

\begin{enumerate}
\item the error term, namely $ q|q\alpha +p_1|  |q\beta + p_2|$ in the simplest case; this is the setting of classical multiplicative Diophantine approximation cf. \cite{BugSurvey}; 
\item the projection onto a unit sphere, this is just $ \left( \frac{q\alpha + p_1}{|q\alpha + p_1|}, \frac{q\beta + {p_2}}{|q\beta + p_2|} \right)$ for our toy example, but becomes more interesting in higher dimensions; cf. \cites{AGT, AlamGhoshTrans, KSW} for earlier work;
\item congruence conditions on the approximates $(q, p_1, p_2)$. This goes back to \cite{Szusz58} and has been recently studied in cf. \cite{AlamGhoshYu};
\item and finally, an object which provides a measure of the relative sizes of $|q|, |q\alpha +p_1| $ and $|q\beta +p_2|$ in terms of an associated unimodular lattice. This is sometimes referred to as the \emph{shape} of the lattice cf. \cites{Schmidt98, AES}. 

\end{enumerate}

We have only mentioned a representative sample of prior works on the above questions which have been extensively worked on. The work of Shapira and Weiss \cite{SW22} provides a unified general treatment. However, their work and the preceding works deal with classical Diophantine approximation. A key point of the present paper is to provide a unified general answer to the distribution of the above Diophantine data, in the \emph{multiplicative setting}. This presents considerable challenges, both conceptual and technical, and new ideas are needed to overcome them. From a conceptual point of view, the main issue that we have needed to address is the construction of a good (Poincar\'{e}) cross-section for actions of higher rank diagonal groups on the space of lattices. As far as we are aware, this is the first work where such a cross-section has been constructed and its properties studied. This issue also comes with technical difficulties. Further technical challenges are encountered while studying the genericity properties of adelic higher rank diagonal group actions. We refer the reader to Section \ref{sec:keying}  for a discussion of the key ideas in our paper.\\

A consequence of our main results gives that for almost every $\alpha, \beta$, the aforementioned constituents of the Diophantine packet jointly equidistribute in a natural phase space equipped with an explicit measure.  Our main theorem can be viewed as a very general \emph{counting result} in Diophantine approximation.\\

In order to state our main results, we begin by introducing the necessary definitions and notation. This paper is, by necessity, somewhat heavy on notation and draws on concepts from several distinct areas of mathematics. For the reader’s convenience, a table of notation is provided in Appendix~\ref{Table of notations}.

\subsection*{$\e$-approximation}

Suppose $m_1, \ldots , m_k$ and $n_1, \ldots, n_r$ are positive integers for some $ r,k \in \N$. Fix norms on each of $\R^{m_1}, \ldots, \R^{m_k}$ and $\R^{n_1}, \ldots, \R^{n_r}$. With slight abuse of notation, we denote each of these norms by $\|.\|$.

We denote $m=m_1+ \cdots + m_k$, $n= n_1 + \cdots + n_r$ and $d= m+n$. We define norms on $\R^m,$ $\R^n$ and $\R^d$ as
{\begin{align}
    \|(x_1, \ldots, x_k)\|&= \max_i \|x_i\| \quad \text{ for all } (x_1, \ldots, x_k) \in \R^{m_1} \times \cdots \times \R^{m_k} = \R^m \label{eq: def norm 1} \\
    \|(y_1, \ldots, y_r)\| &= \max_j \|y_j\|\quad \text{ for all }  (y_1, \ldots, y_r) \in \R^{n_1} \times \cdots \times \R^{n_r} = \R^n \label{eq: def norm 2} \\
    \|(x,y)\| &= \max\{ \|x\|, \|y\|\} \quad \text{ for all }  (x,y) \in \R^m \times \R^n = \R^d. \label{eq: def norm 3}
\end{align}}

For simplicity of notation later, we define the projection map $\varrho_1$ (resp. $\varrho_2$) from $\R^d= \R^m \times \R^n$ onto $\R^m$ (resp. $\R^n$). Similarly define for all $1 \leq i \leq k$, the projection maps $\rho_i$ from $ \R^m= \R^{m_1} \times \cdots \times \R^{m_k}$ onto the $i$-th component, $\R^{m_i}$ . Also define for all $1 \leq j \leq r$, the projection map $\rho_j': \R^n= \R^{n_1} \times \cdots \times \R^{n_r} \rightarrow \R^{n_j}$.

{Fix $\e > 0$, and fix positive real numbers $0 < \eta_1, \ldots, \eta_k$.} Given $\theta = (\theta_{ij}) \in \M_{m \times n}(\R)$, we call a vector $ (p,q ) \in \Z^{m} \times \Z^n$ an $\e$-approximation ({with respect to} the decomposition $m=m_1+ \cdots + m_k,$ $n= n_1 + \cdots +n_r$, choice of $\eta_1, \ldots, \eta_k${, and the choice of norms on $\R^{m_1}, \ldots, \R^{m_k}, \R^{n_1}, \ldots, \R^{n_r}$}) if 
\begin{align}
\nonumber \gcd(p,q) &=1, \\
   \nonumber \|\rho_i(p+ \theta q)\| &\leq \eta_i \text{ for all } i=1, \ldots, k,\\
    \label{eq: def e approx} 
    \left( \prod_{i=1}^k \|\rho_i(p+\theta q)\|^{m_i} \right)&. \left(\prod_{j=1}^r \max \{ 1, \|\rho_j'(q)\|^{n_j}\} \right) \leq \e.
\end{align}

\begin{rem}
    Clearly, the definition of $\e$-approximation of $\theta \in \M_{m \times n}(\R)$ depends on the choice of decomposition $m=m_1+ \cdots + m_k$, $n= n_1 + \cdots + n_r$, the choice of norms on $\R^{m_1}, \ldots, \R^{n_r}$ as well as the choice of $\eta_1, \ldots, \eta_k$. For the case $r = 1$ and $m_1 = m_2= \dots = m_k = 1$, {the equation} \eqref{eq: def e approx} exactly coincides with classical multiplicative Diophantine approximation.  
\end{rem}

Let us {record} some easy observations: 
\begin{enumerate}
\item { Without loss of generality, we may assume that
    $$
    \min_{q_i \in \mathbb{Z}^{n_i} \setminus \{0\}} \|q_i\| \geq 1 \quad \text{for all } i.
    $$
    Indeed, the general case can be reduced to this one by rescaling the norms on $\mathbb{R}^{n_i}$ by a factor $c > 0$ and replacing $\varepsilon$ with $\varepsilon \cdot c^{-n}$, for $c$ large enough. We will assume this normalisation throughout the rest of the paper.}
     \item {For Lebesgue almost every $\theta \in \mathcal{M}_{m \times n}(\mathbb{R})$, there are no solutions to \eqref{eq: def e approx} for which $\|\rho_i(p + \theta q)\| = 0$ for some $i$. Hence, we may assume that 
    $$
    \prod_{i=1}^k \|\rho_i(p + \theta q)\|^{m_i} > 0.
    $$
    In particular, if $\|\rho_j'(q)\| \neq 0$ for all $1 \leq j \leq r$, then the condition \eqref{eq: def e approx} may be replaced by
    \begin{align}
        \label{eq: def modified e approx}
        0 < \left( \prod_{i=1}^k \|\rho_i(p + \theta q)\|^{m_i} \right) \cdot \left( \prod_{j=1}^r \|\rho_j'(q)\|^{n_j} \right) \leq \varepsilon.
    \end{align}}
    
\end{enumerate}
\vspace{0.2in}
    
\subsection*{Natural Objects to consider} \hfill
{We now describe some natural objects associated with an $\varepsilon$-approximate pair $(p, q)$ for a matrix $\theta \in \mathcal{M}_{m \times n}(\mathbb{R})$ satisfying \eqref{eq: def modified e approx}.}

\subsubsection{The Error Term} 
The first natural object to consider is the \emph{error term}
\begin{align}
    \label{defdisp}
    \left( \prod_{i=1}^k \|\rho_i(p + \theta q)\|^{m_i} \right) \cdot \left( \prod_{j=1}^r \|\rho_j'(q)\|^{n_j} \right),
\end{align}
which lies in the interval $[0, \varepsilon]$. For ease of notation, {we denote the map that sends $(\theta, p, q)$ to the expression in \eqref{defdisp} by $\disp(.)$.}

\subsubsection{Projections}
The second natural object to associate with any Diophantine approximation $(p,q)$ is the direction of the vector $(p + \theta q, q)$, which corresponds to its \emph{projection} onto the unit {sphere; see, for instance, \cite{AGT}. This object can be written explicitly as:}
\begin{align}
    \label{defproj}
    \left( \frac{\rho_1(p + \theta q)}{\|\rho_1(p + \theta q)\|}, \ldots, \frac{\rho_k(p + \theta q)}{\|\rho_k(p + \theta q)\|}, \frac{\rho_1'(q)}{\|\rho_1'(q)\|}, \ldots, \frac{\rho_r'(q)}{\|\rho_r'(q)\|} \right).
\end{align}

For simplicity in later discussions, we define the map
$$
\proj: \M_{m \times n}(\R) \times \Z^m \times \Z^n \setminus \disp^{-1}(0) \rightarrow \Sphere^{m_1} \times \cdots \times \Sphere^{m_k} \times \Sphere^{n_1} \times \cdots \times \Sphere^{n_r}
$$
{to be the map} sending $(\theta, p, q)$ to the expression in \eqref{defproj}, where $\Sphere^l$ denotes the unit sphere in $\R^l$ with respect to the chosen norm on $\R^l$.

\subsubsection{The congruence condition}

In Diophantine approximation, {one is often interested in a subclass of approximates that satisfy prescribed \emph{congruence conditions}}. This is equivalent to studying the distribution of approximates $(p,q)$ in $(\Z/N\Z)^d$ for all $N \in \N$, {which can be recast in terms of their} distribution in $\hZ^d$, the $d$-fold Cartesian product of $\hZ$, the profinite completion of $\Z$. Thus, the third natural object to consider is the image of $(p,q)$ in $\hZp^d$, where $\hZp^d$ denotes the closure of the image of $\Z_\prim^d$ under the natural inclusion $\Z^d \hookrightarrow \hZ^d$. 

See \cite{AlamGhoshYu} for Diophantine approximation with congruence conditions, as well as \cite{SW22}; \cite{AG23} and \cite{AlamGhoshHan} for counting problems with congruence conditions, and \cite{borda2023} for limit theorems with congruence conditions.

\subsubsection{Relative Size and Derived Lattices}
\label{ sec: The Relative Size}
{
The three objects introduced above capture much of the geometric and arithmetic structure underlying $\e$-approximates. However, they do not yet reflect a crucial aspect: the \emph{relative sizes} of the quantities
\[
\|\rho_1(p+\theta q)\|, \ldots, \|\rho_k(p+\theta q)\|, \quad \|\rho_1'(q)\|, \ldots, \|\rho_r'(q)\|.
\]
To encode this information, we associate to each $\e$-approximate a suitably scaled unimodular lattice.}

{
Given $\theta \in \M_{m \times n}(\R)$, define the unimodular lattice
\begin{align}
    \label{eq:def lambda theta}
    \Lambda_\theta = \begin{pmatrix}
        I_m & \theta \\ & I_n
    \end{pmatrix} \Z^d.
\end{align}
By construction, $\Lambda_\theta$ contains the vector $(p+\theta q, q)$ as {a} primitive element. To incorporate the relative magnitudes of the projected components, we consider the scaled lattice
\[
\Lambda_\theta(p,q) = \begin{pmatrix}
    \|\rho_1(p+\theta q)\|^{-1} I_{m_1} \\ & \ddots \\ && \|\rho_r'(q)\|^{-1} I_{n_r}
\end{pmatrix} \Lambda_\theta.
\]
{ This lattice is obtained by scaling $\Lambda_\theta$ so that the vector $(p+\theta q, q)$ maps to $\proj(\theta, p, q)$. Note that $\Lambda_\theta(p,q)$ is not unimodular—it has covolume equal to $\disp(\theta, p, q)^{-1}$—but it records the sizes of the individual components in a precise geometric way.} }

{
More specifically, for each $1 \leq i \leq k$, the quantity $\|\rho_i(p+\theta q)\|^{-m_i}$ is the covolume of the $m_i$-dimensional lattice
\[
\Lambda_\theta(p,q) \cap \left( \{0\}^{m_1 + \cdots + m_{i-1}} \times \R^{m_i} \times \{0\}^{m_{i+1} + \cdots + m_k + n} \right)
\]
within its ambient subspace. Similarly, for each $1 \leq j \leq r$, the quantity $\|\rho_j'(q)\|^{-n_j}$ equals the covolume of the $n_j$-dimensional lattice
\[
\varrho_2(\Lambda_\theta(p,q)) \cap \left( \{0\}^{n_1 + \cdots + n_{j-1}} \times \R^{n_j} \times \{0\}^{n_{j+1} + \cdots + n_r} \right).
\]
Thus, $\Lambda_\theta(p,q)$ faithfully encodes the size data of the projected components.}

{
However, this lattice also retains information captured by the earlier objects. It has covolume $\disp(\theta, p, q)^{-1}$ and contains a primitive vector on the sphere $\Sphere^{m_1} \times \cdots \times \Sphere^{n_r}$—namely, the point $\proj(\theta, p, q)$. To isolate the relative size data while removing the influence of the direction and error term, we now define a \emph{derived lattice}.}

\medskip

{Let $\bfe_1, \ldots, \bfe_d$ denote the standard basis vectors of $\R^d$, and let $\E_d^j$ be the space of unimodular lattices in $\R^d$ that contain $\bfe_j$ as a primitive vector.}

{We define a map
\[
\lambda_j: \left(\M_{m \times n}(\R)  \times (\Z^m \times \Z^n)_\prim\right) \setminus \left( \disp^{-1}(0) \cup \{ (\theta, p, q) : (p + \theta q, q)_j = 0 \} \right) \to \E_d^j
\]
by setting
\[
\lambda_j(\theta, p, q) := A_j(\proj(\theta,p,q), \disp(\theta,p,q))\, \Lambda_\theta(p,q),
\]
where for $x \in \R^{m+n}$ satisfying $x_j \neq 0$ and $\gamma>0$, the matrix $A_j(x, \gamma) \in \GL_d(\R)$ is the unique matrix {with determinant $\gamma$} that maps $x$ to $\pm \bfe_j$—with sign matching the $j$-th coordinate of $x$—and acts as scalar multiplication on the orthogonal complement $\sspan\{\bfe_1, \ldots, \bfe_{j-1}, \bfe_{j+1}, \ldots, \bfe_d\}$.}

{By construction, the covolume of $\lambda_j(\theta, p, q)$ is
\[
\det(A_j(\proj(\theta,p,q), \disp(\theta,p,q))) \cdot \text{covolume}(\Lambda_\theta(p,q)) = \disp(\theta, p, q) \cdot \disp(\theta, p, q)^{-1} = 1,
\]
and it contains the vector
\[
A_j(\proj(\theta, p, q), \disp(\theta, p, q)) \proj(\theta, p, q) \in \{ \pm \bfe_j \}.
\]
Hence, $\lambda_j(\theta, p, q) $ is a unimodular lattice that contains $\bfe_j$ as a primitive vector.}

{
In summary, we now associate to each $\e$-approximate a fourth object: the \emph{derived lattice} $\lambda_j(\theta, p, q) \in \E_d^j$, for any fixed $1 \leq j \leq d$. }

\begin{rem}
Note that for $j>m$, the lattice $\lambda_j(\theta,p,q)$ still contains the information about the \emph{relative sizes} of the quantities
\[
\|\rho_1(p+\theta q)\|, \ldots, \|\rho_k(p+\theta q)\|, \quad 
\|\rho_1'(q)\|, \ldots, \|\rho_r'(q)\|.
\]

For simplicity, we illustrate this in the case $j=d$. In this case, there exists a constant $c=c(\theta,p,q)$ such that the covolume of the $m_i$-dimensional lattice
\[
\lambda_d(\theta,p,q) \cap 
\left( \{0\}^{m_1 + \cdots + m_{i-1}} \times \R^{m_i} \times 
\{0\}^{m_{i+1} + \cdots + m_k + n} \right)
\]
within its ambient subspace equals $c\,\|\rho_i(p+\theta q)\|^{-m_i}$.

Similarly, for each $1 \le j \le r-1$, the covolume of the $n_j$-dimensional lattice
\[
\varrho_2(\lambda_d(\theta,p,q)) \cap 
\left( \{0\}^{n_1 + \cdots + n_{j-1}} \times \R^{n_j} \times 
\{0\}^{n_{j+1} + \cdots + n_r} \right)
\]
equals $c\,\|\rho_j'(q)\|^{-n_j}$.

Finally, the quantity $c\,\|\rho_r'(q)\|^{-n_r}$ equals the covolume of the $n_r$-dimensional lattice
\[
\rho_r'(\varrho_2(\lambda_d(\theta,p,q))).
\]
\end{rem}
\begin{rem}
    Note that different values of $j$ become more natural for various values of $m,n$. In particular, for $n=1$, the choice $j=d$ is most natural and was considered in \cite{SW22}.
\end{rem}

\medskip

\subsection{Natural measures}
\label{subsec:Natural measures}

{Fix an integer $1 \leq j \leq d$. For notational convenience, we define the space
\begin{align}
    \label{ eq: def z j}
    \ZZ_j := \E_d^j \times \Sphere^{m_1} \times \cdots \times \Sphere^{m_k} \times \Sphere^{n_1} \times \cdots \times \Sphere^{n_r} \times [0,\e] \times \hZp^d.
\end{align}
We also define the map {
\begin{align}
    \label{ eq: def: Theta j}
    \Theta_j: \{(\theta,p,q) \in \Mat \times (\Z^m \times \Z^n)_\prim:  \disp(\theta,p,q) \in (0,\e),   (p+\theta q, q)_j \neq  0 \}  \longrightarrow \ZZ_j
\end{align}}
by
\begin{align}
    \label{eq: def Xi}
    \Theta_j(\theta, p, q) := \left( \lambda_j(\theta, p, q),\, \proj(\theta, p, q),\, \disp(\theta, p, q),\, (p, q) \right).
\end{align}}

{
The space $\ZZ_j$ is a product of several geometric spaces, each of which carries a natural measure. We now describe these component measures, and hence the product measure on $\ZZ_j$.}

\begin{enumerate}
    \item For all $1 \leq j \leq d$, note that $\E_d^j$ can be identified with a homogeneous space of the group $$H^j= \{h \in \SL_d(\R): h. \bfe_j =\bfe_j\}.$$ Thus, the most natural measure on {$\E_d^j$} is the unique $H^j$-invariant probability measure, denoted by $m_{\E_d^j}$. 
    \item For each $l \in \{m_1, \ldots, m_k, n_1, \ldots, n_r\}$, define the measure $\mu^{(\Sphere^l)}$ on $\Sphere^l$ as the pushforward of Lebesgue measure {$m_{\R^l}$ restricted to $\{x \in \R^l : \|x\| \leq 1\}$} under the projection map $x \mapsto x / \|x\|$. Note that $\mu^{(\Sphere^l)}$ is not necessarily a probability measure.
    \item {On the interval $[0,\e]$, we use the Lebesgue measure restricted to this interval, denoted $m_{\R}|_{[0,\e]}$.} 
    \item {The set $\hZp^d$ is an orbit under the natural action of $\SL_d(\hZ)$ on $\hZ^d$. Thus, the natural measure on $\hZp^d$ is the unique $\SL_d(\hZ)$-invariant probability measure, denoted by $m_{\hZp^d}$.}
\end{enumerate}

{Putting these together,} the natural measure on the full product space is
\begin{align}
    \label{eq:def tmu j}
    \tmu^j= m_{\E_d^j} \times   \mu^{(\Sphere^{m_1})} \times \cdots \times \mu^{(\Sphere^{m_k})} \times \mu^{(\Sphere^{n_1})} \times \cdots \times \mu^{(\Sphere^{n_r})} \times m_{\R}|_{[0,\e]} \times  m_{\hZp^d}.
\end{align}
We also define $\mu^j$ as the probability measure obtained by normalising $\tmu^j$, for $1 \leq j \leq d$.

\vspace{0.2in}

\subsection*{Main Theorems} 
The first main result of this paper is the following {equidistribution} theorem.

\begin{thm}
    \label{thm: cor to main thm}
    Fix an index $1 \leq j \leq d$. Then, for Lebesgue almost every $\theta \in \M_{m \times n}(\R)$, the following holds. Let $(p_l, q_l) \in \Z^m \times \Z^n$ denote the sequence of $\e$-approximates to $\theta$, ordered by increasing $\|q_l\|$. Then the sequence{
    \begin{equation*}
        \left( \Theta_j(\theta, p_l, q_l) \right)_{l \in \N}
    \end{equation*}
    becomes equidistributed in the space $\ZZ_j$ with respect to the probability measure $\mu^j$. In other words, for every bounded continuous function $f$ on $\ZZ_j$, we have:
    \begin{equation}
        \label{eq: main thm general 2}
        \lim_{N \to \infty} \frac{1}{N} \sum_{l=1}^N f\left( \Theta_j(\theta, p_l, q_l) \right) 
        = \mu^j(f).
    \end{equation}}
\end{thm}

{Theorem~\ref{thm: cor to main thm} is a consequence of our second main result, which provides an asymptotic count of $\e$-approximates satisfying the required geometric and arithmetic constraints.}

\begin{thm}
    \label{main thm}
    {Fix an index $1 \leq j \leq d$, and let $A \subset \ZZ_j$ be a measurable set whose boundary has zero measure with respect to $\tmu^j$, that is,
    \[
    \tmu^j(\partial A) = 0.
    \]
    Then, for Lebesgue almost every $\theta \in \M_{m \times n}(\R)$, the number of $\e$-approximates $(p, q) \in \Z^m \times \Z^n$ of $\theta$ satisfying
    \[
    \Theta_j(\theta, p, q) \in A \quad \text{and} \quad \|q\| \leq e^T
    \]
    is asymptotic to
    \[
    \frac{c_{k+r-1}(n_1, \ldots, n_r)}{(k+r-1)! \, \zeta(d)} \, \tmu^j(A) \, T^{k+r-1}
    \]
    as $T \to \infty$, where
    \begin{align}
        \label{eq: def c k r 1}
        c_{k+r-1}(n_1, \ldots, n_r) = \sum_{(x_1, \ldots, x_r) \in \{0,1\}^r} (-1)^{r - (x_1 + \cdots + x_r)} (n_1 x_1 + \cdots + n_r x_r)^{k+r-1}.
    \end{align}}
\end{thm}

{\begin{rem}
\label{rem: importance  of A}
    We emphasise that by choosing the set \( A \) appropriately in Theorem \ref{main thm}, one can impose various types of constraints on the approximates. These include:
    \begin{itemize}
        \item \emph{Congruence conditions} via the \( \hZ^d \) component,
        \item \emph{Directional constraints} via the \( \Sphere^{m_1} \times \cdots \times \Sphere^{n_r} \) component,
        \item \emph{Quality of approximation} via the \( \R \) component, and
        \item {\emph{Relative size constraints} via the \( \E_d^j \) component.}
    \end{itemize}
\end{rem}
}
\medskip

\begin{rem}
For a fixed $\theta$, some quantities such as $\lambda_j(\theta, p, q)$ or $\proj(\theta, p, q)$ may be undefined for certain $\e$-approximates $(p, q)$. Theorems~\ref{thm: cor to main thm} and~\ref{main thm} should be interpreted as asserting that the number of such exceptional approximates is $o(N)$ and $o(T^{k+r-1})$ respectively, and hence negligible.
\end{rem}

\begin{rem}
{Using Theorem~\ref{muJM}, one sees that Theorem~\ref{main thm} is equivalent to the following statement: for Lebesgue almost every $\theta \in \M_{m \times n}(\R)$, if $(p_l, q_l) \in \Z^m \times \Z^n$ is the sequence of $\e$-approximates ordered by increasing $\|q_l\|$, then
\begin{equation}
\label{eq: main thm 2}
\lim_{T \to \infty} \frac{1}{T^{k+r-1}} \sum_{\|q_l\| \leq e^T} \delta_{\Theta_j(\theta,p_l,q_l)} 
= \frac{c_{k+r-1}(n_1, \ldots, n_r)}{(k+r-1)! \, \zeta(d)} \, \tmu^j,
\end{equation}
where the convergence is in the \emph{tight topology} (see Appendix~\ref{Tight Convergence}). In other words, for every bounded continuous function $f$ on $\ZZ_j$, we have
\begin{equation}
    \label{eq: main thm 1}
    \lim_{T \to \infty} \frac{1}{T^{k+r-1}} \sum_{\|q_l\| \leq e^T} f\left( \Theta_j(\theta,p_l,q_l) \right) = \frac{c_{k+r-1}(n_1, \ldots, n_r)}{(k+r-1)! \, \zeta(d)} \, \tmu^j(f). 
\end{equation}}

{We emphasise that neither the sequence on the left-hand side nor the measure on the right-hand side of \eqref{eq: main thm 2} are probability measures. Therefore, convergence in the tight topology is strictly stronger than convergence in the weak topology. In the weak topology, convergence of measures only guarantees convergence of integrals against compactly supported continuous functions. In particular, it does \emph{not} imply convergence for constant test functions, and therefore cannot rule out loss of mass. For example, a sequence of probability measures $\nu_\ell$ on a non-compact space may converge weakly to the zero measure.

In contrast, convergence in the tight topology implies that the limiting measure retains the total mass: if $\nu_\ell$ is a sequence of probability measures converging tightly, then the limit $\nu$ is also a probability measure. Thus, Equation~\eqref{eq: main thm 2} implies not only equidistribution in the weak sense, but also the absence of escape of mass.}
\end{rem}

{\begin{rem}
    Theorem~\ref{thm: cor to main thm} follows as a consequence of Theorem~\ref{main thm}. To see this, fix $\theta \in \Mat$ satisfying equation~\eqref{eq: main thm 1}. Let $(p_l, q_l) \in \Z^m \times \Z^n$ denote the sequence of $\e$-approximates to $\theta$, ordered by increasing $\|q_l\|$. Fix a bounded continuous function $f$ on $\ZZ_j$, and define
    \[
    T_N = \log \|q_N\|.
    \]
    Then we have
    \begin{align}
    \label{eq: xax 1}
        \lim_{N \to \infty} \frac{1}{N} \sum_{l=1}^N f\left( \Theta_j(\theta, p_l, q_l) \right)
        = \lim_{N \to \infty} \frac{ \frac{1}{T_N^{k+r-1}} \sum_{\|q_l\| \leq e^{T_N}} f\left( \Theta_j(\theta, p_l, q_l) \right)}{ \frac{1}{T_N^{k+r-1}} \sum_{\|q_l\| \leq e^{T_N}} \mathbf{1}\left( \Theta_j(\theta, p_l, q_l) \right)},
    \end{align}
    where $\mathbf{1}$ denotes the constant function on $\ZZ_j$ with value $1$. Since $T_N \to \infty$ as $N \to \infty$, we may apply equation~\eqref{eq: main thm 1} to both the numerator and denominator in \eqref{eq: xax 1}. This yields
    \begin{align*}
        \lim_{N \to \infty} \frac{1}{N} \sum_{l=1}^N f\left( \Theta_j(\theta, p_l, q_l) \right)
        =  \frac{ \frac{c_{k+r-1}(n_1, \ldots, n_r)}{(k+r-1)! \, \zeta(d)} \, \tmu^j(f)}{\frac{c_{k+r-1}(n_1, \ldots, n_r)}{(k+r-1)! \, \zeta(d)} \, \tmu^j(\mathbf{1})} = \mu^j(f).
    \end{align*}
    The conclusion of Theorem~\ref{thm: cor to main thm}, namely equation~\eqref{eq: main thm general 2}, thus follows.
\end{rem}}

\begin{rem}
\label{rem:reduction of main thm}
    It turns out that it is enough to prove Theorem \ref{main thm} for $j=d$. Indeed, consider the map
\begin{align}
\label{eq: def phi i d}
\phi_{jd} :\ZZ_d \to \ZZ_j
\end{align}
defined by
\[
\phi_{jd}(\Lambda, x, \gamma, v) := {\left(A_j(x, \gamma) A_d(x, \gamma)^{-1} \Lambda, x, \gamma, v\right)}.
\]
This map pushes forward the measure $\tmu^d$ to $\tmu^j$, and satisfies
\[
\phi_{jd}\left(\Theta_d(\theta,p,q)\right)  = \Theta_j(\theta,p,q).
\]
{Hence, the general case reduces to the case $j = d$. For notational simplicity, we will simply denote $\E^d_d$ by $\E_d$, the measure $m_{\E^d_d}$ by $m_{\E_d}$ throughout this paper.}
\end{rem}
\vspace{0.2in}

{The third main theorem of the paper focuses on the approximates in the special case $k = r = 1$.}

\begin{thm}
\label{thm: cor to Main thm time visits}
    Fix $1 \leq j \leq d$. {Let $A \subset \ZZ_j$ be a subset satisfying
    \begin{align}
        \label{eq: con cor time 1}
        &\tmu^j(\partial A) = 0, \quad  \quad \tmu^j(A) > 0, \\
        \label{eq: con cor time 2}
        A &\subset \left\{ (\Lambda, v, \gamma, w) \in \E_d^j \times (\Sphere^m \times \Sphere^n) \times [0,\e] \times \hZ^d : v_j \geq 0 \right\}.
    \end{align}
    Fix an integer $s \in \Z_{\geq 0}$.} Then there exists a probability measure $\nu^{A,s}$ on $\R^m$ such that the following holds for Lebesgue almost every $\theta \in \M_{m \times n}(\R)$. Let $(p_l, q_l) \in \Z^m \times \Z^n$ be the sequence of $\e$-approximates of $\theta$, ordered by decreasing $\|p_l + \theta q_l\|$, and satisfying
    \begin{equation}
        \label{eq: cor to Main thm time visits}
        \Theta_j(\theta, p_l, q_l) \in A.
    \end{equation}
    Then the sequence
    \[
    \left( \|q_{l+s}\|^{n/m}(p_l + \theta q_l) \right)_l
    \]
    equidistributes {with respect to the measure $\nu^{A,s}$}. {That is, for every bounded continuous function $f$ on $\R^m$, we have
    \begin{equation*}
        \lim_{N \to \infty} \frac{1}{N} \sum_{l=1}^N f\left( \|q_{l+s}\|^{n/m}(p_l + \theta q_l) \right)
        = \nu^{A,s}(f).
    \end{equation*}}
\end{thm}

\begin{rem}
    Note that the condition \eqref{eq: con cor time 2} is needed {to select only once between an $\e$-approximate $(p, q)$ and its negative $(-p, -q)$.} Otherwise, if we consider both the approximates $(p,q)$ and $(-p,-q)$, then ordering according to decreasing $\|p+ \theta q\|$ is not unique.
\end{rem}

\begin{rem}
    Using the techniques developed in the paper, a similar version of Theorem \ref{thm: cor to Main thm time visits} is possible, where $(p_l,q_l)$ are ordered according to increasing $\|q_l\|$. In this case, the distribution of $\|q_{l+i}\|/\|q_{l+i-1}\|$ can also be studied. We leave the details to the interested readers.
\end{rem}

\vspace{0.2in}

\subsection{Key Ingredients}\label{sec:keying}
The broad strategy used in our paper is to utilise a cross-section in the space of lattices. This is the strategy used by Shapira and Weiss \cite{SW22}, by Cheung and Chevallier \cite{CC19}, and goes back to Arnoux and Nogueira \cite{AN93}. The general theory of cross-sections was developed by Kakutani and Ambrose, and is nicely summarised in \cite{SW22}, and cross sections have found great use in homogeneous dynamics, cf. \cites{AthreyaCheung, Marklof2010, MarStro}. However, several new ideas are required to deal with even the simplest multiplicative toy example in the introduction. This section briefly discusses the new ideas developed in this paper. Approximately, two essential ingredients are used in the paper.

The first is the construction of a cross-section for a multiparameter flow in the adelic space $\SL_d(\A)/\SL_d(\Q)$ (defined in Section \ref{sec:Adelizing the Setup}) and the association of a natural measure with it. This posed a significant conceptual challenge, as the existing theories only address one-parameter flows. To the best of our knowledge, an analogous theory for multi-parameter flows does not exist in the literature, making this the first study of a cross-section for such flows. The analogous theory faces difficulties at several points due to the absence of any natural {\em return time function} and {\em first return map}, which play an important role for cross-sections in the one-parameter setting. We hope that the theory of cross-sections for multi-parameter flows, developed in this paper through an example, will find many more interesting applications. In particular, we refer the reader to \cite{AG24Levy}, where this theory is used to prove L\'{e}vy-Khintchine Theorem in a generalised setup. 

The construction also posed a significant technical challenge, as the cross-section should ideally allow for a one-to-one correspondence between the time visits of the lattice $\Tilde{\Lambda}_\theta$ (a definition is provided in due course) to the cross-section up to time $T$ and the Diophantine approximates $(p,q)$ such that {$\|q\| \leq e^T$}. For the case $k=r=1$, corresponding to the usual additive Diophantine approximation, this construction is possible because the domain, obtained from Dani's correspondence, can be nicely tessellated. However, an equivalent cross-section for the multiplicative case proved elusive. As a result, we opted for a cross-section that satisfies a weaker property: while every time visit of the lattice $\Tilde{\Lambda}_\theta$ to the cross-section up to time $T$ corresponds to a Diophantine approximate $(p,q)$ with {$\|q\| \leq e^T$}, the number of Diophantine approximates that do not correspond to time visits is negligible. This cross-section has a particular advantage even when $k=r=1$ and the flow is one-parameter, as it allows for the treatment of Diophantine approximation of matrices, whereas in \cite{SW22}, the authors treated the case $n=1$.

The second technical issue is to prove the Birkhoff genericity of
\begin{align}
    \label{eq: def tilde theta}
    \Tilde{\Lambda}_\theta= \left(\begin{pmatrix}
    I_m & \theta \\ & I_n
\end{pmatrix}, e_f \right)\SL_d(\Q) \in \XA_d.
\end{align}
for Lebesgue almost every $\theta$, under the multi-parameter diagonal flow $a_{t}$ (precise definitions to appear later). To deal with this, we split the problem into two parts. The first part involves proving the Birkhoff genericity of $\Lambda_\theta$ in $\X_d$ (the space of all unimodular lattices). The proof requires strong results of \cite{KM1}, \cite{BG21} regarding effective equidistribution of translates of topological tori under multiparameter flows. Specifically, we prove a new result on the effective equidistribution of the orbit, which generalises several results of \cite{KSW} to multiparameter flows. The second part involves proving the equivalence of the Birkhoff genericity of $\Lambda_\theta$ in $\X_d$ to the Birkhoff genericity in $\XA_d$. This answers a question asked in \cite[Remark~13.1]{SW22}.

\subsection{Structure of the paper} 
Section~\ref{sec:corollaries} derives several corollaries of the main theorem and provides a brief historical overview of Diophantine approximation with constraints. Section~\ref{sec: Notation} introduces the basic notation used throughout the paper. Section~\ref{sec:Birkhoff Genericity} establishes the Birkhoff genericity of $\Tilde{\Lambda}_\theta$ for Lebesgue-almost-every $\theta$ under a multi-parameter flow, and answers the question posed by Shapira and Weiss~\cite[Remark~13.1]{SW22}. 

Section~\ref{sec: The Cross-section} introduces the cross-section in the real case and associates a measure to it. Section~\ref{sec: Explicit formula of the cross-sectional measure} provides an explicit description of this measure, and Section~\ref{sec:Properties of sed} develops further properties of the cross-section. Section~\ref{sec:Adelizing the Setup} lifts both the cross-section and the cross-sectional measure to the adele space, extending all the necessary properties to this setting. Section~\ref{sec:Cross-section Correspondence} demonstrates the significance of the cross-section by relating the time visits of $\Tilde{\Lambda}_\theta$ under the multi-parameter flow to $\varepsilon$-approximates of $\theta$, establishing a sharp lower bound for their asymptotic count. The complementary upper bound is obtained in Section~\ref{sec: Upper Bound for Counting} using classical results on almost-sure asymptotic counts of solutions to Diophantine inequalities~\cites{WYYK, S60, Gallagher}. These bounds are combined in Section~\ref{sec: Proof of main Theorem} to complete the proof of Theorem~\ref{main thm}.

Sections~\ref{sec: Time visits} and~\ref{section: Time visit proof} focus on the special case $k=r=1$, where the flow becomes one-parameter and the theory of cross-sections is comparatively simpler and more extensively studied in the literature. Section~\ref{sec: Time visits} analyses the return time function and first return map for the cross-section, and Section~\ref{section: Time visit proof} uses these results to prove Theorem~\ref{thm: cor to Main thm time visits}.

Appendix~\ref{Tight Convergence} recalls basic facts about Jordan measurable sets and tight convergence; Appendix~\ref{sec: Convergence in Compact extension} establishes an equivalence---under certain assumptions---between the equidistribution of an orbit in a factor space and in its compact extension, a key ingredient in Section~\ref{sec:Birkhoff Genericity} for proving the equivalence of Birkhoff genericity in the real and adelic settings; Appendix~\ref{Appendix: Measure Theory} studies the effective equidistribution of the orbit under the multi-parameter diagonal flow, extending the techniques of~\cite[Section~3]{KSW} to the multi-parameter case; and Appendix~\ref{sec: measure of JT} computes the measure of the set $J^T$.

\noindent \textbf{Acknowledgements.}
We thank Uri Shapira, and Barak Weiss for many helpful discussions of their paper \cite{SW22}, which served as the starting point of our investigations. We thank Jens Marklof for helpful discussions. We are especially grateful to the anonymous referee for a careful and patient reading and for numerous valuable suggestions, which have greatly improved the paper.


\section{Diophantine Corollaries and Prior Work}
\label{sec:corollaries}
There has been extensive prior work on Diophantine approximation with constraints. We refer the reader to Section 3 in \cite{SW22}, which contains a comprehensive list of references. In the same section, Shapira and Weiss explain how their work generalises many such results. Our main theorems also generalise these existing results; the main novelty is that we are able to treat \emph{multiplicative} approximation for the first time. This distinguishes our paper from the work of Shapira and Weiss \cite{SW22} and earlier results. In this section, we illustrate the strength of our theorems by discussing some applications.

\subsection{Equidistribution of displacement vector}

Let $\theta \in \M_{m \times n }(\R)$, and let $v_l=(p_l, q_l) \in \Z^m \times \Z^n$ denote the sequence of $\e$- approximates to $\theta$, defined as in Section \ref{Introduction}, {ordered according to increasing $\|q_l\|$}. We want to study the distribution of 
\begin{align}
    \label{eq: cor 1}
    \left(\frac{\rho_1(p_l+ \theta q_l)}{\|\rho_1(p_l+ \theta q_l)\|}, \ldots, \frac{\rho_k(p_l+ \theta q_l)}{\|\rho_k(p_l+ \theta q_l)\|}, \frac{\rho_1'(q_l)}{\|\rho_1(q_l)\|}, \ldots, \frac{\rho_r'(q_l)}{\|\rho_r'(q_l)\|}, \disp(\theta, (p_l, q_l)) \right)
\end{align}
in the space $\Sphere^{m_1} \times \cdots \times \Sphere^{m_k} \times \Sphere^{n_1} \times \cdots \times \Sphere^{n_r}\times \R$. When $k=1$ and $n=1$, this is equivalent to considering the distribution of $ q_l^{1/m}(p_l+\theta q_l)$. From Theorems \ref{main thm}, we immediately get the following statements, valid for any norm.

\begin{cor}
Let $\e>0$ be arbitrary and let $\mu'$ be the normalization of the product measure $\mu^{(\Sphere^{m_1})} \times \cdots \times \mu^{(\Sphere^{m_k})} \times \mu^{(\Sphere^{n_1})} \times \cdots \times \mu^{(\Sphere^{n_r})} \times m_{\R}|_{[0,\e]}$.

    Then for a.e. $\theta \in \M_{m \times n}(\R)$, the sequence \eqref{eq: cor 1} of errors and direction of $\e$-approximations of $\theta$ equidistributes with respect to $\mu'$. 
\end{cor}

\medskip

\subsection{Equidistribution in $\hZ^d$ and congruence constraints}

For a positive integer $t$ and a vector {$a = (a_1, \ldots, a_d) \in (\Z / t\Z)^d$}, we say that {$a$ is \emph{primitive mod $t$} if there does not exist a non-unit $b \in \Z / t\Z$} and a vector {$c = (c_1, \ldots, c_d) \in (\Z / t\Z)^d$} such that $a_i = b c_i$ for all $i$. Let $N_{t,d}$ denote the number of primitive vectors in $(\Z / t\Z)^d$. It is not hard to verify that
\[
N_{t,d} = \prod_{i=1}^j \left( p_i^{d s_i} - p_i^{d(s_i - 1)} \right),
\]
where $t = p_1^{s_1} \cdots p_j^{s_j}$ is the prime factorization of {$t$}.

The following corollary follows immediately from Theorem \ref{main thm}.

\begin{cor}
    For Lebesgue almost every $\theta \in \M_{m \times n}(\R)$, the sequence {$(v_l)_{l \in \N}$ of $\e$-approximations satisfies
    $$
    \frac{1}{N} |\{1 \leq l \leq N: v_l \equiv a \mod t \}| \longrightarrow_{N \rightarrow \infty} c,
    $$}
     where
    $$
    c = \begin{cases}
        \frac{1}{N_{t,d}} \text{ if $a$ is primitive mod $t$} \\
        0 \text{ otherwise.}
    \end{cases}
    $$
\end{cor}

\medskip

\subsection{Comparison with the work of Shapira and Weiss}
We define $\e^+$-approximations of a vector $\theta \in \R^m$ as integers $(p,q) \in \Z^m \times \N$ satisfying $\gcd(p,q)=1$ and $q^{1/d}\|p+ q\theta\| < \e$. In our earlier notation with choice of $k=r=n=1$, these are exactly $\e^d$-approximation $(p,q)$ of $\theta$ satisfying $q>0$.

One of the main theorems in \cite{SW22}, namely {Theorem} 1.2, discusses the $\e^+$-approximation of Lebesgue a.e. $\theta \in \R^m$. Their theorem followed from a much more general statement in the same paper, namely \cite[Thm.~2.1]{SW22}, which for $\e^+$-approximation of Lebesgue a.e. $\theta \in \R^m$ says that
\begin{thm}[{\cite[Thm.~2.1]{SW22}}]
    For Lebesgue almost every $\theta \in \M_{m \times 1}(\R)$, the following holds. Let $(p_l, q_l) \in \Z^m \times \N$ be the sequence of $\e^+$-approximations of $\theta$, ordered according to increasing $q_l$. Then the sequence 
    {\begin{equation}
    \label{eq: main thm SW22}
        (\lambda_d(\theta, p_l, q_l), q_l^{1/d} (p_l + q_l \theta) , (p_l, q_l))_{l \in \N} \in \E_d^d \times \R^{m} \times \hZ^d
    \end{equation}}
    equidistributes with respect to the probability measure obtained by normalising the finite measure $m_{\E_{d}^d} \times m_{\R^m}|_{\{x: \|x\| \leq \e\}} \times m_{\hZp^d}$.
\end{thm}
Now, it is easy to see that Theorem \ref{main thm} (in the special case of $k=r=n=1$) generalises the above theorem. To see this, note that the study of the distribution of $q_l^{1/d} (p_l + q_l \theta) \in \R^m$ is the same as studying the distribution of 
{$$
\left( \frac{p_l+ q_l \theta}{\|p_l+ q_l \theta\|}, \frac{q_l}{|q_l|}, q_l \|p_l+ q_l \theta\|^d \right) = (\proj(\theta,p,q), \disp(\theta,p,q)) \in \Sphere^m  \times \Sphere^1 \times \R.
$$}
Thus, Theorem \ref{main thm}, when viewed as a counting result with constraints, clearly generalises the above theorem. 

\begin{rem}
    The conclusion of \cite[Thm.~2.1]{SW22} for best approximates of Lebesgue a.e. $\theta$ is similarly generalised in our companion paper \cite{AG24Levy}; see also \cites{AG25DL, AG25ELK}.
\end{rem}


\section{Notations}
\label{sec: Notation}

{ Throughout the paper, we will work in both real and adelic settings, and several notions will be used in both contexts. We begin by recording a basic measure-theoretic concept that will be invoked repeatedly, namely that of \emph{Jordan measurability} with respect to a given measure. After this, we introduce the notation for the real and adelic homogeneous spaces appearing in our arguments. For the reader's convenience, a table of notation is provided at the end of the paper.}
\begin{defn}
\label{defJM}
    Let $X$ be a locally compact second countable space and $\nu$ be a finite regular Borel measure on $X$. We say that $E \in \mathcal{B}_X$ is \textit{Jordan measurable with respect to $\nu$} (abbreviated $\nu-$JM) if $\nu(\partial_X(E))=0$.
\end{defn}

The following simple property will be useful.

\begin{lem}
\label{JMintersect}
    If $E,F \subset X$ are $\nu$-JM, then $E \cap F$ is also a $\nu$-JM subset of $X$.
\end{lem}

\subsection{Real Setup}
\label{sec: Notation real}
{In this brief section, we introduce the real machinery that will be required throughout the paper.} Let us define $G = \SL_d(\R)$, $\Gamma = \SL_d(\Z)$, and $\X_d = G / \Gamma$. The space $\X_d$ is isomorphic to the space of unimodular lattices in $\R^d$, where the isomorphism is given by
\[
g \Gamma \mapsto g \Z^d.
\]
Since $\Gamma$ is a lattice in $G$, we may equip $\X_d$ with the unique $G$-invariant Siegel-Haar probability measure $\mu_{\X_d}$.

\medskip
{For \( t = (t_1, \ldots, t_{k+r-1}) \in \R^{k+r-1} \), we define the diagonal matrix \( a_t \in G \) by
\begin{equation}
    \label{defatt}
    a_t = \begin{pmatrix}
        e^{t_1} I_{m_1} & & & & & & \\
        & \ddots & & & & & \\
        & & e^{t_k} I_{m_k} & & & & \\
        & & & e^{-t_{k+1}} I_{n_1} & & & \\
        & & & & \ddots & & \\
        & & & & &  e^{-t_{k+r}} I_{n_r}
    \end{pmatrix},
\end{equation}
where
\[
    t_{k+r} = \frac{m_1 t_1 + \cdots + m_k t_k - n_1 t_{k+1} - \cdots - n_{r-1} t_{k+r-1}}{n_r}.
\]}\\

We will borrow some notation from Section 8 of \cite{SW22}. Given a unimodular lattice $\Lambda$, we define the set of all primitive vectors in $\Lambda$ by $\Lambda_\prim$. Given a subset $W \subset \R^d$ and $l \geq 1$, we define 
\begin{equation}
    \label{defX(W,l)}
    \X_d(W,l) := \{\Lambda: \# (\Lambda_{\prim} \cap W) \geq l\}.
\end{equation}
For $l = 1$, we omit $l$ and denote 
\begin{equation}
    \label{defX(W)}
    \X_d(W):= \X_d(W,1).
\end{equation}
We also define 
\begin{equation}
    \label{defXsharp}
    \X_d^\sharp(W):= \X_d(W) \setminus \X_d(W,2).
\end{equation}
{There is a natural map 
\begin{equation}
    \label{defv-unique}
    \vector: \X_d^\sharp(W) \rightarrow W \text{,  defined by  } \{\vector(\Lambda)\} =\Lambda_\prim \cap W.
\end{equation}}
With this notation, we are now ready to state the following results from \cite{SW22}.
\begin{lem}[{\cite[Lem.~8.1]{SW22}}]
\label{SW22 Lem 8.1}
    Let $W \subset  \R^d$ be a compact set, $V \subset W$ a relatively open subset and $l \geq 1$ an integer. 
    \begin{itemize}
        \item The set $\X_d(W,l)$ is closed in $\X_d$.
        \item The set $\X_d^{\sharp}(W) \cap \X_d(V)$ is open in $\X_d(W)$.
        \item The map $\vector : \X_d^\sharp(W) \rightarrow W$ is continuous.
    \end{itemize}
\end{lem}
\begin{lem}[{\cite[Lem.~8.2]{SW22}}]
\label{Sw22 Lem 8.2}
    Let $W \subset \R^d$ be an open set. For any $l \geq 1$, $\X_d(W,l)$ is open in $\X_d$.
\end{lem}
\begin{lem}[{\cite[Lem.~8.8]{SW22}}]
\label{SW22Lem8.8}
    Let $L \subset \SL_d(\R)$ be a closed subgroup, with left Haar measure $m_L$. Let $\Lambda_0 \in \X_d$ be such that $L\Lambda_0$ is a closed orbit supporting a finite $L$-invariant measure {$\mu_{L\Lambda_0}$}. Let $W \subset \R^d$ such that for any $v \in W$, there is a $\delta>0$ such that $$m_{L}(\{g \in B_\delta(L): g.v \in W\})=0,$$ (where $B_\delta(L)$ denotes the $\delta$-ball around the identity element of $L$, with respect to some metric inducing the topology). Then $\mu_{L\Lambda_0}(\X_d(W))=0.$
\end{lem}

\subsection{Adelic Setup}
\label{ sec: Adelic Setup}
In this brief section, we introduce an adelic machinery that lies at the heart of incorporating congruence conditions into Diophantine analysis. Our presentation here is heavily influenced by Shapira and Weiss \cite{SW22}, and in particular, we follow their notation. 

Let $\Primes$ {denote} the set of rational primes. Let $\A= \R \times \A_f= \R \times \prod_{p \in \Primes}'\Q_p$ {denote} the ring of adeles. Here $\prod'$ stands for the restricted product, i.e, a sequence $\underbar{$\beta$}= (\beta_\infty, \beta_f)= (\beta_\infty,\beta_2, \beta_3 \ldots, \beta_p, \ldots )$ belongs to $\A$ if and only if $\beta_p \in \Z_p$ for all but finitely many $p$. As suggested by the notation, we denote the real coordinate of a sequence $\underbar{$\beta$} \in \A$ by $ \beta_\infty$ and the sequence of $p$-adic coordinates by $\beta_f= (\beta_p)_{p \in \Primes}.$ The rational numbers $\Q$ are embedded in $\A$ diagonally, that is, $q \in \Q$ is identified with the constant sequence $(q,q, \ldots)$. We let $\SL_d(\A)= \SL_d(\R) \times \SL_d(\A_f)= \SL_d(\R) \times \prod_{p \in \Primes}'\SL_d(\Q_p)$ and use similar notation $(g_\infty, g_f)= (g_\infty, (g_p)_{p \in \Primes })$ to denote elements of $\SL_d(\A).$ By a theorem of Borel, the diagonal embedding of $\SL_d(\Q)$ in $\SL_d(\A)$ is a lattice in $\SL_d(\A)$. Let $$K_f:= \prod_{p \in \Primes } \SL_d(\Z_p)$$ and $$\pi_f: \SL_d(\A) \rightarrow \SL_d(\A_f), \ \pi_f(g_\infty, g_f):=g_f.$$ Then $K_f$ is a compact open subgroup of $\SL_d(\A_f).$ Via the embedding {$\SL_d(\A_f) \simeq \{e\} \times \SL_d(\A_f)$} we also think of $K_f$ as a subgroup of $\SL_d(\A)$. We shall use the following two basic facts (see \cite[Chap.~7]{PR94}):
\begin{itemize}
    \item[(i)] The intersection $K_f \cap \pi_f(\SL_d(\Q)) $ is equal to $\pi_f(\SL_d(\Z))$.
    \item[(ii)] The projection $\pi_f(\SL_d(\Q))$ is dense in $\SL_d(\A_f)$.
\end{itemize}
Let $$\XA_d= \SL_d(\A)/\SL_d(\Q), $$ and let $\mu_{\XA_d}$ denote the $\SL_d(\A)$-invariant probability measure on $\XA_d$. There is a natural projection $\pi: \XA_d \rightarrow  \X_d$, which can be described as follows: Given $\tilde x = (g_\infty, g_f)\SL_d(\Q) \in \XA_d$, using $(ii)$ and the fact that $K_f$ is open, we may replace the representative $(g_\infty, g_f)$ by another representative $(g_\infty \gamma, g_f\gamma)$, where $\gamma \in \SL_d(\Q)$ is such that $g_f\gamma \in K_f$. We then define $\pi(\tilde x)= g_\infty \gamma \SL_d(\Z)$. This is well defined since, if $g_f\gamma_1, g_f \gamma_2 \in K_f$, then by $(i)$, $\pi_f(\gamma_1^{-1}\gamma_2) \in K_f \cap \pi_f(\SL_d(\Q))= \pi_f(\SL_d(\Z))$, and so $g_\infty\gamma_1 \SL_d(\Z)= g_\infty\gamma_2 \SL_d(\Z)$. 

It is clear that $\pi$ intertwines the actions of $G_\infty:= \SL_d(\R)$ on $\XA_d$ and $\X_d.$ Since there is a unique $G_\infty$-invariant probability measure on $\X_d$, we have $\pi_*(\mu_{\XA_d})= \mu_{\X_d}$. In particular, the  group $\{a_{t}\} \subset G_\infty$ defined as in \eqref{defatt}
acts on both of these spaces, and $\pi$ is a factor map for these actions.

\section{Birkhoff Genericity under \texorpdfstring{$a_t$}{at}}
\label{sec:Birkhoff Genericity}
{The main aim of this section is to study the Birkhoff genericity of $\tilde \Lambda_\theta \in \XA_d$ under $a_t$, where $\tilde \Lambda_\theta$ is defined as 
\begin{align}
\label{eq: def tilde Lambda}
    \tilde \Lambda_\theta = \left( \begin{pmatrix}
    I_m & \theta \\ & I_n
\end{pmatrix}, e_f \right) \SL_d(\Q) \in \XA_d,
\end{align}
and $\theta$ varies in $\Mat$. To formalise this, we begin with the following definition.}

\begin{defn}
\label{def: generic flow}
     We will say that $x \in \X_d$ (resp. $\XA_d$) is $(a_t, \mu_{\X_d})$-generic  (resp. $(a_t, \mu_{\XA_d})$-generic) if
    $$
    \lim_{T \rightarrow \infty} \frac{1}{m_{\R^{k+r-1}}(J^T)} \int_{J^T} \delta_{a_{t}x}\, dt= \mu_{\X_d} \text{ (resp. } \mu_{\XA_d}),
    $$
    with respect to the tight topology,  where {$J^T$} is defined as 
    \begin{align}
    \label{eq: def J^T}
        J^T= \{(t_1, \ldots, t_{k}, s_1, \ldots, s_{r-1}) \in (\R_{\geq 0})^{k+r-1} : \substack{ s_i \leq T \text{ for all } i \\ 0 \leq m_1t_1 + \cdots+ m_kt_k- n_1s_1 -\cdots- n_{r-1}s_{r-1} \leq n_rT} \}.
    \end{align}
    Note that
    \begin{align}
        \label{eq: measure of J T}
        m_{\R^{k+r-1}}(J^T)= \frac{T^{k+r-1}c_{k+r-1}(n_1, \ldots, n_r)}{m_1 \cdots n_{r-1}. (k+r-1)!},
    \end{align}
    where $c_{k+r-1}(n_1, \ldots, n_r)$ is defined as in \eqref{eq: def c k r 1}; {see Appendix~\ref{sec: measure of JT}.}
\end{defn}
\vspace{0.2in}
{The following proposition constitutes the main result of this section.}
\begin{prop}
\label{cor1 Generic}
     For Lebesgue almost every $\theta \in \M_{m \times n}(\R)$, the coset $\Tilde{\Lambda}_\theta$ is $(a_{t}, \mu_{\XA_d})$-generic.
\end{prop}

{We start by first proving an effective equidistribution of the orbit of $\Lambda_\theta$ (defined as in \eqref{eq:def lambda theta}) in ${\X_d}$.}
\begin{thm}
\label{thm: Effective Equidistribution}
    Let {$f \in \R . \mathbf{1} + C_c^\infty(\X_d)$} and $\e >0$ be given. Then for Lebesgue almost every $\theta \in \M_{m \times n}(\R)$, we have 
     {\begin{align*}
        \frac{1}{m_{\R^{k+r-1}}(J^T)}  \int_{t \in J^T} f(a_t \Lambda_\theta) dt = \int_{\X_d} f \, d\mu_{\X_d} +  \begin{cases}
            {o}(T^{-1} (\log T)^{\frac{3}{2}+r+\e}), \text{ if } k+r > 2\\
             {o}(T^{-1/2} (\log T)^{\frac{3}{2}+\e}), \text{ if } k+r = 2.
             \end{cases}
     \end{align*}}
\end{thm}

\begin{rem}
   Theorem \ref{thm: Effective Equidistribution} for $k+r=2$, i.e, $r=k=1$ was proved in \cite[Thm.~3.1]{KSW}.
\end{rem}{\begin{rem}
The notation $\R \cdot \mathbf{1} + C_c^\infty(\X_d)$ denotes the set of all functions on $\X_d$ that can be written as the sum of a constant and a compactly supported smooth function. That is, each element is of the form
\[
f(x) = c + \phi(x),
\]
where $c \in \R$ and $\phi \in C_c^\infty(\X_d)$.
\end{rem}}

\vspace{0.2in}

{
The key ingredient in the proof of Theorem \ref{thm: Effective Equidistribution} is Theorem \ref{Pointwise Main Lemma} along with the following two theorems.
\begin{thm}
\label{KM1 Thm} 
    There exists $\delta_0 > 0$ such that for every $f \in \R.\one + \Cc(\X_d)$, we have
    \begin{align}
        \int_{\M_{m \times n}((0,1))} f(a_t \Lambda_\theta) \, d\theta = \int_{\X_d} f(x) d\mu_{\X_d}(x) + O_f(e^{-\delta_0 \lfloor \tilde t \rfloor})
    \end{align}
    for all $t \in (\R_+)^{k+r-1}$, where $\tilde t= (t, \frac{m_1t_1 + \cdots + m_kt_k- n_1 t_{k+1} - \cdots - n_{r-1}t_{k+r-1}}{n_r})$ and for any $x \in \R^{k+r}$ we define $\lfloor x \rfloor= \min\{|x_1|, \ldots, |x_{k+r}|\}$. The implicit constants are independent of $t$.
\end{thm}
\begin{proof}
By \cite[Thm.~1.3]{KM1}, for every $f\in\mathbb{R}\cdot\mathbf{1}+\Cc(\X_d)$ and
$\phi\in\Cc(\M_{m\times n}(\mathbb{R}))$ we have
\begin{align}
\label{eq: azx 1}
\int_{\M_{m\times n}(\mathbb{R})}
\phi(\theta)\,f(a_t\Lambda_\theta)\,d\theta
=
\int_{\X_d} f(x)\,d\mu_{\X_d}(x)
+
O_f\!\left(e^{-\delta_0\lfloor\tilde t\rfloor}\right)
\end{align}
for all $t\in(\mathbb{R}_+)^{k+r-1}$.

Let $\phi_0\in C_c^\infty(\Mat)$ be a smooth function with
$\int_{\Mat}\phi_0=1$, and define
\[
\phi=\phi_0 * \ind_{\M_{m\times n}((0,1))}.
\]
Then $\phi\in C_c(\M_{m\times n}(\mathbb{R}))$ and satisfies
$$
\sum_{\xi \in \M_{m \times n}(\Z)} \phi(\theta + \xi) = 1 \qquad \text{for all $\theta \in \Mat$.}
$$
Now since
\[
f(a_t\Lambda_\theta)=f(a_t\Lambda_{\theta+\xi})
\qquad
\text{for all }\xi\in\M_{m\times n}(\mathbb{Z}),
\]
we get that
\begin{align}
    \label{eq: azx 2}
    \int_{\M_{m\times n}(\mathbb{R})}
\phi(\theta)f(a_t\Lambda_\theta)\,d\theta
=
\int_{\M_{m\times n}((0,1))}
f(a_t\Lambda_\theta)\,d\theta .
\end{align}
Substituting \eqref{eq: azx 2} into \eqref{eq: azx 1} yields the desired result. Hence proved.
\end{proof}

\begin{thm}[{\cite[Thm.~1.1]{BG21}}] 
\label{BG21 Thm}
    There is a $\delta_0 >0$  such that for all $f,g \in \R.\one + \Cc(\X_d)$ such that 
    $$
    \int_{\M_{m \times n}((0,1))} f(a_{{s}} \Lambda_\theta) g(a_{{t}} \Lambda_\theta)d\theta= \left(\int_{\X_d}f (x) d\mu_{\X_d}(x)\right) \left(\int_{\X_d} g(x) d\mu_{\X_d}(x) \right)+ O_{f,g}(e^{ -\delta_0 \min \{ \lfloor \tilde{s} \rfloor, \lfloor \tilde{t} \rfloor, \|\tilde s- \tilde t\|_\infty  \}}),
    $$
    for all ${s},{t} \in (\R_+)^{k+r-1} $, where the implied constant is independent of $s,t$.
\end{thm}

\begin{proof}[Proof of Theorem \ref{thm: Effective Equidistribution}]
Fix $f \in \R.\one + \Cc(\X_d)$. Then by Theorem \ref{KM1 Thm} and \ref{BG21 Thm}, we have 
\begin{align*}
    &\int_{\M_{m \times n}((0,1))} \left(f(a_t \Lambda_\theta) - \int_{\X_d} f(x) d\mu_{\X_d}(x) \right) \left(f(a_s \Lambda_\theta) - \int_{\X_d} f(x) d\mu_{\X_d}(x) \right) \, d\theta  \\
    &= \int_{\M_{m \times n}((0,1))} f(a_t \Lambda_\theta) f(a_s \Lambda_\theta) \, d\theta - \left(\int_{\X_d} f(x) d\mu_{\X_d}(x) \right)^2 \\
    &- (\int_{\X_d}f (x) d\mu_{\X_d}(x)) \left( \int_{\M_{m \times n}((0,1))} f(a_t \Lambda_\theta) \, d\theta - \int_{\X_d} f  (x) d\mu_{\X_d}(x)\right) \\
    &- (\int_{\X_d}f (x) d\mu_{\X_d}(x)) \left( \int_{\M_{m \times n}((0,1))} f(a_s \Lambda_\theta) \, d\theta - \int_{\X_d} f (x) d\mu_{\X_d}(x) \right)\\
    &= O_{f}(e^{ -\delta \min \{ \lfloor {\tilde s} \rfloor, \lfloor { \tilde t} \rfloor, \|\tilde s- \tilde t\|_\infty  \}}).
\end{align*}
Now by Theorem \ref{Pointwise Main Lemma}, the result follows. 
\end{proof}}
\vspace{0.2in}

{Theorem~\ref{thm: Effective Equidistribution} implies the following corollary, about Birkhoff genericity of $ \Lambda_\theta \in {\X_d}$ under $a_t$, for Lebesgue almost every $\theta \in \Mat$.}

\begin{cor}
\label{cor: Birkhoff genericity}
    For Lebesgue almost every $\theta \in \M_{m \times n}(\R)$, {the lattice} ${\Lambda}_\theta$ is  $(a_{t}, \mu_{\X_d})$-generic.
\end{cor}
\begin{proof}
    By Theorem \ref{thm: Effective Equidistribution}, we have that for a.e. $\theta \in \M_{m \times n}(\R)$
\begin{align}
\label{eq:d1}
    \frac{1}{m_{\R^{k+r-1}}(J^T)}  \int_{t \in J^T} f(a_t \Lambda_\theta) dt \rightarrow \int_{\X_d} f,
\end{align}
as $T \rightarrow \infty.$  Taking a countable dense subset of $\R.\one +\Cc(\X_d)$ and using a density argument, we can make sure that \eqref{eq:d1} holds for all $\theta$ in a fixed co-null subset of $\M_{m \times n}(\R)$ and for all $f \in \R. 1 + C_c(\X_d)$, {which proves the corollary.}
\end{proof}
\vspace{0.2in}

{The following lemma proves the equivalence of the Birkhoff genericity of $\Lambda_\theta$ in ${\X_d}$ to the Birkhoff genericity of $\Tilde{\Lambda}_\theta$ in $\XA_d$.}
\begin{lem}
\label{cor: Adele gen iff gen}
    Given $\theta \in \M_{m \times n}(\R)$, then $\Lambda_\theta$ is $(a_t, \mu_{\X_d})$-generic {if and only if} $\Tilde{\Lambda}_\theta$ is $(a_t, \mu_{\XA_d})$-generic.
\end{lem}
\begin{proof}
    It is well known that $\mu_{\XA_d}$ is an $\{a_{t}: t \in \R^{k+r-1} \}$-ergodic measure (see for instance \cite[Lem.~7.6]{SW22}). The corollary now follows by applying Theorem \ref{thm: Generecity in compact extension}, with $Y= \XA_d$, $m_Y= \mu_{\XA_d}$, $K=K_f$, $H= \{a_t: t\in \R^{k+r-1}\}$ and the sequence of measures
    $$
    \mu_{T}= \frac{1}{m_{\R^{k+r-1}}(J^T)} \int_{J^T} \delta_{a_t\Tilde{\Lambda}_\theta}.
    $$
\end{proof}

\begin{rem}
    In \cite{SW22}, it was asked in Remark 13.1 to find examples of measures $\mu$ on $\R^m$ such that for $\mu$-a.e. $\theta$, $\Tilde{\Lambda}_\theta$ is $(a_t, \mu_{\XA_d})$-generic (for the special case of $k=r=n=1$). Lemma \ref{cor: Adele gen iff gen}, along with results of \cite{SW22} and \cite{solanwieser}, gives a wide range of examples for which the above conclusion holds.
\end{rem}

{\begin{proof}[Proof of Proposition \ref{cor1 Generic}]
    The proposition directly follows from Lemma \ref{cor: Adele gen iff gen} together with Corollary \ref{cor: Birkhoff genericity}.
\end{proof}}

\section{The Cross-section}
\label{sec: The Cross-section}

{In {this section}, we construct a cross-section for a multi-parameter flow in the adelic space $\SL_d(\A)/\SL_d(\Q)$ and associate a natural measure to it. Given the lack of an existing theory for multi-parameter cross-sections, the analysis is somewhat involved. To facilitate this, we begin by constructing a cross-section in the more familiar space $\X_d = \SL_d(\R)/\SL_d(\Z)$, study its properties in detail, and then lift the construction to the adelic setting. The construction and analysis in the real case are presented in Sections~\ref{sec: The Cross-section}, \ref{sec: Explicit formula of the cross-sectional measure}, and \ref{sec:Properties of sed}, while the lifting to the adelic space is carried out in Section~\ref{sec:Adelizing the Setup}.}

 \subsection{Construction of the Cross-section} 
{{For $\e>0$,} we define
\begin{equation}
\label{eq: def Le}
    \Le := \left\{(x,y) \in (\Sphere^{m_1} \times \cdots \times \Sphere^{m_k} \times \Sphere^{n_1} \times \cdots \times \Sphere^{n_{r-1}}) \times \R^{n_r} : \|y\|^{n_r} \leq \e \right\}.
\end{equation}
The cross-section for our multi-parameter flow {$a_{t}$} on the space $\X_d$ is defined as
\begin{equation}
\label{eq:def sed}
    \sed := \X_d(\Le),
\end{equation}
by which we mean that for $\mu_{\X_d}$-almost every $\Lambda \in \X_d$, the set
\[
T_{\e}(\Lambda) := \{t \in \R^{k+r-1} : a_{t} \Lambda \in \sed\}
\]
is discrete. This is formalised in the following lemma.}

\begin{lem}
    For every $\Lambda \in \X_d$, the set $T_{\e}(\Lambda)$ is a discrete set.
\end{lem}
\begin{proof}
{Fix $\Lambda \in \X_d$ and let $t \in \R^{k+r-1}$. We argue by contradiction. Suppose then, that there exists a sequence $(s_l) \to 0$ such that $a_{s_l} a_t \Lambda \in \sed$ for all $l$. Then, for each $l$, there exists $v_l \in \Lambda$ such that $a_{s_l} a_t v_l \in \Le$, i.e.,
\[
a_t v_l \in a_{-s_l} \Le.
\]
This implies that
\begin{align}
    \label{ abcdef 1}
    a_t v_l \in \bigcup_{0 < \|s\| < \delta} a_{-s} \Le, \quad \text{for all $\delta > 0$ and correspondingly for all sufficiently large $l$.}
\end{align}
Using the discreteness of $\Lambda$ and the boundedness of the set $\bigcup_{0 < \|s\| < \delta} a_{-s} \Le$, the equation \eqref{ abcdef 1} implies that the sequence $(v_l)$ has a constant subsequence. Replacing $(v_l)$ by this subsequence, we may assume without loss of generality that $(v_l)$ is a constant sequence, say $v$. Then, by \eqref{ abcdef 1}, we have
\[
a_t v \in \bigcap_{\delta > 0} \bigcup_{0 < \|s\| < \delta} a_{-s} \Le = \emptyset,
\]
which is a contradiction. Hence, $T_{\e}(\Lambda)$ is a discrete set.}
\end{proof}

\begin{rem}
{Note that for $\mu_{\X_d}$-almost every lattice $\Lambda$, the set $T_\e(\Lambda)$ is unbounded. This follows from the observation that the union
\begin{align}
\label{eq: set non zero}
    \bigcup_{t \in [0,1]^{k+r-1}} a_t \sed
\end{align}
coincides with the set of all lattices in $\X_d$ that contain a nonzero vector in the set $\bigcup_{t \in [0,1]^{k+r-1}} a_t \Le$. Since this union contains an open subset of $\R^d$, Lemma~\ref{Sw22 Lem 8.2} implies that the set in \eqref{eq: set non zero} has positive $\mu_{\X_d}$-measure. }

{By the ergodicity of the $a_t$-action on $\X_d$, it follows that for $\mu_{\X_d}$-a.e.\ $\Lambda$, the set
\[
\{ t \in \R^{k+r-1} : a_t \Lambda \in \cup_{t \in [0,1]^{k+r-1}} a_t \sed \}
\]
is unbounded. In particular, this implies that $T_\e(\Lambda)$ is unbounded for $\mu_{\X_d}$-almost every $\Lambda \in \X_d$.}
\end{rem}

\subsection{Special subsets of the Cross-section}
We define some special subsets of $\sed$ {that} will be useful later. {First, consider}
\begin{align}
\label{eq:def Lej}
    \Le^+= \{z \in  \Le:  z_d \geq 0\}.
\end{align}
Note that 
\begin{align*}
    \sed= \X_d(\Le)= \X_d(\Le^+),
\end{align*}
{since for any \( z \in \Le \), either \( z \) or \( -z \) lies in \( \Le^+ \). }

{Define also}
\begin{align}
    \ssed&= \X_d^\sharp (\Le^+).
\end{align}
{By construction,} if \( \Lambda \in \ssed \) and \( v \in \Lambda \cap \Le^+ \), then \( v_d \neq 0 \). {Otherwise, both \( v \) and \( -v \) would lie in \( \Le^+ \), violating the condition that \( \Lambda \in \ssed \), which ensures no such symmetric pair occurs. }
    
{Next, for any \( \delta > 0 \),} define the subsets
\[
(\sed)_{\geq \delta} = \left\{ x \in \sed : \min \left\{ \|t\|_{\infty} : t \in T_{\e}(x) \setminus \{(0,0)\} \right\} \geq 2\delta \right\},
\]
{and}
\[
(\sed)_{< \delta} = \sed \setminus (\sed)_{\geq \delta}.
\]
This definition {ensures} that if \( Y \subset (\sed)_{\geq \delta} \), then the map
\[
(-\delta,\delta)^{k+r-1} \times Y \to \X_d, \quad (t,y) \mapsto a_{t} y,
\]
is injective.

\begin{lem}
\label{muJMsed>a}
For all \( \delta > 0 \), the sets \( (\sed)_{\geq \delta} \) and \( (\sed)_{< \delta} \) are Borel subsets of \( \sed \).
\end{lem}
\begin{proof}
Since \( (\sed)_{\geq \delta} = \sed \setminus (\sed)_{< \delta} \), it {suffices to show} that \( (\sed)_{< \delta} \) is Borel for all \( \delta > 0 \).

Let
{\[
W_\delta = \bigcup_{ 0<\|t\|_{\infty} < 2\delta } a_{t}^{-1}(\Le),
\]}
which is a Borel subset of \( \R^d \). {Then,}
\[
(\sed)_{< \delta} = \left\{ \Lambda \in \sed : \Lambda_\prim \cap W_\delta \neq \emptyset \right\} = \X_d(\Le) \cap \X_d(W_\delta),
\]
which is Borel by Lemmas~\ref{SW22 Lem 8.1} and~\ref{Sw22 Lem 8.2}.
\end{proof}

\subsection{Measure on the cross-section}
\label{sec: The Cross-section: Measure on the cross-section}
For every Borel measurable set \( E \subset \sed \), we define the quantity \( \mu_{\sed}(E) \) by
\begin{equation}
    \label{defmu_S sup}
    \mu_{\sed}(E) := \sup_{\tau > 0} \left\{ \frac{1}{\tau^{k+r-1}} \mu_{\X_d}(E^{I_\tau}) \right\},
\end{equation}
{where 
\begin{align}
    I_\tau &:= [0,\tau]^{k+r-1}, \label{eq: def: I tau} \\
    E^I &:= \{ a_{t} x : t \in I,\ x \in E \}. \label{eq: defE^I}
\end{align}}
This section {is devoted to proving the following proposition, which establishes that \( \mu_{\sed} \) defines a measure on \( \sed \), referred to as a cross-sectional measure.}

\begin{prop}
\label{prop mu is measure}
The map \( \mu_{\sed} \) from {Borel} subsets of \( \sed \) to \( \R_{\geq 0} \), defines a measure on \( \sed \).
\end{prop}
\medskip

{To prove this proposition, we will require the following lemmas.}

\begin{lem}
\label{lem:mu_S_limit}
For every Borel measurable set \( E \subset \sed \), we have
\begin{equation}
    \label{defmu_S}
    \mu_{\sed}(E) = \lim_{\tau \to 0} \frac{1}{\tau^{k+r-1}} \mu_{\X_d}(E^{I_\tau}).
\end{equation}
\end{lem}
\begin{proof}
{The result is trivial} if \( \mu_{\sed}(E) = 0 \), so we assume \( \mu_{\sed}(E) > 0 \).

    Let \( 0 < \e' < \mu_{\sed}(E) \) be arbitrary. We {aim} to show that there exists \( \delta > 0 \) such that for all \( \tau < \delta \),
\[
\frac{1}{\tau^{k+r-1}} \mu_{\X_d}(E^{I_\tau}) > \mu_{\sed}(E) - \e'.
\]
By the definition of the {supremum, there exists} \( \kappa > 0 \) such that
\[
\frac{1}{\kappa^{k+r-1}} \mu_{\X_d}(E^{I_\kappa}) > \mu_{\sed}(E) - \frac{\e'}{2}.
\]
Choose \( l \in \mathbb{N} \) large enough that
\[
\left(1 + \frac{1}{l} \right)^{k+r-1} \leq \frac{\mu_{\sed}(E) - \frac{\e'}{2}}{\mu_{\sed}(E) - \e'},
\]
{and define} \( \delta = \kappa/l \). Now {for any} \( 0 < \tau < \delta \), {there exists} a unique \( t \in \mathbb{N} \), {with} \( t \geq l \), such that 
\[
\tau \in \left( \frac{\kappa}{t+1}, \frac{\kappa}{t} \right].
\] {This implies:}
{\begin{align*}
    \mu_{\sed}(E) - \frac{\e'}{2}
    &\leq \frac{1}{\kappa^{k+r-1}} \mu_{\X_d}(E^{I_\kappa}) \leq \frac{1}{(t\tau)^{k+r-1}} \mu_{\X_d}(E^{I_{(t+1)\tau}}) \\
    &= \frac{1}{(t\tau)^{k+r-1}} \mu_{\X_d} \left( \bigcup_{\substack{0 \leq j_1, \ldots, j_{k+r-1} \leq t \\ j_i \in \Z}} a_{(j_1 \tau, \ldots, j_{k+r-1} \tau)} E^{I_\tau} \right).
\end{align*}
We estimate the measure of the union using the $a_t$-invariance of \( \mu_{\X_d} \):
\begin{align*}
    \mu_{\X_d} \left( \bigcup_{\substack{0 \leq j_1, \ldots, j_{k+r-1} \leq t \\ j_i \in \Z}} a_{(j_1 \tau, \ldots, j_{k+r-1} \tau)} E^{I_\tau} \right)
    &\leq \sum_{\substack{0 \leq j_1, \ldots, j_{k+r-1} \leq t \\ j_i \in \Z}} \mu_{\X_d}(a_{(j_1 \tau, \ldots, j_{k+r-1} \tau)} E^{I_\tau}) \\
    &= (t+1)^{k+r-1} \mu_{\X_d}(E^{I_\tau}).
\end{align*}
Substituting back, we obtain
\begin{align*}
    \mu_{\sed}(E) - \frac{\e'}{2}
    &\leq \left(1 +  \frac{1}{t} \right)^{k+r-1} \left( \frac{1}{\tau^{k+r-1}} \mu_{\X_d}(E^{I_\tau}) \right) \\
    &\leq \left( 1 + \frac{1}{l} \right)^{k+r-1} \left( \frac{1}{\tau^{k+r-1}} \mu_{\X_d}(E^{I_\tau}) \right) \\
    &\leq \frac{\mu_{\sed}(E) - \frac{\e'}{2}}{\mu_{\sed}(E) - \e'} \left( \frac{1}{\tau^{k+r-1}} \mu_{\X_d}(E^{I_\tau}) \right).
\end{align*}
Rearranging this inequality gives
\[
\frac{1}{\tau^{k+r-1}} \mu_{\X_d}(E^{I_\tau}) > \mu_{\sed}(E) - \e',
\]
as desired. This completes the proof.}
\end{proof}

\begin{lem}
    \label{cross sectional measure on restricted set}
    Let $\delta > 0$ be arbitrary. Then {the following holds} for any Borel subset $E \subset (\sed)_{\geq \delta}$ and Borel subset $I \subset [0, \delta]^{k+r-1}$
    \begin{align}
    \label{eq: cross sectional measure on restricted set }
        \mu_{{\X_d}}(E^I)= \mu_{\sed}(E) m_{\R^{k+r-1}}(I).
    \end{align}
\end{lem}
\begin{proof}
Fix a Borel subset $E \subset (\sed)_{\geq \delta}$. {First, note that by injectivity of} the map $(t,x) \mapsto a_tx$ on $(-\delta, \delta)^{k+r-1} \times (\sed)_{\geq \delta}$, it {suffices to prove} that \eqref{eq: cross sectional measure on restricted set } holds for any Borel set {$I \subset [0, \delta]^{k+r-1}$} of the form $[x_1,x_1+y) \times \cdots \times [x_{k+r-1}, x_{k+r-1}+ y)$. {Next, using the invariance of $\mu_{\X_d}$ under $a_t$, and the identity
    $$
     E^{[x_1,x_1+y) \times \cdots \times [x_{k+r-1}, x_{k+r-1}+ y) }  = a_{(x_1, \ldots, x_{k+r-1})} E^{I_y},
    $$
    we reduce to proving \eqref{eq: cross sectional measure on restricted set } for sets of the form $I_y$.}

    {To prove this, fix $y > 0$, and for each $N \in \mathbb{N}$, decompose $I_y$ as a union of $N^{k+r-1}$ many boxes of the form 
    $$
      \left( \frac{l_1y}{N} ,\ldots,  \frac{l_{k+r-1 } y}{N} \right) + I_{y/N},
    $$
    where $l_i \in \{0, \ldots, N-1\}$ for all $i$. This implies
    \begin{align}
        E^{I_y} = \bigcup_{0 \leq l_1, \ldots, l_{k+r-1} \leq N-1} a_{ (l_1y/N,\ldots, l_{k+r-1}y/N)} E^{I_{y/N}}. \label{ eq: aaabb 1}
    \end{align}
Now, using injectivity on $(-\delta, \delta)^{k+r-1} \times (\sed)_{\geq \delta}$, we have
\begin{align}
    \mu_{\X_d} \left( \bigcup_{0 \leq l_1, \ldots, l_{k+r-1} \leq N-1} a_{ (l_1 y/N,\ldots, l_{k+r-1} y/N)} E^{I_{y/N}} \right) 
&= \sum_{0 \leq l_1, \ldots, l_{k+r-1} \leq N-1} \mu_{\X_d} \left( a_{ (l_1 y/N,\ldots, l_{k+r-1} y/N)} E^{I_{y/N}} \right) \nonumber\\
&= N^{k+r-1} \mu_{\X_d}(E^{I_{y/N}}) ,\label{ eq: aaabb 2}
\end{align}
where the last equality follows using $a_t$-invariance of $\mu_{\X_d}$.}

{Finally, combining \eqref{ eq: aaabb 1}, \eqref{ eq: aaabb 2} and taking the  limit as $N \rightarrow \infty$, we obtain
\begin{align*}
\mu_{\X_d}(E^{I_y}) 
&= \lim_{N \rightarrow \infty}N^{k+r-1} \mu_{\X_d}(E^{I_{y/N}}) 
= \lim_{N \rightarrow \infty} \frac{m_{\mathbb{R}^{k+r-1}}(I_y)}{m_{\mathbb{R}^{k+r-1}}(I_{y/N})} \mu_{\X_d}(E^{I_{y/N}}) \\
&= m_{\mathbb{R}^{k+r-1}}(I_y) \, \mu_{\sed}(E), \quad \text{by Lemma \ref{lem:mu_S_limit}.}
\end{align*}
This proves the lemma.}
\end{proof}

\begin{lem}
\label{ lem : claim 2}
For every Borel set $E \subset \sed$, {the following holds:}
{\begin{align}
\label{claim2}
    \mu_{\sed}(E)= \sum_{l=1}^\infty \mu_{\sed}(E(l)),
\end{align} }
where 
\begin{align}
\label{eq: abcccc 1}
    E(l)= \begin{cases}
        E \cap (\sed)_{\geq 1/4} & \text{for } l=1, \\
        E \cap (\sed)_{\geq 1/2^{l+1}} \cap (\sed )_{<1/2^l} & \text{for } l>1.
    \end{cases}
\end{align}
\end{lem}
\begin{proof}
{By the definition of $\mu_{\sed}$, it is clear that }
{\[
\mu_{\sed}(E) \leq \sum_{l=1}^\infty \mu_{\sed}(E(l)).
\]}
{To prove the reverse inequality, observe that for every $N \in \N$,}
{\begin{align*}
    \mu_{\sed}(E) &\geq \mu_{\sed}\left(\bigcup_{l=1}^{N-1}  E(l) \right) \\
    &= 2^{(N(k+r-1)}\mu_{{\X_d}}\left( \left(\bigcup_{l=1}^{N-1} E(l) \right)^{I_{1/2^N}} \right) && \text{(by Lemma \ref{cross sectional measure on restricted set})} \\
    &= 2^{N(k+r-1)} \sum_{l=1}^{N-1}  \mu_{{\X_d}} \left( E(l)^{I_{1/2^N}} \right) && \text{(by disjointness of the sets $E(1)^{I_{1/2^N}}, \ldots, E(N-1)^{I_{1/2^N}}$)} \\
    &= \sum_{l=1}^{N-1}  \mu_{\sed}(E(l)) && \text{(again by Lemma \ref{cross sectional measure on restricted set}).}
\end{align*}}
{Taking the limit as $N \to \infty$ yields the desired inequality, completing the proof.}
\end{proof}

\medskip

\begin{proof}[Proof of Proposition \ref{prop mu is measure}]
The fact that $\mu_{\sed}(\emptyset)=0$ is {immediate}. We now {verify} countable additivity. 

Let $E$ be the disjoint union of sets $(E_i)_{i \in \N}$, {i.e., \( E = \bigsqcup_{i \in \N} E_i \). For each \( i \in \N \), decompose \( E_i = \bigsqcup_{l \in \N} E_i(l) \), and similarly decompose \( E = \bigsqcup_{l \in \N} E(l) \), where the sets $E(l), E_i(l)$ are defined as in \eqref{eq: abcccc 1}. }It is clear that 
\[
E(l) = \bigsqcup_{i \in \N} E_i(l).
\]
{Moreover, Lemma \ref{cross sectional measure on restricted set} implies that}
\[
\mu_{\sed}(E(l)) = \sum_{i \in \N} \mu_{\sed}(E_i(l)).
\]
The result now follows from Lemma \ref{ lem : claim 2} and Fubini’s Theorem.
\end{proof}


\section{Explicit formula of Cross-sectional measure}
\label{sec: Explicit formula of the cross-sectional measure}
{This section provides {an} explicit formula for the cross-sectional measure $\mu_{\sed}$. To that end, we begin by introducing some notation.}

{Recall from Remark \ref{rem:reduction of main thm} that $(\E_d, m_{\E_d})$ denotes the measure space $(\E_d^d, m_{\E_d^d})$. For notational simplicity, let us define
\begin{align}
    \label{eq: def C e}
    \cL_\e = \{y \in \R^{n_r} : y_{n_r} > 0, \ \|y\|^{n_r} \leq \e\}.
\end{align}
We also define the map
{\begin{align}
    \label{eq:def varphi genral}
    \varphi: \E_d \times \Sphere^{m_1} \times \cdots \times &\Sphere^{n_{r-1}} \times \cL_\e \longrightarrow \sed, \\
    \varphi(\Lambda, x, x', \gamma) &= u(x, x', \gamma) \Lambda, \nonumber
\end{align}}
where 
\begin{align}
    \label{eq: def u general}
    u(x, x', \gamma) = 
    \begin{pmatrix}
        \gamma^{\frac{-1}{d-1}} I_{d-1} & \begin{pmatrix} x \\ x' \end{pmatrix} \\
        0 & \gamma
    \end{pmatrix},
\end{align}
{for $x\in \Sphere^{m_1}\times\cdots\times\Sphere^{n_{r-1}}$, $x'\in\R^{n_r-1}$ and $\gamma\in\R_{+}$.}}

{To see that \(\varphi\) is well defined, note that for any \(\Lambda\in\E_d\), the standard basis vector \(e_d\in\Lambda\).  Hence for every 
\[
v=(x,x',\gamma)\in \Sphere^{m_1}\times\cdots\times\Sphere^{n_{r-1}}\times\cL_\e,
\]
we have 
\[
v \;=\; u(v)\,e_d \;\in\; u(v)\,\Lambda,
\]
and since \(v\in\Le\), this shows \(u(v)\Lambda\in\sed\) and therefore the map \eqref{eq:def varphi genral} is well defined.}

\vspace{0.2in}

The main aim of this section is to prove the following proposition.
\begin{prop}
\label{measure equal general}
The measure $\mu_{\sed}$ {is given by} 
{\begin{equation}
    \label{mu_sed measure general}
    \mu_{\sed} = \frac{m_1 \cdots m_k n_1 \cdots n_{r-1}}{\zeta(d)} \cdot \varphi_* \left( m_{\E_d} \times \mu^{(\Sphere^{m_1})} \times \cdots \times \mu^{(\Sphere^{n_{r-1}})} \times m_{\R^{n_r}}|_{\cL_\e} \right).
\end{equation}}
\end{prop} \vspace{0.1in}

We will need the following lemmas before proceeding.
\begin{lem}
\label{lem:Pushforward Measure general}
    Let $\Phi$ denote the map 
    $$
 \Phi: \R^{k+r-1} \times \Sphere^{m_1} \times \cdots \times \Sphere^{n_{r-1}} \times \R^{n_r-1} \times \R_+ \times H \rightarrow G, \ \Phi(s,v, h) \mapsto a_{s}u(v)h,
 $$ 
  where $H \leq \SL_d(\R)$ equals {
\begin{align}
\label{eq: def H}
    H := \left\{\begin{pmatrix}
        A & 0 \\ x & 1
    \end{pmatrix}: A \in \SL_{d-1}(\R), x \in \R^n \right\}.
\end{align}}
 {Then $\Phi$ is a continuous, injective and open map.  Furthermore, the pullback of $m_G$ equals 
 $$
 \frac{m_1 \cdots m_k n_1 \cdots n_{r-1}}{\zeta(d)}  \left(m_{\R^{k+r-1}}  \times \mu^{(\Sphere^{m_1})} \times \cdots \times \mu^{(\Sphere^{n_{r-1}})} \times m_{\R^{n_r-1}} \times m_{\R_+} \times m_{H} \right),
 $$
 where $m_{\R_+}$ denotes the restriction of Lebesgue measure to $\R_+$ and $m_G$ (resp.\ $m_H$) denotes the Haar measure on $G$ (resp.\ $H$), normalized so that the fundamental domain of $\Gamma$ (resp.\ $\Gamma_H := \Gamma \cap H$) in $G$ (resp.\ $H$) has measure $1$.}
\end{lem}   
\begin{proof}
The continuity of $\Phi$ is clear. The injectivity of $\Phi$ follows from the fact that one can recover the values of $(s,v)$ from the last column of $\Phi(s,v,h)$ and then we can recover information about $h$ as $(a_s  u(v))^{-1}\Phi(s,v,h)$. Since $\Phi$ is injective, so by the {invariance of domain} theorem, we have that $\Phi $ is open. 

Now, we consider the pullback of the measure $m_G$, and denote it by $\nu$. Using left and right invariance of $m_G$ under the action of $G$, we get that $\nu = m_{\R^{k+r-1}} \times \nu' \times m_H$ for some measure $\nu'$ on $\Sphere^{m_1} \times \cdots \times \Sphere^{n_{r-1}} \times \R^{n_r-1} \times \R_+$. To prove the proposition, it is enough to show that 
    {\begin{align}
    \label{ eq: eq 1}
        \nu'= \frac{m_1 \cdots m_k n_1 \cdots n_{r-1}}{\zeta(d)}  \left( \mu^{(\Sphere^{m_1})} \times \cdots \times \mu^{(\Sphere^{n_{r-1}})} \times m_{\R^{n_r-1}} \times m_{\R_+} \right) .
    \end{align}}
    {To establish \eqref{ eq: eq 1}, fix subsets $E_i \subset \Sphere^{m_i}$ for all $i = 1, \ldots, k$ and $E_{j} \subset \Sphere^{n_{j-k}}$ for all $k+1 \leq j \leq k+r-1$ and $E_{k+r} \subset \R^{n_r-1}$ and $0<\alpha<\beta$. Also fix a subset $X \subset H$ of measure $1$ with respect to $m_H$ and set 
    $$
    E_0= \left\{ \left( \frac{t_1}{m_1},\ldots, \frac{t_k}{m_k}, -\frac{t_{k+1}}{n_1}, \ldots, -\frac{t_{k+r-1}}{n_{r-1}} \right):  t_i \in [0,1] \text{ for all } i= 1 , \ldots , k+r-1 \right\}.
    $$
    Then, to prove \eqref{ eq: eq 1}, it suffices to verify that
    \begin{align}
        &\nu(E_0 \times E_1 \times \cdots \times E_{k+r} \times [\alpha,\beta] \times X) \nonumber \\
        &\quad =\frac{m_1 \cdots m_kn_1 \cdots n_{r-1}}{\zeta(d)} m_{\R^{k+r-1}}(E_0) \mu^{(\Sphere^{m_1})}(E_1) \cdots \mu^{(\Sphere^{n_r-1})}(E_{k+r-1}) m_{\R^{n_r-1}}(E_{k+r})m_{\R_+}([\alpha, \beta]) m_H(X) \nonumber \\
        &\quad =\frac{1}{\zeta(d)}  \mu^{(\Sphere^{m_1})}(E_1) \cdots \mu^{(\Sphere^{n_r-1})}(E_{k+r-1}) m_{\R^{n_r-1}}(E_{k+r})(\beta-\alpha). \label{ eq: eq 2}
    \end{align}}

    { To prove \eqref{ eq: eq 2}, consider the map }
     \begin{equation}
     \nonumber
         \varPsi : \R \times \R^{d-1} \times H \rightarrow G
     \end{equation}
     given by
     \begin{equation}
     \nonumber
         \varPsi(t, z, h) = \begin{pmatrix}
             e^t I_{d-1} \\  &e^{-(d-1)t}
         \end{pmatrix} \cdot \begin{pmatrix}
             I_{d-1} & z \\ &1
         \end{pmatrix} \cdot h.
     \end{equation}
     {It is proved in the proof of \cite[Prop.~8.7]{SW22} that $\varPsi$ is an injective continuous open map such that }
     \begin{equation}
     \nonumber
         \frac{d-1}{\zeta(d)} (m_{\R} \times m_{\R^{d-1}} \times m_H)(B)= m_G(\varPsi(B))
     \end{equation}
     for all Borel subset $B \subset \R \times \R^{d-1} \times H$.
     {Noting} that
     \begin{equation}
     \nonumber
         \Phi(\R^{k+r-1} \times \Sphere^{m_1} \times \cdots \times \Sphere^{n_{r-1}} \times \R^{n_r-1} \times \R_+ \times H ) \subset \varPsi(\R \times \R^{d-1} \times H),
     \end{equation}
    {we get}
    \begin{align}
        &\nu(E_0 \times E_1 \times \cdots \times E_{k+r} \times [\alpha,\beta] \times X) \nonumber\\
        &= m_G\left(\Phi(E_0 \times E_1 \times \cdots \times E_{k+r} \times [\alpha,\beta] \times X) \right) \nonumber\\
        & =\frac{d-1}{\zeta(d)} (m_{\R} \times m_{\R^{d-1}} \times m_H)(\varPsi^{-1}(\Phi(E_0 \times E_1 \times \cdots \times E_{k+r} \times [\alpha,\beta] \times X))). \label{eq: abd 0}
    \end{align}
    {Note that the set $\varPsi^{-1}(\Phi(E_0 \times E_1 \times \cdots \times E_{k+r} \times [\alpha,\beta] \times X))$ equals
    \begin{align}
    \label{eq: abd 1}
        \left\{
\left(t, \left(e^{-t +\frac{t_1}{m_1}}x_1, \ldots, e^{-t + \frac{t_{k+r-1}}{n_{r-1}}}x_{k+r-1},  e^{-dt} \gamma^{-1}  x_{k+r} \right), g_{(\underline{t},\gamma)} h
\right)
\;\middle|\;
\begin{array}{l}
t = \dfrac{t_1 + \cdots + t_{k+r-1} - n_r \log \gamma}{n_r(d-1)}, \\[6pt]
 x_i \in E_i, \quad t_i \in [0,1] \text{ for all } i, \\[6pt] \gamma \in [\alpha, \beta], \quad h \in X
\end{array}
\right\},
    \end{align}
where 
    $$
    g_{(\underline{t},\gamma)}= \begin{pmatrix} e^{-t} \gamma^{\frac{-1}{(d-1)}} . \begin{pmatrix}
            e^{t_1/m_1}I_{m_1} & \\ & \ddots \\ & & e^{t_{k+r-1}}I_{n_{r-1}} \\&&&e^{-(t_1+ \cdots t_{k+r-1})/ n_r}I_{n_r-1}
        \end{pmatrix} & \\ &1
        \end{pmatrix}.
    $$
Using Fubini’s theorem, the measure of the set in \eqref{eq: abd 1} with respect to $m_{\R} \otimes m_{\R^{d-1}} \otimes m_H$ is the same as the measure of the set
\begin{align}
    \label{eq: abd 2}
        \left\{
\left(t, \left( e^{-t +\frac{t_1}{m_1}}x_1, \ldots, e^{-t + \frac{t_{k+r-1}}{n_{r-1}}}x_{k+r-1},  e^{-dt} \gamma^{-1}  x_{k+r}\right)
\right)
\;\middle|\;
\begin{array}{l}
t = \dfrac{t_1 + \cdots + t_{k+r-1} - n_r \log \gamma}{n_r(d-1)}, \\[6pt]
 x_i \in E_i, \quad t_i \in [0,1] \text{ for all } i, \\[6pt] \gamma \in [\alpha, \beta]
\end{array}
\right\},
    \end{align}
 with respect to $m_{\R} \times m_{\R^{d-1}}$. Now, using the decomposition 
 $$
 m_{\R^{d-1}}= m_{\R^{m_1}} \otimes \cdots \otimes m_{\R^{n_{r-1}}} \otimes m_{\R^{n_r-1}},
 $$  
 and then using the isomorphism $\R^l \simeq  \R_{\geq 0} \times \Sphere^l$ given by $x \mapsto (\|x\|^l, \frac{x}{\|x\|})$, to further decompose $m_{\R^{m_1}}, \ldots , m_{\R^{n_{r-1}}}$, we get that the measure of the set in \eqref{eq: abd 2} is the same as the measure of the set
 \begin{align}
   \nonumber
        \left\{
\left( \left(t,  e^{-m_1t +{t_1}}, \ldots, e^{-n_{r-1}t + t_{k+r-1}},  e^{-dt} \gamma^{-1}  x_{k+r} \right), \left(x_1, \ldots, x_{k+r-1}\right)
\right)
\;\middle|\;
\begin{array}{l}
t = \dfrac{t_1 + \cdots + t_{k+r-1} - n_r \log \gamma}{n_r(d-1)}, \\[6pt]
 x_i \in E_i, \quad t_i \in [0,1] \text{ for all } i, \\[6pt] \gamma \in [\alpha, \beta]
\end{array}
\right\}
    \end{align}
with respect to $m_{\R^{k+r-1+ n_{r}}} \otimes \mu^{(\Sphere^{m_1})} \otimes \cdots \otimes \mu^{(\Sphere^{n_{r-1}})}$.

By another application of Fubini's theorem, this equals
\[
\mu^{(\Sphere^{m_1})}(E_1)\cdots
\mu^{(\Sphere^{n_{r-1}})}(E_{k+r-1})
\]
times the measure of
\[
\left\{
\left(t,  e^{-m_1t + t_1}, \ldots, e^{-n_{r-1}t + t_{k+r-1}},  e^{-dt} \gamma^{-1} x_{k+r} \right)
\;\middle|\;
\begin{array}{l}
t = \dfrac{t_1 + \cdots + t_{k+r-1} - n_r \log \gamma}{n_r(d-1)}, \\[6pt]
x_{k+r} \in E_{k+r},\quad t_i \in [0,1] \text{ for all } i, \\[6pt]
\gamma \in [\alpha, \beta]
\end{array}
\right\}
\]
with respect to $m_{\R^{k+r-1+n_r}}$.

To compute this measure, consider the map
\[
(t_1,\ldots,t_{k+r-1},\gamma,x_{k+r})
\mapsto
\left(
t,
e^{-m_1t+t_1},
\ldots,
e^{-n_{r-1}t+t_{k+r-1}},
e^{-dt}\gamma^{-1}x_{k+r}
\right),
\]
where
\[
t=\frac{t_1+\cdots+t_{k+r-1}-n_r\log\gamma}{n_r(d-1)}.
\]
A direct computation of the Jacobian determinant shows that its absolute value equals $\frac{1}{d-1}$. Integrating over the domain
\[
t_i\in[0,1],\qquad \gamma\in[\alpha,\beta],\qquad x_{k+r}\in E_{k+r},
\]
we obtain that the above measure equals
\[
\frac{1}{d-1}\, m_{\R^{n_r-1}}(E_{k+r})(\beta-\alpha).
\]

Therefore, the measure of the set in \eqref{eq: abd 1} equals
\[
\frac{1}{d-1}
\mu^{(\Sphere^{m_1})}(E_1)\cdots
\mu^{(\Sphere^{n_{r-1}})}(E_{k+r-1})
\,m_{\R^{n_r-1}}(E_{k+r})(\beta-\alpha).
\]
Thus \eqref{eq: abd 0} implies \eqref{ eq: eq 2}, and the lemma follows.

}\end{proof}

\begin{lem}
\label{two points zero measure}
    The set $\ssed$ is a co-null open subset of $\sed$.
\end{lem}
\begin{proof}
By Lemma~\ref{SW22 Lem 8.1}, we have that $\ssed = {\X_d}^\sharp(\Le^+)$ is an open subset of $\sed = {\X_d}(\Le^+)$. To show that $\ssed$ is co-null {with respect to} $\mu_{\sed}$, it suffices to prove that {$(\sed \setminus \ssed)^{\R^{k+r-1}} = {\X_d}(\Le^+,2)^{\R^{k+r-1}}$} has zero measure {with respect to} $\mu_{{\X_d}}$.

To see {this}, note that if $\Lambda \in {\X_d}(\Le^+,2)^{\R^{k+r-1}}$, then $\Lambda$ contains two primitive vectors {of the form} $a_t v, a_t w$, where $t \in \R^{k+r-1}$ and $v, w \in \Le^+$. 

If $v$ and $w$ are linearly independent, then $\Lambda$ {admits} a basis containing {$a_t v$ and $a_t w$}, which satisfy
\[
\|\rho_1(\varrho_1(a_t v))\| = \|\rho_1(\varrho_1(a_t w))\| = e^{t_1}.
\]
{Otherwise $v=-w$, and hence we have $\Lambda_{\mathrm{prim}} \cap (\R^{d-1} \times \{0\}) \neq \emptyset$.}

{In either case,} ${\X_d}(\Le^+,2)^{\R^{k+r-1}}$ is contained in the set
{\begin{align}
\label{eq:temp1}
\left\{ A \cdot \Gamma : \text{either } \|\rho_1(\varrho_1(A_{\cdot 1}))\| = \|\rho_1(\varrho_1(A_{\cdot 2}))\| \text{ or } A_{dd} = 0 \right\},
\end{align}
where $A_{\cdot j}$ denotes the $j$-th column of $A$, and $A_{ij}$ denotes the $(i,j)$-th entry of the matrix $A$.}

{It is straightforward to verify that} the set in~\eqref{eq:temp1} has zero measure with respect to $\mu_{{\X_d}}$. {This completes the proof.}
\end{proof}

\begin{lem}
\label{lem: General psi continuous}
Define {the map
\[
\phi: \ssed \rightarrow \E_d \times \Sphere^{m_1} \times \cdots \times \Sphere^{n_{r-1}} \times \R^{n_r-1} \times \R_{+}
\]
by
\begin{equation}
\label{eq: def phi general}
    \phi(\Lambda) = \left( u(\vector(\Lambda))^{-1} \Lambda,\ \vector(\Lambda) \right),
\end{equation}
where \(\vector(\Lambda)\) defined in~\eqref{defv-unique}, equals the unique primitive vector of $\Lambda$ in $\Le^+$.} Then \(\phi\) {is the} inverse of \(\varphi|_{\varphi^{-1}(\ssed)}\), and is a homeomorphism between \(\ssed\) and \(\varphi^{-1}(\ssed)\).
\end{lem} 
\begin{proof}
{First we check that \(\phi\) is well defined, i.e.\ that for every \(\Lambda\in\ssed\), the vector
\[
v \;:=\; \vector(\Lambda)
\]
satisfies \(v\in\Sphere^{m_1}\times\cdots\times\Sphere^{n_{r-1}}\times\R^{n_r-1}\times\R_{+}\).  Equivalently, we must show its last coordinate \(v_{d}\) is strictly positive.}

{Suppose, by way of contradiction, that \(v_{d}=0\).  Then both \(v\) and its negative \(-v\) lie in {$\Le^+$ defined as in \eqref{eq:def Lej}}. But \(\Lambda\in\ssed=\X^\sharp(\Le^+)\) admits exactly one primitive vector in \(\Le^+\), so we cannot have two distinct primitives \(v\) and \(-v\).  This contradiction forces \(v_{d}>0\), proving \(\phi\) is well defined.}

{Continuity of \(\phi\) follows from Lemma~\ref{SW22 Lem 8.1}(iii), and \(\varphi\) is also continuous by construction.  Finally, one checks directly that \(\phi\) and \(\varphi\bigl|_{\varphi^{-1}(\ssed)}\) are mutual inverses.  Hence \(\phi\) is a homeomorphism onto \(\varphi^{-1}(\ssed)\).}
 
\end{proof}

{We now prove Proposition \ref{measure equal general}.}

\begin{proof}[Proof of Proposition \ref{measure equal general}]
{Throughout this proof, let \(H \), \( m_H\) and $m_G$ be as in Lemma~\ref{lem:Pushforward Measure general}.}

{We first show that \(\varphi^{-1}(\sed\setminus \ssed)\) has measure zero. Observe that \((\Lambda_0, x, x', \gamma) \in \varphi^{-1}(\sed \setminus \ssed)\) if and only if the transformed lattice \(u(x, x', \gamma)\Lambda_0\) contains a point in \(\Le^+ \setminus (x, x', \gamma)\). Equivalently, this means that \(\Lambda_0\) contains a point in the set  
\[
W_{(x, x', \gamma)} := u(x, x', \gamma)^{-1} \Le^+ \setminus \bfe_d.
\]
A simple computation shows that 
\[
W_{(x, x', \gamma)} \cap \R \bfe_d = \emptyset,
\]
which in turn implies that  
\[
\E_d \cap \X_d(W_{(x, x', \gamma)}) \subset \left\{ A \cdot \Gamma : A_{\cdot 1} \in W_{(x, x', \gamma)},\ A_{\cdot d} = \bfe_d \right\},
\]
where \(A_{\cdot j}\) denotes the \(j\)-th column of the matrix \(A\). Since \(W_{(x,x',\gamma)}\) has Lebesgue measure zero and the map
\[
H \;\longrightarrow\;\R^d,\qquad h\longmapsto h\,\bfe_1
\]
pushes \(m_H\) forward to a constant multiple of Lebesgue measure on \(\R^d\), we get that the set
\[
\left\{ A \in \SL_d(\R) : A_{\cdot 1} \in W_{(x, x', \gamma)},\ A_{\cdot d} = \bfe_d \right\} \subset H
\]
has measure zero with respect to \(m_H\). Consequently,
\[
m_{\E_d}\left( \E_d \cap \X_d(W_{(x, x', \gamma)}) \right) = 0.
\]
Now, since
\[
\varphi^{-1}(\sed \setminus \ssed) = \left\{ (\Lambda_0, x, x', \gamma) \in \E_d \times \Le^+ : \Lambda_0 \in \X_d(W_{(x, x', \gamma)}) \right\},
\]
Fubini's theorem implies that \(\varphi^{-1}(\sed \setminus \ssed)\) has measure zero.}

\smallskip

{ Since $\varphi^{-1}(\sed \setminus \ssed)$ has measure zero, it suffices, by Lemma~\ref{lem: General psi continuous}, to prove \eqref{mu_sed measure general} locally for points of $\ssed$. That is, for every $x_0 = \varphi(y_0) \in \ssed$, there exists a neighborhood $\mathcal{U} = \varphi(\mathcal{V})$ of $x_0$ with $\mathcal{V} \subset \varphi^{-1}(\ssed)$ open in $\E_d \times \Sphere^{m_1} \times \cdots \times \Sphere^{n_{r-1}} \times \R^{n_r-1} \times \R_{+}$ containing $y_0$, such that
\begin{align}
    \label{eq: xxxz 1}
    \mu_{\sed}|_{\mathcal{U}} = \frac{m_1 \cdots m_k \, n_1 \cdots n_{r-1}}{\zeta(d)} \cdot \varphi_* \left( \left( m_{\E_d} \times \mu^{(\Sphere^{m_1})} \times \cdots \times \mu^{(\Sphere^{n_{r-1}})} \times m_{\R^{n_r}}|_{\cL_\e} \right)|_{\mathcal{V}} \right).
\end{align}
To prove this, fix $x_0 = \varphi(\Lambda_0, v) \in \ssed$. Using Lemmas~\ref{lem:Pushforward Measure general} and~\ref{lem: General psi continuous}, and the discreteness of $\Gamma$ and $\Gamma_H$, we can find open sets $\mathcal{W}_H \subset H$, $\mathcal{W}_G \subset G$, $U \subset \Sphere^{m_1} \times \cdots \times \Sphere^{n_{r-1}} \times \cL_\e$, and $\delta>0$ satisfying
\begin{align}
    \label{eq: llal 1}
    \{ a_t : t \in [0, \delta]^{k+r-1} \} \cdot \{ u(v) : v \in U \} \cdot \mathcal{W}_H \subset \mathcal{W}_G,
\end{align}
such that $\Lambda_0 \in \{h \cdot \Gamma_H : h \in \mathcal{W}_H\}$, $v \in U$, and $\mathcal{W}_H$ and $\mathcal{W}_G$ are contained in fundamental domains for the right actions of $\Gamma_H$ on $H$ and $\Gamma$ on $G$, respectively.

Note that $\mathcal{W}_H$ and $\mathcal{W}_G$ map homeomorphically onto their images in $\E_d$ and $\X_d$, respectively, and the pushforwards of the measures $m_H|_{\mathcal{W}_H}$ and $m_G|_{\mathcal{W}_G}$ equal the restriction of $m_{\E_d}$ and $m_{\X_d}$ to these images. Thus, to prove \eqref{eq: xxxz 1} for $\mathcal{V} = \mathcal{W}_H \times U$, equation~\eqref{eq: llal 1} and the definition of $\mu_{\sed}$ imply that it suffices to show that the map
\[
\Phi: \R^{k+r-1} \times \Sphere^{m_1} \times \cdots \times \Sphere^{n_{r-1}} \times \R^{n_r-1} \times \R_{+} \times H \to G, \quad (t, v, h) \mapsto a_t u(v) h
\]
pushes forward the measure
\[
\frac{m_1 \cdots m_k \, n_1 \cdots n_{r-1}}{\zeta(d)} \cdot \left( m_{\R^{k+r-1}} \times \mu^{(\Sphere^{m_1})} \times \cdots \times \mu^{(\Sphere^{n_{r-1}})} \times m_{\R^{n_r}}|_{\cL_\e} \times m_H \right)
\]
to $m_G|_P$, where $P$ is the image of $\Phi$. This follows directly from Lemma~\ref{lem:Pushforward Measure general}.

This completes the proof of \eqref{mu_sed measure general}, and hence the proof of the  proposition.}

\end{proof}


\section{Properties of Cross-section}
\label{sec:Properties of sed}

\begin{lem}
\label{muJM alpha}
   For {every} $\delta>0$, the set $(\sed)_{\geq \delta}$ is  $\mu_{\sed} $-JM. 
\end{lem}
\begin{proof}
{Fix \(\delta > 0\).} Since \(\partial_{\sed}((\sed)_{\geq \delta}) = \partial_{\sed}((\sed)_{< \delta})\), it {suffices} to prove that \((\sed)_{< \delta}\) is \(\mu_{\sed}\)-JM. {To prove this, let \(\Le^\circ\) denote the set}
{\begin{align}
\label{eq: def L zero}
    \Le^\circ = \{(x,y) \in \R^{d-n_r} \times \R^{n_r}: (x,y) \in \Le, \|y\|^{n_r} < \e\}.
\end{align}}
{Also, define the following sets:}
{\begin{align*}
    W_\delta &= \bigcup_{\|t\|_{\el} \leq \delta} a_{t}^{-1}(\Le), \\
    W_\delta^\circ &= \bigcup_{0< \|t\|_{\el} < \delta} a_{t}^{-1}(\Le^\circ), \\
    D_\e &= \bigcup_{\|t\|_{\el} \leq \delta} a_{t}^{-1}(\Sphere^{m_1} \times \cdots \times \Sphere^{m_k} \times  \Sphere^{n_1} \times \cdots \times \Sphere^{n_{r-1}} \times (\e^{1/n_r} \Sphere^{n_r})), \\
    F &= \bigcup_{\|t\|_{\el}= \delta} a_{t}^{-1}(\Le).
\end{align*}}
It is clear that \(W_\delta^\circ\) is an open subset of \(\R^d\), {while} \(W_\delta\), \(D_\e\), and \(F\) are closed subsets of \(\R^d\), and we have
\[
W_\delta = W_\delta^\circ \cup D_\e \cup F \cup \Le.
\]

{Note that} \((\sed)_{< \delta}\) is contained in the set
\[
E_1 := {\X_d}(\Le) \cap {\X_d}(W_\delta,4),
\]
which is closed in \(\sed\) by Lemma \ref{SW22 Lem 8.1}. {Also, by Lemmas~\ref{SW22 Lem 8.1} and~\ref{Sw22 Lem 8.2}, we deduce that the set}
\[
E_2 := {\X_d}^\sharp(\Le^+) \cap {\X_d}(W_\delta^\circ,2) \subset (\sed)_{< \delta}
\]
is open in \(\sed\). It follows that
\[
\partial_{\sed}((\sed)_{< \delta}) \subset E_1 \setminus E_2.
\]

It is easy to see that
\begin{equation}
\label{E_1 minus E_2}
E_1 \setminus E_2 \subset {\X_d}(\Le^+,2) \cup {\X_d}(D_\e) \cup {\X_d}(F) .
\end{equation}

We now show that {the intersection of each of the sets on the right-hand side of \eqref{E_1 minus E_2} with \(\sed\) has \(\mu_{\sed}\)-measure zero.}

First, the fact that \(\mu_{\sed}({\X_d}(\Le^+,2)) = 0\) follows from Lemma \ref{two points zero measure}.

For the second {term}, note that
\[
{\X_d}(D_\e)^{\R^{k+r-1}} = {\X_d}\left(\left\{(x_1, \ldots, x_k, y_1, \ldots, y_r) \in \R^{m_1} \times \cdots \times \R^{n_r} : \prod_{i=1}^k \|x_i\|^{m_i} \cdot \prod_{j=1}^r \|y_j\|^{n_j} = \e \right\}\right).
\]
By Lemma \ref{SW22Lem8.8}, we {conclude} that \(\mu_{{\X_d}}({\X_d}(D_\e)^{\R^{k+r-1}}) = 0\), {and hence} \(\mu_{\sed}(\sed \cap {\X_d}(D_\e)) = 0\), {by definition of $\mu_{\sed}$.}

{Next}, we show that \(\mu_{\sed}(\sed \cap {\X_d}(F)) = 0\). {Suppose} \(\Lambda \in (\sed \cap {\X_d}(F))^{\R^{k+r-1}}\). Then there exist {\(s = (s_1, \ldots, s_{k+r-1}),\  t = (t_1, \ldots, t_{k+r-1}) \in \R^{k+r-1}\) with \(\|t\|_{\infty} = \delta\),} such that
\[
\Lambda \in a_s \left({\X_d}(\Le) \cap {\X_d}(a_t \Le) \right).
\]
Thus, \(\Lambda\) has a basis containing two vectors \(v, w \in \Lambda_{\prim}\), satisfying
\[
v \in a_s(\Sphere^{m_1} \times \cdots \times \Sphere^{n_{r-1}} \times \R^{n_r}), \quad w \in a_s(a_t(\Sphere^{m_1} \times \cdots \times \Sphere^{n_{r-1}} \times \R^{n_r})).
\]
Since there exists \(1 \leq j \leq k+r-1\) such that \(t_j = \pm \delta\), we get
\[
\|\rho_j(\varrho_1(v))\| = e^{\mp \delta} \|\rho_j(\varrho_1(w))\| \quad \text{if } 1 \leq j \leq k,
\]
or
\[
\|\rho_{j-k}'(\varrho_2(v))\| = e^{\mp \delta} \|\rho_{j-k}'(\varrho_2(w))\| \quad \text{if } k+1 \leq j \leq k+r.
\]
This {implies} that \((\sed \cap {\X_d}(F))^{\R^{k+r-1}}\) is contained in the set
\[
\left\{ A \cdot \Gamma: \begin{array}{l}
\text{either } \|\rho_j(\varrho_1(A_{\cdot 1}))\| = e^{\delta} \|\rho_j(\varrho_1(A_{ \cdot 2}))\| \text{ for some } 1 \leq j \leq k, \\
\text{or } \|\rho_{j-k}'(\varrho_2(A_{ \cdot 1}))\| = e^{\delta} \|\rho_{j-k}'(\varrho_2(A_{ \cdot 2}))\| \text{ for some } k+1 \leq j \leq k+r
\end{array} \right\},
\]
{which clearly has \(\mu_{{\X_d}}\)-measure zero. Here $A_{\cdot j}$ denotes the $j$-th column of $A$. This completes the proof.} 
\end{proof}

\begin{lem}
\label{lem: reasonable}
{Let $L_\e^\circ$ be as in \eqref{eq: def L zero}, and define
\begin{align}
    \label{eq: def ued}
    \ued := \ssed \cap {\X_d}(L_\e^\circ).
\end{align}
Then} $\ued$ is an open subset of $\sed$ and satisfies the following properties:
\begin{enumerate}
    \item {$\mu_{{\X_d}}((\cl_{\X_d}(\sed) \setminus \ued)^{(0,1)^{k+r-1}}) = 0$, where $\cl_{\X_d}(\cdot)$ denotes the closure of set in ${\X_d}$.}
    \item The map $(0,1)^{k+r-1} \times \ued \to {\X_d}$, given by $(t, x) \mapsto a_t x$, is open.
\end{enumerate}
\end{lem}
\begin{proof}
From Lemma \ref{SW22 Lem 8.1}, we know that $\ued$ is open in $\sed$, and that $\sed$ is closed {in} ${\X_d}$. Moreover, Proposition~\ref{measure equal general} and Lemma~\ref{two points zero measure} {imply} that $\ued$ is a co-null open subset of $\sed$. {Therefore,}
{\begin{align*}
    \mu_{{\X_d}}((\cl_{\X_d}(\sed) \setminus \ued)^{(0,1)^{k+r-1}}) 
    &= \mu_{{\X_d}}((\sed \setminus \ued)^{(0,1)^{k+r-1}}) \\
    &= \mu_{\sed}(\sed \setminus \ued) \cdot m_{\R^{k+r-1}}((0,1)^{k+r-1}) = 0,
\end{align*}}
{proving part (1).}

{ For part~(2), recall from Lemma~\ref{lem:Pushforward Measure general} that the map
\[
\Phi: \R^{k+r-1} \times \Sphere^{m_1} \times \cdots \times \Sphere^{n_{r-1}} \times \R^{n_r-1} \times \R_+ \times H \to G, \quad (s, v, h) \mapsto a_s u(v) h,
\]
is open. This induces an open map
\[
\tilde{\Phi}: \R^{k+r-1} \times \E_d \times \Sphere^{m_1} \times \cdots \times \Sphere^{n_{r-1}} \times \cL_\e^\circ \to {\X_d},
\]
given by $\tilde{\Phi}(s, h\Z^d, v) = a_s u(v) h \Z^d$, where
\[
\cL_\e^\circ = \{ y \in \R^{n_r} : y_{n_r} > 0, \ \|y\|^{n_r} < \e \}.
\]

By Lemma~\ref{lem: General psi continuous}, the map $\phi$ is a homeomorphism onto its image. Moreover, $\phi(\ued)$ is open, since it is the inverse image of the open set $\ued$ under the continuous map $\varphi$. Therefore, $\phi|_{\ued}$ is open. 

Clearly,
\[
\phi(\ued) \subset \E_d \times \Sphere^{m_1} \times \cdots \times \Sphere^{n_{r-1}} \times \cL_\e^\circ.
\]
Hence, the composition
\[
(t, x) \longmapsto a_t x \;=\; \tilde{\Phi} \circ (\mathrm{id}_{\R^{k+r-1}}, \phi|_{\ued}),
\]
from $(0,1)^{k+r-1} \times \ued$ to ${\X_d}$, is open as the composition of open maps. This proves part~(2).}
\end{proof}

\section{Adelizing the Setup}
\label{sec:Adelizing the Setup}
{We now consider the lift of the cross-section to the adelic space and study its properties.  }

\begin{thm}
\label{thm: Adeles}
{Let 
\begin{align}
    \label{eq: def adeles}
    \seda:= \pi^{-1}(\sed), \quad \ \sseda:= \pi^{-1}(\ssed), \text{ and} \quad \ \ueda := \pi^{-1}(\ued).
\end{align}
Then the following statements hold.}
    \begin{enumerate}
        \item The set $\seda$ is closed in $\XA_d$.
        \item For every $\Lambda \in \XA_d$, the set $\Tilde{T}_{\e}(\Lambda) := \{t \in \R^{k+r-1}: a_{t}\Lambda \in \seda \}$ is a discrete set.
        \item For any $\delta > 0$, the sets $(\seda)_{ \geq \delta}$ and $(\seda)_{< \delta}$ are Borel subsets of $\seda$, where 
        $$(\seda)_{\geq \delta}= \{ x \in \seda: \min\{\|t\|_{\el}: t \in \tilde{T}_{\e}(x) \setminus \{(0,0)\}\} \geq 2\delta \}= \pi^{-1}((\sed)_{\geq \delta}),$$ $$(\seda)_{< \delta}= \seda \setminus (\seda)_{\geq \delta}.$$
        \item {For every Borel measurable set \( E \subset \seda \), the limit
{\begin{equation}
    \label{defmu_Sa}
    \mu_{\seda}(E) = \lim_{\tau \to 0} \frac{1}{\tau^{k + r - 1}} \mu_{\XA_d}(E^{I_\tau})
\end{equation}}
exists and defines a measure \( \mu_{\seda} \) on the adelic cross-section. Here, \( I_\tau \) and \( E^{I} \) are defined as in equations~\eqref{eq: def: I tau} and~\eqref{eq: defE^I}.}
    \item The cross-sectional measure $\mu_{\seda}$ satisfies $$\pi_*(\mu_{\seda})= \mu_{\sed}.$$
    \item $\sseda$ is a co-null open subset of $\seda$.
    \item For every $\delta>0$, we have that for any Borel set $E \subset (\seda)_{\geq \delta}$ and Borel subset $I \subset [0,\delta]^{k+r-1}$, the following holds
    $$
    \mu_{\XA_d}(E^I)= \mu_{\seda}(E) m_{\R^{k+r-1}}(I).
    $$
    \item For any $\delta > 0$, the sets $(\seda)_{ \geq \delta}$ and $(\seda)_{< \delta}$ are $\mu_{\seda}$-JM subsets of $\seda$.
    \item The set $\ueda$ is an open subset of $\seda$.
    \item $\mu_{\XA_d}((\cl_{\XA_d}(\seda) \setminus \ueda)^{(0,1)^{k+r-1}})=0$.
    \item The map {$(0,1)^{k+r-1} \times \ueda \rightarrow \XA_d, ((t, x) \mapsto a_{t}x)$} is open.
    \end{enumerate}
\end{thm}
{Using the results proved in Sections \ref{sec: The Cross-section}, \ref{sec: Explicit formula of the cross-sectional measure}, and \ref{sec:Properties of sed}, Theorem \ref{thm: Adeles} follows directly. We therefore omit the proof.}

\subsection{Description of $\mu_{\seda}$}
{This section aims to prove Proposition \ref{prop:pushforward general}, which is an analogue of Proposition~\ref{measure equal general} for the measure $\mu_{\seda}$.}

{Let us define the maps \(\tilde \psi, \tilde{\psi}^-: \sseda \rightarrow \ZZ_d\) by
    \begin{align}
        \label{defpsi 1}
        \tilde \psi &= (\vartheta \circ \phi \circ \pi, \psi_f), \\
        \tilde \psi^- &= (\vartheta^- \circ \phi \circ \pi, -\psi_f), \label{defpsi 2}
    \end{align}
where $\phi$ is defined as in~\eqref{eq: def phi general}, and
\[
\vartheta, \vartheta^- : \E_d \times \Sphere^{m_1} \times \cdots \times \Sphere^{n_{r-1}} \times \R^{n_r - 1} \times \R_{+} \rightarrow \E_d \times \Sphere^{m_1} \times \cdots \times \Sphere^{n_r} \times [0, \e]
\]
are defined by}
{\begin{align*}
    \vartheta(\Lambda, x, y, y') &= \left( w\left(x, \tfrac{(y, y')}{\|(y, y')\|}, \|(y, y')\|\right) u(x, y, y') \Lambda,\ x,\ \tfrac{(y, y')}{\|(y, y')\|},\ \|(y, y')\|^{n_r} \right),  \\
    \vartheta^-(\Lambda, x, y, y') &= \left( w\left(x, \tfrac{(y, y')}{\|(y, y')\|}, \|(y, y')\|\right) u(x, y, y') \Lambda,\ -x,\ -\tfrac{(y, y')}{\|(y, y')\|},\ \|(y, y')\|^{n_r} \right), 
\end{align*}
where for $(\alpha, \beta, \gamma) \in \R^{d-1} \times \R_+ \times \R_+$, the matrix $w(\alpha, \beta, \gamma) \in \SL_d(\R)$ is defined as
\begin{align}
    \label{eq:def v}
    w(\alpha, \beta, \gamma) = 
    \begin{pmatrix}
        I_{d - n_r} & \\
        & \gamma I_{n_r}
    \end{pmatrix}
    \begin{pmatrix}
        (\gamma^{n_r} \beta)^{-1/(d-1)} I_{d-1} & \alpha \\
        & \beta
    \end{pmatrix}.
\end{align}}

\noindent
Furthermore, the map $\psi_f: \sseda \rightarrow \hZp^d$ is defined as follows.
For $\tilde{\Lambda} \in \sseda$, {let us} write $\tilde{\Lambda} = (g_\infty, g_f)\SL_d(\Q)$ with $g_f \in K_f$, and define $\Lambda = \pi(\tilde{\Lambda}) = g_\infty \Z^d$. Since $\Lambda \in \ssed$, the vector $\vector(\Lambda)$, defined as in \eqref{defv-unique}, is the unique primitive vector of $\Lambda$ lying in $\Le^+$. The columns of $g_\infty$ form a basis for the lattice $\Lambda$. {Thus, by choosing a suitable representative of the coset $g_\infty \SL_d(\Z)$}, we may assume that the $d$-th column of $g_\infty$ is equal to $\vector(\Lambda)$. The uniqueness of $\vector(\Lambda)$ ensures that {if another representative $g_\infty \gamma$ with $\gamma \in \SL_d(\Z)$ also satisfies this condition, then we must have}
\[
g_\infty \bfe_d = \vector(\Lambda) = g_\infty \gamma \bfe_d,
\]
{which implies that $\gamma$ belongs to $H \cap \SL_d(\Z)$ (defined as in \eqref{eq: def H}). We now define}
\[
\psi_f(\Lambda) := g_f \bfe_d,
\]
{that is, if $g_f = (g_p)_{p \in \mathbf{P}}$, then $\psi_f(\tilde{\Lambda})$ is the element of $\hat{\Z}^d$ whose $p$-th coordinate is the {$d$-th} column of $g_p$. The discussion above shows that this definition is independent of the chosen representative, and hence $\psi_f$ is well defined.}

\begin{lem}
\label{lem: adelic psi continuous}
 The maps $\tilde \psi$ and $\tilde \psi^-$ are $K_f$-equivariant and continuous, {where the action of $K_f$ on $\ZZ_d$ is given by the product of the trivial action on $\E_d \times \Sphere^{n_1} \times \cdots \times \Sphere^{n_r} \times (0,\epsilon)$ and the usual action on $\widehat{\mathbb{Z}}^d$.}
\end{lem}
\begin{proof}
    The fact that $\tilde{\psi}$ (respectively, $\tilde{\psi}^-$) commutes with the $K_f$-action on $\sseda$ and on $\ZZ_d$ follows {immediately} from the definition of $\tilde{\psi}$ (respectively, $\tilde{\psi}^-$). From Lemma \ref{lem: General psi continuous}, it follows that $\vartheta \circ \phi \circ \pi$  (respectively, $\vartheta^- \circ \phi \circ \pi$) is continuous, so we have to establish the continuity of $\psi_f$. The proof of the latter fact follows along the same lines as the argument in Lemma 10.1 of \cite{SW22}, hence is skipped. Thus, the lemma follows.
\end{proof}

\begin{prop}
\label{prop:pushforward general}
    The measure $\mu_{\seda}$ satisfies
   { $$
    (\tilde \psi)_* \mu_{\seda} + (\tilde \psi^-)_* \mu_{\seda} =   \frac{m_1\cdots m_k n_1 \cdots  n_{r-1}}{\zeta(d)} \tmu^{d}. $$}
\end{prop}

\begin{proof}
First of all note that using Theorem \ref{thm: Adeles}(5),  Proposition \ref{measure equal general}, Lemmas~\ref{two points zero measure} and~\ref{lem: General psi continuous}, we can easily see that 
\begin{align}
        \label{eq: Pushforward genral}
        (\vartheta  \circ \phi \circ \pi)_* \mu_{\seda} + (\vartheta^-  \circ \phi \circ \pi)_* \mu_{\seda}=  \frac{m_1 \cdots m_k n_1 \cdots n_{r-1}}{\zeta(d)} (m_{\E_d} \times \mu^{(\Sphere^{m_1})} \times \cdots \times \mu^{(\Sphere^{n_r})} \times m_{\R}|_{[0,\e]}),
    \end{align}
{We denote by $P$ the natural projection}
{$$
P: \E_d \times \Sphere^{n_1} \times \cdots \times \Sphere^{n_r} \times [0,\e] \times \hZp^d \longrightarrow \E_d  \times \Sphere^{n_1} \times \cdots \times \Sphere^{n_r} \times [0,\e].
$$}
Then we have a commutative diagram:
\[
  \begin{tikzcd}[column sep = 4.5cm]
    \text{$\sseda$} \arrow{r}{ \text{$\tilde \psi$ (resp. $\tilde \psi^-$)}} \arrow{d}{\text{$\pi$}} & \text{$\E_d \times \Sphere^{m_1} \times \cdots \times \Sphere^{n_r} \times [0,\e] \times \hZp^d $} \arrow{d}{\text{$P$}} \\
    \text{$\ssed$} \arrow{r}{\text{$\vartheta  \circ \phi $ (resp. $ \vartheta^-  \circ \phi $)}} & \text{$ \E_d \times \Sphere^{m_1} \times \cdots \times \Sphere^{n_r} \times [0,\e] $}
  \end{tikzcd}
\]
Note that the measure $\mu_{\seda}$ is {invariant under the action of $K_f$}. By {the} $K_f$-equivariance of $\tilde{\psi}$ (respectively, $\tilde{\psi}^-$), {it follows that the pushforward measure} $(\tilde{\psi})_* \mu_{\seda}$ (respectively, $(\tilde{\psi}^-)_* \mu_{\seda}$) is {also $K_f$-invariant. Since $K_f$ acts transitively on $\hat{\Z}^d_\prim$, and $m_{\hat{\Z}^d_\prim}$ is the unique $K_f$-invariant probability measure on this space, we conclude} that $(\tilde{\psi})_* \mu_{\seda}$ (respectively, $(\tilde{\psi}^-)_* \mu_{\seda}$) must be {equal to} the product of the measures $P_* \big((\vartheta \circ \phi)_* \mu_{\sed}\big)$ (respectively, $P_* \big((\vartheta^- \circ \phi)_* \mu_{\sed}\big)$) and $m_{\hat{\Z}^d_\prim}$. Also note that by commutativity of the diagram, we have $P \circ \vartheta  \circ \phi=  \vartheta  \circ \phi \circ \pi$ and $P \circ \vartheta^-  \circ \phi=  \vartheta^-  \circ \phi \circ \pi$. Thus proposition follows from \eqref{eq: Pushforward genral}.
\end{proof}


\subsection{Thickening preserves JM}

{In this section, we show that if a set \( E \subset \seda \) is Jordan measurable (as defined in Definition~\ref{defJM}), then the thickened set \( E^{J^\tau} \) is also Jordan measurable. We now state the precise result.}

\begin{lem}
\label{Translated is JM}
{ For any \(\mu_{\seda}\)-Jordan measurable subset \(E\) and any \(\tau > 0\), the thickened set \(E^{J^\tau}\) (defined in \eqref{eq: def J^T} and \eqref{eq: defE^I}) is \(\mu_{\XA_d}\)-Jordan measurable.}
\end{lem}
\begin{proof}
 We want to show that $\mu_{\XA_d}(\partial_{\XA_d}(E^{J^\tau}))=0$. This will follow once we show 
    \begin{equation}
    \label{partial equaliity}
        \partial_{\XA_d}(E^{J^\tau}) \subset (\seda \setminus \ueda)^{J^\tau} \cup (\partial_{\seda} E)^{J^\tau} \cup \seda^{\partial_{\R^{k+r-1}}( J^\tau)},
    \end{equation}
    {where $\ueda$ is defined as in \eqref{eq: def adeles}.}

   Indeed, all the sets on the right hand side of \eqref{partial equaliity} are $\mu_{\XA_d}$-null: The first one by choice of $\ueda$, and the second one by the assumption on $E$ and the definition of the measure $\mu_{\seda}.$ {For the third set, note that $\pi(\seda^{\partial_{\R^{k+r-1}} (J^\tau)})$ is contained in the union of sets
    $$
    {\X_d} (\{(x_1, \ldots, x_{k+r}) \in \R^{m_1} \times \cdots \times \R^{m_k} \times \R^{n_1} \times \cdots \times \R^{n_r}: \|x_i\| \in \{1, e^\tau\} \} ),
    $$
    as $i$ varies in $1, \ldots, k+r-1$, along with the set
    $$
    {\X_d}(\{(x_1, \ldots, x_{k+r}) \in \R^{m_1} \times \cdots \times \R^{m_k} \times \R^{n_1} \times \cdots \times \R^{n_r}: \|x_1\|^{m_1}\cdots \|x_{k+r-1}\|^{n_{r-1}} \in \{1, e^{n_r\tau}\} \} ).
    $$}
    By Lemma \ref{SW22Lem8.8}, we get that each of the above sets in the union has  measure zero. Thus by Theorem \ref{thm: Adeles}, we get that $\seda^{\partial_{\R^{k+r-1}} (J^\tau)}$ has zero measure.

\medskip

{  We now prove \eqref{partial equaliity}. Let \( x \in \partial_{\XA_d}(E^{J^\tau}) \). Then there exists a sequence \( t_\ell \in J^\tau \) and \( y_\ell \in E \) such that \( a_{t_\ell} y_\ell \to x \). Passing to a subsequence, we may assume \( t_\ell \to t_0 \in J^\tau \) and \( y_\ell \to y_0 \in \cl_{\XA_d}(E) \), so that \( x = a_{t_0} y_0 \). {Since \( \seda \) is closed in \( \XA_d \), we have \( \cl_{\XA_d}(E) = \cl_{\seda}(E) \), so in particular \( y_0 \in \cl_{\seda}(E) \subset \seda \).}

\smallskip

 We now consider different possibilities:

\textbf{Case 1:} \( y_0 \notin \ueda \). Then \( x \) lies in the right-hand side of \eqref{partial equaliity}.

\textbf{Case 2:} \( y_0 \in \ueda \) and \( t_0 \in \partial_{\R^{k+r-1}}(J^\tau) \). Then again, \( x \) belongs to the right-hand side.

\textbf{Case 3:} \( y_0 \in \ueda \), \( t_0 \in \Int_{\R^{k+r-1}}(J^\tau) \), but \( y_0 \notin \Int_{\seda}(E) \). Then \( y_0 \in \partial_{\seda}(E) \), and once again \( x \) lies in the right-hand side.}

{
\textbf{The remaining case:} 
\[
y_0 \in \ueda \cap \Int_{\seda}(E)
\quad \text{and} \quad
t_0 \in \Int_{\R^{k+r-1}}(J^\tau).
\]
Let \( t_1 = (1/2, \ldots, 1/2) \in \R^{k+r-1} \). By Theorem~\ref{thm: Adeles}, the map
\[
(0,1)^{k+r-1} \times \ueda \longrightarrow \XA_d,
\qquad
(t, y) \longmapsto a_{t_0 - t_1} a_t y
\]
is open. In particular, \( E^{J^\tau} \) contains the open set
\[
\left\{ a_t y : 
\begin{array}{l}
t \in \Int_{\R^{k+r-1}}(J^\tau) \cap \left(t_0 - t_1 + (0,1)^{k+r-1}\right), \\
y \in \ueda \cap \Int_{\seda}(E)
\end{array}
\right\},
\]
which contains the point \( x = a_{t_0} y_0 \). This contradicts the assumption that \( x \in \partial_{\XA_d}(E^{J^\tau}) \), i.e., that \( x \notin \Int_{\XA_d}(E^{J^\tau}) \). Therefore, no such \( x \) exists, completing the proof of \eqref{partial equaliity} and hence the lemma.}
\end{proof}


\section{Cross-section Correspondence}
\label{sec:Cross-section Correspondence}
This section is devoted to connecting dynamics on the cross-section we have constructed in the previous sections, with Diophantine approximation. More concretely, we relate successive time tuples $t_i$ for which $a_{t_i}\tilde \Lambda_\theta \in \seda,$ to Diophantine properties of $\theta$. Let us introduce some notation before proceeding. For a Borel subset $E$ of $\seda$, $\beta \in \R$, $x \in \XA_d$ and $T>0$, we define $N(x,T,\beta, E)$ as the set
\begin{equation} \label{def:N(x,T,E)}
N(x,T,\beta, E) :=\left\{ 
(t_1, \ldots, t_{k+r-1}) \in (\mathbb{R}_{\geq 0})^{k+r-1} : 
\left.
\begin{array}{l}
a_t x \in E, \\ t_j \leq T \text{ for all } k+1 \leq j \leq k+r-1, \\
0 \leq \sum_{i=1}^k m_it_i \leq \sum_{j=1}^{r-1}n_jt_{k+j}+  n_rT -\beta
\end{array}
\right.
\right\}.
\end{equation}

\begin{lem}
\label{Number ineq}
    Let $\theta  \in \M_{m \times n}(\R)$ be given. Suppose that $(p_l, q_l) \in \Z^{m} \times \Z^{n} $ is the sequence of $\e$-approximations (ordered according to increasing $\|q_l\|$), which satisfy the condition that the $n$-th component of $q_l$ is non-zero and $\disp(\theta,p_l , q_l) \neq 0$.
    Define a sequence ${t}_l= (t_{1,l}, \ldots, t_{k+r-1,l} )$ as solutions of the equations
    \begin{align*}
        e^{-t_{i,l} } &= \|\rho_i(p_l + \theta q_l)\| \text{ for all } 1\leq i \leq k \\
        e^{t_{i,l} } &= \| \rho_{i-k}'(q_l)\| \text{ for all } k+1\leq i \leq k+r-1. \\
    \end{align*}
    Then, if we define $\Y_T(\theta)= \{{t}_l: \|q_l\| \leq e^T \}$, we have that 
    $$
    N(\tilde \Lambda_\theta, T,\log \e , \seda ) \cap ((-\log \eta_1, \infty) \times \cdots \times (-\log \eta_k, \infty) \times \R^{r-1}) \subset \Y_T(\theta).
    $$
\end{lem}
\begin{proof}
    Clearly if {$t= (t_1, \ldots, t_{k+r-1}) \in N(\tilde \Lambda_\theta, T,\log \e, \seda ) \cap ((-\log \eta_1, \infty) \times \cdots \times (-\log \eta_k, \infty) \times \R^{r-1})$}, then there exists a vector $(p,q)\in (\Z^{m} \times \Z^{n} )_\prim$ such that $ a_{t} ( p + q \theta, q) \in \Le$. This implies that for $i=1,\ldots, k$, we have $ \|\rho_i(p + \theta q)\| = e^{-t_i} \leq \eta_i$ since $t_i \geq - \log \eta_i  $. Similary, for $1 \leq j \leq r-1$, we have {$\|\rho_j'(q)\|= e^{t_{k+j}} \leq e^T$. Also, we have 
    $$
    \|\rho_r'(q)\|^{n_r} \leq \e e^{m_1t_1+ \cdots +m_kt_k- n_1t_{k+1} -\cdots -n_{k+r-1}t_{k+r-1}} \leq \e e^{n_rT-\log \e }= e^{n_rT}.
    $$ 
    \noindent Finally we have $\disp(\theta,p,q)= \left( \prod_{i=1}^k \|\rho_i(e^{t_i}(p+\theta q))\|^{m_i} \right). \left(\prod_{j=1}^r  \|\rho_j'(e^{-t_{j+k}}q)\|^{n_j}\} \right) \in  (0,\e]$,} where $t_{k+r}= 1/n_r(m_1t_1 + \cdots + m_kt_k - n_1y_{k+1} -\cdots -n_{r-1}y_{k+r-1})$.
    Thus, $t$ belongs to $\Y_T(\theta)$. This proves the lemma.
\end{proof}

\begin{lem}
\label{lem: Important Commutativity}
    Keeping the notations as in Lemma \ref{Number ineq}, for almost every $\theta \in \M_{m \times n}(\R)$, the following holds. The map $(p_{l},q_l) \mapsto t_l$ is a {two-to-one} map. In fact, if $(p_l, q_l)$ maps to $t_l$, then the other one mapping to $t_l$ is $(-p_l, -q_l)$. Also for each $l$,  we have $a_{t_l}\tilde\Lambda_\theta \in \ssed$. Moreover, if the $n$-th coordinate of $(q_l)$ is positive, then
    \begin{align}
        \tilde \psi({a_{t_l}\tilde\Lambda_\theta}) &= (\lambda_d(\theta, p_l, q_l),\proj(\theta,p_l,q_l), \disp(\theta,p_l,q_l) ,(p_l,q_l) ) \nonumber \\ \label{eq: Image under psi}
         \tilde \psi^-({a_{t_l}\tilde\Lambda_\theta}) &= (\lambda_d(\theta, -p_l, -q_l),\proj(\theta,-p_l,-q_l), \disp(\theta,-p_l,-q_l) ,(-p_l,-q_l) ). 
    \end{align}
\end{lem}
\begin{proof}
    Suppose  $(p,q)$ and $(p', q')$ are $\e$-approximates of $\theta$ (satisfying conditions of Lemma \ref{Number ineq}) such that both of these {vectors} map to same vector $t \in \R^{k+r-1}$. Then by definition of $t$, we have 
    \begin{align}
        \label{eq: temp equality}
        \|\rho_i(p + q\theta)\|= \|\rho_i(p' + q'\theta)\| \ \ \ \ \ (=e^{-t_i})
    \end{align}
    for all $i=1, \ldots, k$. Now, for fixed $(p,q) \neq \pm (p',q')$, the measure of solutions to \eqref{eq: temp equality} is zero. Since $\Z^{2d}$ is countable, we get that for almost every $\theta \in \M_{m \times n}(\R)$, \eqref{eq: temp equality} implies that $(p,q)= \pm (p', q')$. This proves the first part of the lemma. The second part is clear.
\end{proof}

\subsection{Lower Bound for Counting}
\label{Reduction of Theorem}

{The above correspondence between best approximation and the hitting to the cross-section gives the following lower bound on the counting of best approximation with constraints.}

\begin{lem}
\label{mainmult3}
    For almost every $\theta \in \M_{m \times n}(\R)$, the following holds. For any $\mu_{\seda}$-JM set $E \subset \seda$ and $\beta>0$, we have 
    \begin{align}
    \label{eq:temp7}
        \liminf_{T \rightarrow \infty} \frac{1}{m_{\R^{k+r-1}}(J^T)} \# N(a_{\eta_0}\tilde\Lambda_\theta, T,\beta, E) \geq \mu_{\seda}(E),
    \end{align}
    where $\eta_0$ is defined as
    \begin{align}
\label{def eta 0}
    \eta_0 := (-\log \eta_1, \ldots, -\log \eta_k, 0, \ldots, 0) \in \R^{k+r-1}.
\end{align}
\end{lem}
\begin{proof}

We claim that the \eqref{eq:temp7} holds for all $\theta \in \M_{m \times n}(\R)$ such that {$\Tilde{\Lambda}_\theta$} is $(a_t, \mu_{\XA_d})$-generic. To see this, fix a $\mu_{\seda}$-JM subset $E \subset \seda$. By Theorem \ref{thm: Adeles}(8) and Lemma \ref{JMintersect}, for any small enough $\alpha>0$, we have $E \cap (\seda)_{\geq n \alpha}$ is $\mu_{\seda}$-JM. By Lemma \ref{Translated is JM}, we have that $(E \cap (\seda)_{\geq n \alpha})^{J^{\alpha}}$ is $\mu_{\XA_d}$-JM. Using the fact that {$\Tilde{\Lambda}_\theta$} is $(a_t, \mu_{\XA_d})$-generic, along with Theorem \ref{muJM} and \eqref{eq: measure of J T}, we have
    \begin{align}
    \label{eq: nn temp 8}
     \lim_{T \rightarrow \infty}  \frac{m_1 \cdots m_kn_1\cdots n_{r-1}. (k+r-1)!}{T^{k+r-1}c_{k+r-1}(n_1, \ldots, n_r)}  \int_{J^T} \ind_{(E \cap (\seda)_{\geq n \alpha})^{J^{\alpha}}}({ a_{\eta_0} a_{t}\Tilde{\Lambda}_\theta})\, dt = \mu_{\XA_d}(a_{\eta_0}^{-1}(E \cap (\seda)_{\geq n \alpha})^{J^{\alpha}}).
    \end{align}
    Now using left invariance of $\mu_{\XA_d}$ and Theorem \ref{thm: Adeles}(7), we get that $$\mu_{\XA_d}(a_{\eta_0}^{-1}(E \cap (\seda)_{\geq  n \alpha})^{J^{\alpha}})= m_{\R^{k+r-1}}(J^{\alpha}) \mu_{\seda}( E \cap (\seda)_{\geq n \alpha}).$$
    
   \noindent Also, note that for all $\beta>0$, we have 
   \begin{align}
       \label{eq: nn temp 6}
        \int_{J^T} \ind_{(E \cap (\seda)_{\geq n \alpha})^{J^\alpha}}({ a_{\eta_0} a_{t}\Tilde{\Lambda}_\theta})\, dt=  m_{\R^{k+r-1}}(J^{\alpha}) \# N( a_{\eta_0}\Tilde{\Lambda}_\theta,T,\beta, E \cap (\seda)_{\geq n \alpha} ) +  O_{\alpha, \beta}(T^{k+r-2}).
   \end{align}
   { To see this, consider the map
\[
J^\alpha \times \left(E \cap (\seda)_{\geq n \alpha}\right) \to \XA_d, \quad (x, t) \mapsto a_t x.
\]
This map is injective by definition of $(\seda)_{\geq n \alpha}$. Consequently, for any $t \in \R^{k+r-1}$ and $x \in \XA_d$ such that
\[
a_t x \in (E \cap (\seda)_{\geq n \alpha})^{J^\alpha},
\]
there exists a unique decomposition $t = t' + t''$ with $t'' \in J^\alpha$ and
\[
a_{t'} x \in E \cap (\seda)_{\geq n \alpha}.
\]
Thus for any $x \in \XA_d$, we have
\begin{align}
\label{eq: nn temp 7}
    \{t \in \R^{k+r-1} : a_t x \in (E \cap (\seda)_{\geq n \alpha})^{J^\alpha} \}
    = \{t \in \R^{k+r-1}: a_t x \in E \cap (\seda)_{\geq n \alpha}\} + J^\alpha.
\end{align}
Intersecting both sides of~\eqref{eq: nn temp 7} with $J^T$ and putting $x= a_{\eta_0} \tilde{\Lambda}_\theta$, we get that
\begin{align}
    \label{eq: nn temp 10}
    \{t \in J^T : a_t a_{\eta_0} \tilde{\Lambda}_\theta \in (E \cap (\seda)_{\geq n \alpha})^{J^\alpha} \}
\end{align}
differs from the set
\begin{align}
\label{eq: nn temp 9}
    N(a_{\eta_0} \tilde{\Lambda}_\theta, T, \beta, E \cap (\seda)_{\geq n \alpha}) + J^\alpha
\end{align}
by points contained in a thickened neighbourhood of the boundary of $J^T$, the thickness of which is bounded in terms of $\alpha$ and $\beta$, that is, there exists $\delta>0$ depending only on $\alpha$ and $\beta$ such that for all $T>1$, the symmetric difference between the sets \eqref{eq: nn temp 10} and \eqref{eq: nn temp 9} is contained in the set
$$
   \left\{t \in \R^{k+r-1} : \inf_{s \in \partial (J^T)} \|t-s\|_{\infty}< \delta  \right\}.
$$}
Since the measure of this thickened boundary is $O_{\alpha, \beta}(T^{k+r-2})$, and the measure of the set in~\eqref{eq: nn temp 10} equals the left-hand side of~\eqref{eq: nn temp 6}, while the measure of the set in~\eqref{eq: nn temp 9} equals the first term on the right-hand side of~\eqref{eq: nn temp 6}---using the fact that $J^\alpha \subset [0, n\alpha]^{k+r-1}$ and that $\|t - s\|_\infty \geq 2n\alpha$ for all distinct $s, t \in N(a_{\eta_0} \tilde{\Lambda}_\theta, T, \beta, E \cap (\seda)_{\geq n\alpha})$---we obtain the estimate~\eqref{eq: nn temp 6}.

Combining~\eqref{eq: nn temp 8} and~\eqref{eq: nn temp 6}, we conclude that
\begin{align*}
    \liminf_{T \to \infty} \frac{1}{m_{\R^{k+r-1}}(J^T)} \# N(a_{\eta_0} \tilde{\Lambda}_\theta, T, \beta, E) 
    &\geq \lim_{T \to \infty} \frac{m_1 \cdots m_k n_1 \cdots n_{r-1} \cdot (k+r-1)!}{T^{k+r-1} c_{k+r-1}(n_1, \ldots, n_r)} \\
    &\quad \cdot \# N(a_{\eta_0} \tilde{\Lambda}_\theta, T, \beta, E \cap (\seda)_{\geq n \alpha}) \\
    &= \mu_{\seda}(E \cap (\seda)_{\geq n \alpha}).
\end{align*}
Since $\alpha > 0$ was arbitrary and $\mu_{\seda}(E \cap (\seda)_{\geq n \alpha}) \to \mu_{\seda}(E)$ as $\alpha \to 0$, the claim follows. The lemma then follows from Proposition~\ref{cor1 Generic}.

\end{proof}


\section{Upper Bound for Counting}
\label{sec: Upper Bound for Counting}

{In Section \ref{Reduction of Theorem}, we established an almost one-to-one correspondence between the visits of the lattice $\Tilde{\Lambda}_\theta$ to the cross-section (up to time $T$) and the set of Diophantine approximates $(p, q)$ with $|q| \leq e^T$. However, since this correspondence is not exact but almost, it yields only a lower bound on the number of Diophantine approximates. The goal of this section is to establish the complementary upper bound. More precisely, we aim to prove the following result.}

\begin{thm}
\label{thm:N^sharp(T,theta)}
    For $T>1$, define $D^\sharp (T,\theta)$ as {the} number of $\e$-approximates $(p,q)$ of $\theta$ satisfying $\|q\| \leq e^T$. For {Lebesgue almost every} $\theta \in \M_{m \times n}(\R)$ we have
    \begin{align}
        \label{eq:N^sharp(T,theta)}
        \limsup_{T \rightarrow \infty} \frac{1}{T^{k+r-1}} D^\sharp (T, \theta) \leq \frac{\e}{\zeta(d)} \frac{m_{\R^d}(B^d_1)}{(k+r-1)!} c_{k+r-1}(n_1, \ldots, n_r),
    \end{align}
    where $c_{k+r-1}(n_1, \ldots, n_r)$ is defined as in \eqref{eq: def c k r 1}.
\end{thm} \vspace{0.2in}

To establish Theorem \ref{thm:N^sharp(T,theta)}, we will need two theorems of Wang and Yu \cite{WYYK} which refine the famous theorems of Gallagher \cite{Gallagher} and Schmidt \cite{S60} on almost sure asymptotic counts of solutions to Diophantine inequalities. We will need Gallagher's property P as introduced in \cite{Gallagher}. A subset $A$ of $\R^m$ is said to satisfy property P if 
\begin{align}
   \label{Property P}
     (x_1, \ldots, x_m) \in A \text{ implies that } (x_1', \ldots, x_m')\in A \text{ for any } 0 \leq x_i' \leq x_i (1\leq i \leq m).
   \end{align}

\begin{thm}[{\cite[Thm.~1]{WYYK}}]
\label{thmWYYK}
   Let $\delta> 0$ be arbitrary. Let $(A_q)_{q \in \N}$ be a sequence of measurable subsets of $[0,1)^m$, satisfying property P. Further assume that $m_{\R^m}(A_q)$ is a decreasing function of $q$, where $m_{\R^m}$ is Lebesgue measure on $\R^m$ giving measure $1$ to the  set $[0,1)^m$. Then, for almost every $\theta= (\theta_1, \ldots, \theta_m)$, we have
   \begin{align*}
       Num(T,\theta, (A_q))= &\sum_{q=1}^T m_{\R^m}(A_q) \\ 
       &+ O\left( \left(\sum_{q=1}^T m_{\R^m}(A_q) \right)^{1/2} \left(\sum_{q=1}^T m_{\R^m}(A_q) q^{-1} \right)^{1/2} 
       \left(\log \left(\sum_{q=1}^T m_{\R^m}(A_q)\right)\right)^{2+\delta} \right),
   \end{align*}
   for all $T \in \N$, where $Num(T, \theta, (A_q))$ denote the number of integers satisfying $1\leq q \leq T$ and $ \{q\theta \} \in A_q.$
   \end{thm}
   Here for $x \in \R^m$, we denote by $\{x\}$ the vector in $\R^m$ whose $i$-th coordinate is the fractional part of $i$-th coordinate of $x$.

\begin{thm}[{\cite[Thm.~2]{WYYK}}]
\label{thm2WYYK}
    Let $\delta> 0$ be arbitrary. Let $\{A_q\}_{q \in \Z^n}$ be a sequence of measurable subsets of $[0,1)^m$, satisfying Property P. Define {$d: \Z^n \setminus \{0\} \rightarrow \N$} by
    $$
    d(q)= \sum_{\substack{s|q_i \\ 1 \leq i \leq n}} 1.
    $$
 Suppose $n>1$. Then for almost every $\theta \in \M_{m \times n}(\R)$, we have
    {\begin{equation*}
        Num(T, \theta, (A_q))= \sum_{\|q\|_{\el} \leq T} m_{\R^m}(A_q) + O\left(\left( \sum_{\|q\|_\infty \leq T} m_{\R^m}(A_q) d(q)  \right)^{1/2} \left(\log \left( \sum_{\|q\|_\infty \leq T} m_{\R^n}(A_q) d(q)    \right)  \right)^{\tfrac{3}{2} + \delta}\right)
    \end{equation*}
   for all $T \in \N$, where $Num(T, \theta, (A_q))$ denote the number of {integral vectors} satisfying $ \|q\|_{\el} \leq T$ and $\{\theta q\} \in A_q$.}
\end{thm}

{ \begin{defn}
    Given $\theta = (\theta_{ij}) \in \M_{m \times n}(\R)$, we call a vector $ (p,q ) \in \Z^{m} \times \Z^n$ an $\e^*$-approximation if it satisfies the following conditions:
\begin{align*}
    \|\rho_i(p+ \theta q)\| \leq \eta_i \text{ for all } i=1, \ldots, k,\\
    \left( \prod_{i=1}^k \|\rho_i(p+\theta q)\|^{m_i} \right). \left(\prod_{j=1}^r  \|\rho_j'(q)\|^{n_j} \right) \leq \e, 
\end{align*}
and $\|\rho_j'(q)\| \neq 0$ for all $j=1, \ldots, r$.
\end{defn}}

{We will use the Theorems \ref{thmWYYK} and \ref{thm2WYYK} in the form of the following proposition to prove Theorem \ref{thm:N^sharp(T,theta)}.}

\begin{prop}
\label{thm:N(T, theta)}
     For $T>1$, $s \in \N$ and $\theta \in \M_{m \times n}(\R)$, define $D(T,\theta,s)$ as number of $\e^*$-approximates $(p,q) \in \Z^m \times \Z^n$ satisfying $\|q\| \leq e^T$ and $s$ divides $\gcd(p,q)$. Then for almost every $\theta \in \M_{m \times n}(\R)$, we have
    \begin{align}
        \label{eq:N(T,theta)}
        \lim_{T \rightarrow \infty} \frac{1}{T^{k+r-1}} D(T, \theta,s) = \frac{\e}{s^d} \frac{1}{(k+r-1)!} m_{\R^{d}}(B^{d}_1)   c_{k+r-1}(n_1, \ldots, n_r),
    \end{align}
   where $c_{k+r-1}(n_1, \ldots, n_r)$ is defined as in \eqref{eq: def c k r 1}. 
\end{prop}
\begin{proof}
{\bf Case I} For all $i$, $\eta_i>0$ is small enough so that $\|x\|< \eta_i$ implies that $\|x\|_{\el} < 1$ for all $x \in \R^{m_i}$.

Define the sets $(A_q)_{q \in \N}$ as 
    \begin{align}
    \label{eq: def Aq}
        A_q= \left\{x \in (B_{\eta_1}^{m_1} \times \cdots \times B_{\eta_k}^{m_k})\cap (\R_{\geq 0})^m  : \left( \prod_{i=1}^k \|\rho_i(x)\|^{m_i} \right). \left(\prod_{j=1}^r \|\rho_j'(q)\|^{n_j} \right) \leq \e \right\} \subset [0,1)^m
    \end{align}
    if $\rho_j'(q) \neq 0$ for all $1 \leq j \leq r$ and set $A_q= \emptyset$ otherwise.

    Since any two norms on $\R^n$ are equivalent, we have constants $\alpha, \beta$ such that
    $$
    \alpha \|x\| \leq \|x\|_\el \leq \beta \|x\| \text{ for all } x \in \R^n.
    $$
    Note that for {Lebesgue almost every} $\theta \in \M_{m \times n}(\R)$, we have
    \begin{align}
        \label{eq: ineq count}
         \sum_{B} Num \left( \left\lfloor \frac{\alpha e^T}{s}  \right\rfloor, B \theta, \frac{1}{s}A_{sq} \right)  \leq D(T, \theta, s)  \leq \sum_{B} Num\left( \left\lfloor \frac{\beta e^T}{s}  \right\rfloor, B \theta, \frac{1}{s}A_{sq} \right),
    \end{align}
    where sum is taken over all diagonal matrices $B$ with entries $\pm 1$ and 
    $$
    \frac{1}{s}A_{sq} =  \left\{ \frac{x}{s}: x\in A_{sq} \right\}.
    $$
    {The inequality \eqref{eq: ineq count} may fail only for those \(\theta\) admitting some \(q\in\Z^n\) with an entry of \(\{\theta q\}\) equal to zero; such \(\theta\) form a Lebesgue-null set.}
    
    We will now use {Theorems} \ref{thmWYYK} and \ref{thm2WYYK} to compute $Num(\lfloor \gamma e^T/s \rfloor,B \theta, \tfrac{1}{s}A_{sq}  )$ for $\gamma = \alpha, \beta$. But first, let us evaluate $m_{\R^m}(\frac{1}{s} A_{sq})$.

    \medskip
    
    Note that $m_{\R^m}(\frac{1}{s} A_{sq}) =0$ if $\rho_j'(q) = 0$ for some $1 \leq j \leq r$, otherwise we have
    \begin{align}
    \label{eq; measure of Aq}
           m_{\R^m}\left(\frac{1}{s}A_{sq} \right)= \frac{\e}{s^d} \frac{1}{2^m}m_{\R^m}(B^m_1) \sum_{i=0}^{k-1}   \frac{1}{i!} \frac{(n\log(s) +\log( \|\rho_1'(q)\|^{n_1} \cdots \|\rho_r'(q)\|^{n_r}) -\log( \e  \eta_1^{-m_1} \cdots \eta_k^{-m_k} ))^i}{  \|\rho_1'(q)\|^{n_1} \cdots \|\rho_r'(q)\|^{n_r} },
    \end{align}
    for all large $q$ and for {finitely many small values} of $q$, we may have 
    \begin{align}
        \label{eq:measure of Aq 2}
        m_{\R^m}\left(\frac{1}{s}A_{sq} \right)= \frac{1}{(2s)^m}m_{\R^m}(B^m_1) \eta_1^{m_1} \cdots \eta_k^{m_k}.
    \end{align}

{ \noindent To see this, note that if \( q \) is small, then it is possible that
\[
\|\rho_i(x)\| \leq \eta_i \text{ for all } i
\quad \implies \quad
\left( \prod_{i=1}^k \|\rho_i(x)\|^{m_i} \right) \cdot
\left( \prod_{j=1}^r \|\rho_j'(sq)\|^{n_j} \right)
\leq \e,
\]
in which case \( A_{sq} = (B_{\eta_1}^{m_1} \times \cdots \times B_{\eta_k}^{m_k}) \cap (\R_{\geq 0})^m \), and hence \eqref{eq:measure of Aq 2} follows.

Otherwise, we push forward the set \( A_{sq} \) under the map
\begin{align}
\label{eq: measure of Aq 3}
(\R_{\geq 0})^m &\longrightarrow \R^k \times \Sphere^{m_1} \times \cdots \times \Sphere^{m_k}, \\
x &\longmapsto \left( \left( \|\rho_1(x)\|^{m_1}, \ldots, \|\rho_k(x)\|^{m_k} \right), \left( \frac{\rho_1(x)}{\|\rho_1(x)\|}, \ldots, \frac{\rho_k(x)}{\|\rho_k(x)\|} \right) \right), \nonumber
\end{align}
which yields the following expression for the measure:
\begin{align*}
&m_{\R^k}\left( \left\{ (x_1, \ldots, x_k) :
\begin{array}{l}
x_i \in [0, \eta_i] \text{ for all } i, \\
x_1 \cdots x_k \leq \dfrac{\e}{s^n \prod_{j=1}^r \|\rho_j'(q)\|^{n_j}}
\end{array}
\right\} \right)
 \times \left( \frac{1}{2^{m_1}} \mu^{(\Sphere^{m_1})}(\Sphere^{m_1}) \right) \cdots \left( \frac{1}{2^{m_k}} \mu^{(\Sphere^{m_k})}(\Sphere^{m_k}) \right),
\end{align*}
which directly implies \eqref{eq; measure of Aq} after explicitly computing the first term and noting that
\[
m_{\R^m}(B^m_1) = \mu^{(\Sphere^{m_1})}(\Sphere^{m_1}) \cdots \mu^{(\Sphere^{m_k})}(\Sphere^{m_k}).
\]}

\medskip

    Thus for $\gamma= \alpha, \beta$ we have for all large enough $T$: {
    \begin{align}
       & {\sum_{ \|q\|_\el \leq \lfloor \gamma e^T/s \rfloor }} m_{\R^m}\left(\frac{1}{s}A_{sq} \right)  \nonumber\\
       &= \frac{\e}{s^d} \frac{1}{2^m}m_{\R^m}(B^m_1)   \sum_{ \substack{1 \leq  \|\rho_i'(q)\|_\el \leq \lfloor \gamma e^T/s \rfloor \\ \text{for all } 1\leq i \leq r} } \frac{1}{(k-1)!}  \frac{(\log( \|\rho_1'(q)\|^{n_1} \cdots \|\rho_r'(q)\|^{n_r}) )^{k-1}}{  \|\rho_1'(q)\|^{n_1} \cdots \|\rho_r'(q)\|^{n_r} } \nonumber\\
       &\quad + O\left( \sum_{ \substack{1 \leq  \|\rho_i'(q)\|_\el \leq \lfloor \gamma e^T/s \rfloor \\ \text{for all } 1\leq i \leq r} } \frac{(\log( \|\rho_1'(q)\|^{n_1} \cdots \|\rho_r'(q)\|^{n_r}) )^{k-2}}{  \|\rho_1'(q)\|^{n_1} \cdots \|\rho_r'(q)\|^{n_r} } \right), \label{eq:temp4}
       \end{align}
      where implicit constant depends on $\e,d,n, \eta_1, \ldots, \eta_k$.}
    Note that
    \begin{align}
        &\frac{1}{(k-1)!}\sum_{ \substack{1 \leq  \|\rho_i'(q)\|_\el \leq \lfloor \gamma e^T/s \rfloor \\ \text{for all } 1\leq i \leq r} }  \frac{(\log( \|\rho_1'(q)\|^{n_1} \cdots \|\rho_r'(q)\|^{n_r}) )^{k-1}}{  \|\rho_1'(q)\|^{n_1} \cdots \|\rho_r'(q)\|^{n_r} } \nonumber \\
        &=  \sum_{\substack{j_1, \ldots, j_{r} \\ j_1 + \cdots + j_{r}=k-1}}  \frac{1}{j_1! \cdots j_r!}  \prod_{i=1}^r \left( \sum_{1 \leq  \|\rho_i'(q)\|_\el \leq \lfloor \gamma e^T/s \rfloor} \frac{(\log \|\rho_i'(q)\|^{n_i})^{j_i}}{\|\rho_i'(q)\|^{n_i}} \right). \label{eq:temp5}
    \end{align}
    Also, note that for all $1 \leq i \leq r$ and $j \in \N$, we have
    {\begin{align}
        \sum_{1 \leq \|\rho_i'(q)\|_\el \leq \lfloor \gamma e^T/s \rfloor}  \frac{(\log \|\rho_i'(q)\|^{n_r})^{j}}{\|\rho_i'(q)\|^{n_i}}  &=  \int_{\substack{x \in \R^{n_i} \\ 1 \leq  \|x\|_\infty \leq \lfloor \gamma e^T/s \rfloor}} \frac{(\log \|x\|^{n_i})^{j}}{\|x\|^{n_i}} \, dx + O(T^j) \nonumber \\
        &= \mu^{(\Sphere^{n_i})}(\Sphere^{n_i}) \frac{n_i^{j+1}}{j+1}  T^{j+1} + O(T^j). \label{eq:temp6}
    \end{align}
    Therefore, using equations \eqref{eq:temp5}, \eqref{eq:temp6} and the fact that {$m_{\R^n}(B^n_1) = \mu^{(\Sphere^{n_1})}(\Sphere^{n_1}) \cdots \mu^{(\Sphere^{n_r})}(\Sphere^{n_r})$}, we get that for $\gamma= \alpha, \beta$
    \begin{align}
        &\frac{1}{(k-1)!}\sum_{ \substack{1 \leq  \|\rho_i'(q)\|_\el \leq \lfloor \gamma e^T/s \rfloor \\ \text{for all } 1\leq i \leq r} }  \frac{(\log( \|\rho_1'(q)\|^{n_1} \cdots \|\rho_r'(q)\|^{n_r}) )^{k-1}}{  \|\rho_1'(q)\|^{n_1} \cdots \|\rho_r'(q)\|^{n_r} } \nonumber \\
        &= { \sum_{\substack{j_1, \ldots, j_{r} \\ j_1 + \cdots + j_{r}=k-1}}  \frac{1}{j_1! \cdots j_r!}  \prod_{i=1}^r \left( \mu^{(\Sphere^{n_i})}(\Sphere^{n_i}) \frac{n_i^{j_i+1}}{j_i+1}  T^{j_i+1} + O(T^{j_i}) \right) }\nonumber \\
        &= m_{\R^n}(B^n_1) \frac{ T^{k+r-1}}{(k+r-1)!}   \sum_{\substack{j_1, \ldots, j_{r} \\ j_1 + \cdots + j_{r}=k-1}}  \frac{(k+r-1)!}{(j_1+1)! \cdots (j_r+1)!} (n_1^{j_1 +1} \cdots n_r^{j_r+ 1}) +O(T^{k+r-2}) \nonumber \\
        &= m_{\R^n}(B^n_1) \frac{ T^{k+r-1}}{(k+r-1)!}  c_{k+r-1}(n_1, \ldots, n_r).   \label{eq:n temp 1}
    \end{align}
    Similarly, we have
    \begin{align}   \label{eq:n temp 2}
        \sum_{ \substack{1 \leq  \|\rho_i'(q)\|_\el \leq \lfloor \gamma e^T/s \rfloor \\ \text{for all } 1\leq i \leq r} } \frac{(\log( \|\rho_1'(q)\|^{n_1} \cdots \|\rho_r'(q)\|^{n_r}) )^{k-2}}{  \|\rho_1'(q)\|^{n_1} \cdots \|\rho_r'(q)\|^{n_r}} = O(T^{k+r-2}).  
    \end{align}}
    \noindent Using equations \eqref{eq:temp4}, \eqref{eq:n temp 1} and \eqref{eq:n temp 2} and the fact that {$m_{\R^d}(B^d_1) = m_{\R^m}(B^m_1) m_{\R^n}(B^n_1)$}, we get that for $\gamma= \alpha, \beta$
    \begin{align*}
        \lim_{T \rightarrow \infty}\frac{1}{T^{k+r-1}} \sum_{\|q\|_\infty \leq \lfloor \gamma e^T/s \rfloor} m_{\R^m}\left(\frac{1}{s}A_{sq} \right)= \frac{m_{\R^d}(B^d_1)}{2^m} \frac{\e}{s^d}  \frac{1}{(k+r-1)!} c_{k+r-1}(n_1, \ldots, n_r).
    \end{align*}
Now, it is easy to see that \begin{tiny}
    \begin{align*}
        &\lim_{T \rightarrow \infty} \frac{1}{T^{k+r-1}} \left( \sum_{ 1 \leq q \leq \lfloor \gamma e^T/s \rfloor} m_{\R^m}\left(\frac{1}{s}A_{sq} \right)  \right)^{\frac{1}{2}} \left( \sum_{ 1 \leq q \leq \lfloor \gamma e^T/s \rfloor} m_{\R^m}\left(\frac{1}{s}A_{sq} \right)q^{-1}  \right)^{\frac{1}{2}}  (\log ( \sum_{\|q\| \leq \lfloor \gamma e^T/s \rfloor} m_{\R^n}\left(\frac{1}{s}A_{sq} \right)) )^{2 + \delta} = 0 \text{ if } n=1 \\
         &\lim_{T \rightarrow \infty} \frac{1}{T^{k+r-1}} \left( \sum_{\|q\|_\infty \leq \lfloor\gamma e^T/s \rfloor} m_{\R^m}\left(\frac{1}{s}A_{sq} \right) d(q)  \right)^{1/2} (\log ( \sum_{\|q\|_\infty \leq \lfloor \gamma e^T/s \rfloor} m_{\R^n}\left(\frac{1}{s}A_{sq} \right) d(q)))^{3/2 + \delta}= 0 \text{ if } n \neq 1.
    \end{align*}
    \end{tiny}
    The result now follows from Theorem \ref{thmWYYK} and \ref{thm2WYYK}.

{\bf Case II} The constants $\eta_i$ {does not} satisfy the condition of Case I. In this case, choose $\eta_i' >0$ for all $1 \leq i \leq r$ such that {$\|x\|< \eta_i'$} implies that $\|x\|_{\el} < 1$ for all $x \in \R^{m_i}$. Note that the set of all $\e^*$-approximates of $\theta$ corresponding to $(\eta_1, \ldots, \eta_k)$ satisfying the conditions of proposition, which are not $\e^*$-approximates of $\theta$ corresponding to $(\eta_1', \ldots, \eta_k')$ satisfying the conditions of proposition is contained in the set {$\cup_I Z_{I}$} where $I$ vary over all non-empty subsets of $\{1, \ldots, k\}$. { Here the set $Z_{I}$ is defined as set of all $(p,q) \in \Z^m \times \Z^n$ such that $\|q\| \leq e^T$, $s$~divides~$\gcd(p,q)$ and
\begin{align}
    \rho_j'(q) &\neq 0 \text{ for all } j=1, \ldots, r,  \label{eq: nn temp 0} \\
   \|\rho_i(p+ \theta q)\| &\leq \eta_i' \text{ for all } i \in \{1, \ldots, k\} \setminus I, \label{eq: nn temp 1} \\
  \eta_i'< \|\rho_i(p+\theta q)\| &\leq \eta_i  \text{ for all } i \in I, \label{eq: nn temp 2}\\
    \left( \prod_{i \in \{1, \ldots, k\}} \|\rho_i(p+\theta q)\|^{m_i} \right).& \left(\prod_{j=1}^r  \|\rho_j'(q)\|^{n_j} \right) \leq { \e}, \label{eq: nn temp 3}
\end{align}
which implies that
\begin{align}
     \left( \prod_{i \in \{1, \ldots, k\} \setminus I} \|\rho_i(p+\theta q)\|^{m_i} \right). \left(\prod_{j=1}^r  \|\rho_j'(q)\|^{n_j} \right) \leq \frac{\e}{\prod_{i \in I} (\eta_i')^{m_i} }. \label{eq: nn temp 4}
\end{align}

Note that the number of all \( (p,q) \in \Z^m \times \Z^n \) satisfying \eqref{eq: nn temp 0}, \eqref{eq: nn temp 1}, \eqref{eq: nn temp 2}, and \eqref{eq: nn temp 4} is bounded by a constant multiple of the number of 
\[
\left( (\rho_i(p))_{i \notin I}, q \right) \in \Z^{m - \sum_{i \in I} m_i} \times \Z^n
\]
satisfying \eqref{eq: nn temp 0}, \eqref{eq: nn temp 1}, and \eqref{eq: nn temp 4}. This follows because, for each \( i \in I \), the coordinate \( \rho_i(p) \) is determined up to finitely many choices by \eqref{eq: nn temp 2}, for each fixed value of \( q \in \Z^n \). Moreover, this finite number of choices can be uniformly bounded above by a constant depending only on \( \eta_i \), \( \eta_i' \), and the choice of norm on \( \R^{m_i} \).

Thus, by Case~I, we conclude that the cardinality of \( Z_I \) is of order \( O(T^{k + r - 1 - |I|}) \). Hence, the proposition follows.}
\end{proof}

\begin{proof}[Proof of Theorem \ref{thm:N^sharp(T,theta)}]
   
   Using Proposition~\ref{thm:N(T, theta)}, it is easy to see that for Lebesgue almost every $\theta \in \M_{m \times n}(\R)$, the number of $\e$-approximates $(p,q)$ of $\theta$ such that $\|\rho_j'(q)\|=0$ for some $j$ are of order $O(T^{k+r-2})$, since they are $\e^*$-approximate of matrices obtained by removing columns from $\theta$. Hence for the proof, we may assume that $D^\sharp(T, \theta)$ counts the number of $\e$-approximates $(p,q)$ of $\theta$ such that $\|\rho_j'(q)\| \neq 0$ for all $j$.

    Let $s_1, s_2, \ldots$ denote the increasing sequence of all prime numbers. Fix an $l \in \N$. It is clear that $D^\sharp(T, \theta)$ is bounded above by $\e^*$-approximates $(p,q)$ of $\theta$ such that $\|q\| \leq e^T$ and $s_i \nmid  \gcd(p,q)$ for all $i= 1, \ldots ,l$. Thus we have
    \begin{align*}
        D^\sharp(T, \theta) \leq D(T,\theta,1) - \sum_{i=1}^l D(T,\theta,p_i) + \sum_{\substack{ i,j=1 \\ i \neq j}}^l  D(T,\theta,p_ip_j) - \cdots + (-1)^l D(T, \theta, p_1 \cdots p_l).
    \end{align*}
    Using Proposition~\ref{thm:N(T, theta)}, we get that
    {\begin{align*}
        \limsup_{T \rightarrow \infty} \frac{1}{T^{k+r-1}} D^\sharp(T, \theta) \leq \left( \prod_{i=1}^l \left(1 - \frac{1}{p_i^d} \right)\right)\frac{m_{\R^d}(B^d_1)\delta}{(k+r-1)!} c_{k+r-1}(n_1, \ldots, n_r). 
    \end{align*}}
    The theorem now follows by taking the limit as $l \rightarrow \infty$.
\end{proof}


\section{Proof of Theorem \ref{main thm}}
\label{sec: Proof of main Theorem}

\begin{proof}
    Assume that $\theta \in \M_{m \times n}(\R)$ is chosen so that Lemma \ref{lem: Important Commutativity} and \ref{mainmult3} {hold}. Then for any {$\mu_{\seda}$-JM} set $E \subset \seda$, we have
    \begin{align}
    \label{eq: sandwich 1}
       & \liminf_{T \rightarrow \infty} \frac{1}{m_{\R^{k+r-1}}(J^T)} \sum_{t \in N(a_{\eta_0} \tilde \Lambda_\theta, T, \log \e - \log (\eta_1^{m_1}\cdots \eta_k^{m_k}), \seda ) + \eta_0 } \ind_E(a_t \tilde \Lambda_\theta)  \geq \mu_{\seda}(E).
    \end{align}
    Note that by Lemma \ref{Number ineq}, we know that 
    \begin{align*}
        N(\tilde \Lambda_\theta, T, \log \e, \seda ) \cap ((-\log \eta_1, \infty) \times \cdots \times (-\log \eta_k, \infty) \times \R^{r-1}) &\subset \Y_T(\theta).
    \end{align*}
    Thus, we have 
    \begin{align}
    \label{eq: inequality}
         N(a_{\eta_0} \tilde \Lambda_\theta, T, \log \e - \log (\eta_1^{m_1}\cdots \eta_k^{m_k}), \seda ) + \eta_0   &\subset \Y_T(\theta). 
    \end{align} 
    Note that by \eqref{eq: sandwich 1} and \eqref{eq: inequality} we have
    \begin{align*}
        &\mu_{\seda}(E) + \mu_{\seda}(\seda \setminus E) \\
        &\leq \liminf_{T \rightarrow \infty} \frac{1}{m_{\R^{k+r-1}}(J^T)} \sum_{t \in \Y_T(\theta) } \ind_E(a_t \tilde \Lambda_\theta) + \liminf_{T \rightarrow \infty} \frac{1}{m_{\R^{k+r-1}}(J^T)} \sum_{t \in \Y_T(\theta) } \ind_{\seda \setminus E}(a_t \tilde \Lambda_\theta) \\
        &\leq \limsup_{T \rightarrow \infty} \frac{1}{m_{\R^{k+r-1}}(J^T)}  \sum_{t \in \Y_T(\theta) } 1 \\
        &\leq \frac{1}{2} \limsup_{T \rightarrow \infty}  \frac{m_1 \cdots m_kn_1 \cdots n_{r-1}. (k+r-1)!}{T^{k+r-1}c_{k+r-1}(n_1, \ldots, n_r)} D^\sharp(T, \theta) \text{  by Lemma \ref{lem: Important Commutativity}} \\
        &\leq \frac{1}{2} \frac{m_{\R^d}(B^d_1) . \e}{\zeta(d)} m_1 \cdots n_{r-1}  \text{ by Theorem \ref{thm:N^sharp(T,theta)}}\\
        &= \mu_{\seda}(\seda) \text{ by Proposition  \ref{prop:pushforward general}}\\
        &= \mu_{\seda}(E) + \mu_{\seda}(\seda \setminus E).
    \end{align*}
    Thus equality holds in each of the above steps, and we get that for any $\mu_{\seda}$-JM set $E \subset \seda$
    \begin{align}
        \label{eq: sandwich 3}
        \lim_{T \rightarrow \infty} \frac{m_1 \cdots  n_{r-1}. (k+r-1)!}{T^{k+r-1}c_{k+r-1}(n_1, \ldots, n_r)} \sum_{t \in \Y_T(\theta) } \ind_E(a_t \tilde \Lambda_\theta)  = \mu_{\seda}(E).
    \end{align}
Since, this holds for all $\mu_{\seda}$-JM set $E \subset \seda$, we have by Theorem \ref{muJM} that
\begin{align}
    \label{eq:a2}
    \frac{m_1 \cdots n_{r-1}. (k+r-1)!}{T^{k+r-1}c_{k+r-1}(n_1, \ldots, n_r)}  \sum_{t \in \Y_T(\theta)} \delta_{a_t \tilde \Lambda_\theta}  \longrightarrow \mu_{\seda},
\end{align}
as $T \rightarrow \infty$ tightly. Since, $\tilde \psi$ and $\tilde \psi^-$ are continuous $\mu_{\seda}$-almost everywhere (on $\sseda$), hence by Theorem \ref{muJM}, we have
\begin{align}
\label{eq:a3}
     \frac{m_1 \cdots  n_{r-1} . (k+r-1)!}{T^{k+r-1}c_{k+r-1}(n_1, \ldots, n_r)}  \sum_{t \in \Y_T(\theta)} \left( (\tilde \psi)_* (\delta_{a_t \tilde \Lambda_\theta}) +  (\tilde \psi^-)_* (\delta_{a_t \tilde \Lambda_\theta})  \right) \longrightarrow (\tilde \psi)_* \mu_{\seda}+   (\tilde \psi^-)_* \mu_{\seda}.
\end{align}
We know that by Lemma \ref{lem: Important Commutativity} that if $t$ corresponds to $\e$-approximates $(p,q)$ and $(-p,-q)$ such that the last coordinate of $q$ is positive, then we have 
\begin{align}
\label{eq:a4}
    (\tilde \psi)_* (\delta_{a_{t} \tilde \Lambda_\theta})&= \delta_{( \lambda_d(\theta, p,q),\proj(\theta,p,q), \disp(\theta,p,q) ,( p,q))},\\
    \label{eq:a4'}
    (\tilde \psi^-)_* (\delta_{(a_{t} \tilde \Lambda_\theta)}) &= \delta_{( \lambda_d(\theta, -p,-q),\proj(\theta,-p,-q), \disp(\theta,-p,-q) , (-p,-q))}.
\end{align}
Also, we have by {Proposition~\ref{prop:pushforward general} that 
\begin{align}
    \nonumber
    &(\tilde \psi)_* \mu_{\seda}+   (\tilde \psi^-)_* \mu_{\seda} \\
    &= {m_1 \cdots n_{r-1}} \frac{1}{\zeta(d)} (m_{\E_d}  \times \mu^{(\Sphere^{m_1})} \times \cdots \times \mu^{(\Sphere^{n_r})} \times m_{\R}|_{(0,\e)} \times m_{\hZp^d}) . \label{eq:a5}
\end{align}}
Combining \eqref{eq:a3}, \eqref{eq:a4}, \eqref{eq:a4'} and \eqref{eq:a5}, we get the desired result.
\end{proof}


\section{Time Visits}
\label{sec: Time visits}
In this section and Section \ref{section: Time visit proof}, we will assume that $k=r=1$, i.e., the flow $(a_t)$ is a one-parameter flow. In this case, given a subset $A \subset \seda$, we define the following two functions on the cross-section $\seda$.
\begin{enumerate}
    \item The {\em return time function} $\tau_A: A \rightarrow \R_{> 0}$ defined as 
    $$
    \tau_A(x)= \min\{t \in \R_{> 0}: a_tx \in A\}.
    $$
    The minimum exists for $\mu_{\seda}|_A$-a.e. $x$, since the set $\{t \in \R: t \in \seda\}$ is discrete for $\mu_{\seda}$-a.e. $x$. Thus, the function is defined $\mu_{\seda}|_A$-almost everywhere.
    \item The {\em first return map} $T_A: A \rightarrow A$ defined as $T_A(x)= a_{\tau_A(x)}x$. Clearly, the function is defined $\mu_{\seda}|_A$-almost everywhere.
\end{enumerate}
Note that it is not yet clear whether the maps $\tau_A$ and $T_A$ are measurable functions.

\begin{lem}
\label{lem: Continuity of tau T}
    For a $\mu_{\seda}$-JM subset of $\seda$ with $\mu_{\seda}(A)>0$, we have that the maps $\tau_A, T_A$ are continuous almost everywhere ({with respect to} measure $\mu_{\seda}|_A$).
\end{lem}
\begin{rem}
\label{rem:continuity of tau implies T}
    Note that it is enough to prove that $\tau_A$ is continuous almost everywhere ({with respect to} measure $\mu_{\seda}|_A$). To see this, note that $T_A$ equals composition of maps $A \rightarrow \R \times A \rightarrow A$, defined as $x \mapsto (\tau_{A}(x), x ) \mapsto a_{\tau_{A}(x)}x$. Using continuity of the map $\R \times A \rightarrow A$, $(t,x) \mapsto a_tx$, we get that $T_A$ is continuous at any point where $\tau_A$ is continuous. Hence, the claim follows.
\end{rem}

Let us first prove the special case of Lemma \ref{lem: Continuity of tau T} for $A= \seda$.

\begin{lem}
\label{lem: continuity of tau T special case}
    The map $\tau_{\seda}$ is continuous almost everywhere ({with respect to} measure $\mu_{\seda}$).
\end{lem}
\begin{proof}
     We will prove that $\tau_{\seda}$ is continuous at all points $\Tilde{\Lambda}$ such that {$\Lambda = \pi(\Tilde{\Lambda})$ does not contain} a point in $\cup_{t \in \R} a_t(\Le \setminus \Le^\circ)$ (defined as in \eqref{eq: def L zero}) and $\Lambda \in \ssed$. Since the latter is a set of full measure in $\seda$ (using Theorem \ref{thm: Adeles}, Lemma \ref{SW22Lem8.8}), the lemma will follow from above statement. To prove this, fix such a $\Tilde{\Lambda}$. Suppose {$\tau_{\seda}(\Tilde{\Lambda}) = l$. Fix $0<\delta <l$ .} 
     
     Note that $x \in \ueda$, hence by Theorem \ref{thm: Adeles} (11) and Proof of Proposition \ref{measure equal general}, there exists a $\gamma> 0$ and an open set $O(\Tilde{\Lambda})$ containing $\Tilde{\Lambda}$ such that the map $[0,\gamma] \times O(\Tilde{\Lambda}) \rightarrow \XA_d$ given by $(t,x) \mapsto a_tx$ is open and injective. Note that since $\seda$ is closed, by shrinking $O(\Tilde{\Lambda})$ and taking smaller $\gamma$, we can assume that $O(\Tilde{\Lambda})^{(0,\gamma)} \cap \seda = \emptyset$. 

     {Consider the set $U= O(\Tilde{\Lambda}) \cap \pi^{-1}(\X_d(\cup_{t \in (l-\delta , l + \delta)} a_{-t}L_\e^\circ ) \setminus  \X_d(\cup_{t \in [\gamma, l-\delta]} a_{-t} \Le )  )$. It clearly contains $\Tilde{\Lambda}$. Since $\cup_{t \in (l-\delta , l + \delta)} L_\e^\circ $ is open, by Lemma \ref{Sw22 Lem 8.2}, we get that $\X_d(\cup_{t \in (l-\delta , l + \delta)} a_{-t}L_\e^\circ )$ is open. Also, since $\cup_{t \in [\gamma, l-\delta]} a_{-t} \Le $ is closed, so by Lemma \ref{SW22 Lem 8.1}, the set $\X_d(\cup_{t \in [\gamma, l-\delta]} a_{-t} \Le )$ is closed. Thus, the set $U$ is open in $\seda$. Note that for any $x \in U$, we have $\tau(x ) \in (l-\delta, l+\delta)$. This follows from the following facts: Firstly, $\tau_{\seda}(x) \geq \gamma$ follows from fact that $x \in O(\Tilde{\Lambda})$ and that $O(\Tilde{\Lambda})^{(0,\gamma)} \cap \seda = \emptyset$. The fact that $\pi(x) \notin \X_d(\cup_{t \in [\gamma, l-\delta]} a_{-t} \Le )$ implies that $\tau_{\seda}(x) > l-\delta$. The fact that $\pi(x) \in \X_d(\cup_{t \in (l-\delta , l + \delta)} a_{-t}L_\e^\circ )$ implies that $\tau_{\seda}(x) < l+\delta$. Thus, corresponding to the given $\delta$, the set $U$ works. Hence, the lemma follows. }
\end{proof}

\begin{lem}
\label{lem: special case continuous T}
    The map $T_{\seda}:\seda \rightarrow \seda$ is continuous almost everywhere ({with respect to} the measure $\mu_{\seda}$) and preserves $\mu_{\seda}$.
\end{lem}
\begin{proof}
    By {Remark} \ref{rem:continuity of tau implies T} and Lemma \ref{lem: continuity of tau T special case}, we get that the map $T_{\seda}:\seda \rightarrow \seda$ is continuous almost everywhere ({with respect to} the measure $\mu_{\seda}$). The fact that $T_{\seda}$ preserves $\mu_{\seda}$ follows from the structure theory of cross-sections for a single parameter flow built by Ambrose and Kakutani (\cite{Ambrose}, \cite{Ambrose_Kakutani}). This theory is beautifully summarised in \cite[Thm.~4.4]{SW22}. In particular, this can be easily observed from the definition of $\mu_{\seda}$ and the third part of \cite[Thm.~4.4]{SW22}.  
\end{proof}

\begin{proof}[Proof of Lemma \ref{lem: Continuity of tau T}]
    Let $Y$ denote the co-null subset of $\seda$ such that both maps $\tau_{\seda}, T_{\seda}$ are continuous at all points of $Y$. Fix a {$\mu_{\seda}$}-JM subset $A$ of $\seda$ such that $\mu_{\seda}(A)>0$. { We claim that the map $\tau_A$ is continuous for $x \in A$ contained in the set
    \begin{align}
        \label{eq: nn temp 11}
       \left(\bigcap_{i \in \Z_{\geq 0}}T^{-i}_{\seda}(Y) \right) \setminus \left(\bigcup_{j \in \Z_{\geq 0}}T^{-j}_{\seda}(\partial_{\seda} A) \right).
    \end{align}
    Note that the set in \eqref{eq: nn temp 11} is a co-null subset of $\seda$, hence the lemma follows from the above claim and Remark \ref{rem:continuity of tau implies T}.

   To prove the claim, fix $x \in A$ contained in the set \eqref{eq: nn temp 11}. Suppose $\tau_A(x)=l$. Assume that $j \in \N$ is such that $\tau_A(x)= \tau_{\seda}(x)+ \cdots + \tau_{\seda}(T_{\seda}^{j-1}(x))$.} Let $\delta>0$ be given. Since all the points $x, T_{\seda}(x), \ldots, T_{\seda}^{j-1}(x)$ belong to $Y$, there are neighbourhoods $U_i$ of $T_{\seda}^i(x)$ for $0 \leq i \leq j-1$ such that $|\tau_{\seda}(y)-\tau_{\seda}(T_{\seda}^i(x)) | < \delta/j$ for all $y \in U_i$. Note that for all {$0 < i \leq j-1$}, we have $T_{\seda}^i(x) \notin A \cup \partial_{\seda}(A)= \cl_{\seda}(A)$. By shrinking $U_i$, we may assume that $U_i \cap \cl_{\seda}(A)= \emptyset$. By continuity of $T^i_{\seda}$ at point $x$, we can find an open set $U$ such that $T_{\seda}^i(U) \subset U_i$ for all $0 \leq i \leq j-1$ and $T^j_{\seda}(U) \subset A^\circ$. We claim that for all $y \in U$, we have $\tau_A(y) \in (l-\delta, l+\delta)$. To see this note that for all $1 \leq i \leq j-1 $, we have $T_{\seda}^i(y) \in U_i \subset \seda \setminus A$ and $T_{\seda}^j(y) \in A$. Thus, $\tau_A(y)= \tau_{\seda}(y)+ \cdots + \tau_{\seda}(T_{\seda}^{j-1}y)$. Now, the claim follows by the definition of $U_i$. Thus, corresponding to each given $\delta>0$, there exists an open set $U$ such that for all $y \in U$, $|\tau_A(y)- \tau_A(x)| < \delta$. Hence, the lemma follows.
\end{proof}

\section{Proof Of Theorem \ref{thm: cor to Main thm time visits}}
\label{section: Time visit proof}

\begin{proof}
        Fix $\theta \in \M_{m \times n}(\R)$ such that
    \begin{enumerate}
        \item {There are only finitely many $\e$-approximates $(p,q)$ of $\theta$ with $j$-th component of $(p+\theta q,q)$ zero, for $j \in \{1, \ldots, d\}$.}
        \item The following equation holds. 
        \begin{align}
            \label{eq: temp9}
            \lim_{T \rightarrow \infty} \frac{1}{T} \sum_{t \in N^*(\Tilde{\Lambda}_\theta, T, \seda)} \delta_{a_t\Tilde{\Lambda}_\theta} = \mu_{\seda},
        \end{align}
       where $$N^*(x,T,A)= \{t \in [0,T]: a_t\Tilde{\Lambda}_\theta \in A\}$$ for any Borel subset $A$ of $\seda$.
       \item There are no integer vectors $(p,q), (p',q') \in \Z^m \times \Z^n$ with $(p',q') \neq \pm (p,q)$ such that $\|p+ \theta q\|= \|p'+ \theta q'\|$.
    \end{enumerate}
    Note that all the conditions hold for a.e. $\theta \in \M_{m \times n}(\R)$. To see this, note that if $1\leq j \leq m$, then condition (1) holds for all $\theta$ such that $\theta_{11}, \theta_{12}, \ldots, \theta_{mn},1$ are linearly independent over $\Q$, which holds for a.e. $\theta$. If $ m+1 \leq j \leq d$, define $i= j-m$. In this case, the existence of infinitely many $\e$-approximates $(p,q)$ of $\theta$ with $q_i =0$ corresponds to the fact that the $m \times (n-1)$-matrix $\theta^i$ obtained by removing the $i$-th column of $\theta$, has infinitely many integral solutions $(p,q) \in \Z^m \times \Z^{n-1}$ to the inequality $\|p+ \theta^i q\|^m \|q\|^n < \e$. By Khintchine's theorem, it is easy to see that later belongs to a null set. Thus, condition (1) holds for a.e. $\theta$. Also, it is easy to see from the proof of Theorem \ref{main thm} that condition (2) holds for a.e. $\theta$. The observation that condition (3) holds for a.e. $\theta$ is easy to see. We will show that the Theorem holds for this $\theta$. This will prove the Theorem.

    Fix $1 \leq j \leq d$ and a $\tmu^{j}$-JM subset $A$ of $\ZZ_j$ satisfying \eqref{eq: con cor time 1} and \eqref{eq: con cor time 2}. Let $(p_l, q_l) \in \Z^m \times \Z^n$ be the sequence of $\e$-approximations of $\theta$, ordered according to decreasing $\|p_l+ \theta q_l\|$ satisfying 
     \begin{align}
         (\lambda_j(\theta, p_l, q_l), \proj(\theta,p_l, q_l), \disp(\theta,p_l, q_l) , (p_l, q_l)) \in A.
     \end{align}
     Let us define the sequence $t_l \in \R$ as $t_l= -\log \|p_l +\theta q_l\|$ (as in Lemma \ref{Number ineq}). Let $\Y_T^A(\theta)= \{t_l : t_l \leq T\}$. Let $\Tilde{A} = \Tilde{\psi}^{-1}( \phi_{jd}^{-1}( A))$, where $\phi_{jd}$ is defined as in Remark~\ref{rem:reduction of main thm}. Using conditions (2), (3) on $\theta$ and steps in proof of Lemma \ref{Number ineq}, it is easy to see that for all $T$, we have $N^*(\Tilde{\Lambda}_\theta, T, \tilde{A}) \subset \Y_T^A(\theta)$ and the number of such elements in $\Y_T^A(\theta) \setminus N^*(\Tilde{\Lambda}_\theta, T, \tilde{A})$ is bounded by a constant {depending on $\theta$ but independent of $T$}. In fact, the latter exactly corresponds to the number of $\e$-approximates $(p,q)$ of $\theta$ such that either $j$-th or $d$-th component of $(p+\theta q,q)$ is zero. Thus, for all large enough $l$, we have 
     \begin{align}
         \label{eq:temp11}
         \tau_A(T_A^s(a_{t_l}\Tilde{\Lambda}_\theta))= -\log\|p_{l+s+1} + \theta q_{l+s+1}\| + \log\|p_{l+s} + \theta q_{l+s}\| .
     \end{align}
     Let us denote the natural projection maps from $\E_d \times \Sphere^{m} \times \Sphere^{n} \times (0,\e) \times \hat{\Z}_\prim^d$ onto $\Sphere^m$ (resp. $(0,\e)$) by $\pi_1$ (resp. $\pi_2$). Let us define $f: \Tilde{A} \rightarrow \R^m$ as 
     $$f(x)= \pi_1(\Tilde{\psi}(x)). \left(\pi_2(\Tilde{\psi}(\tau_A^s(x))) . e^{\tau_A(x) + \cdots + \tau_A^{s-1}(x)} \right)^{1/m}. $$ 
     The function is clearly continuous $\mu_{\sed}|_{\Tilde{A}}$-almost everywhere and satisfy property that for all large enough $l$, we have
    \begin{align}
        \label{temp 13}
        f(a_{t_l}\Tilde{\Lambda}_\theta) = \|q_{l+s}\|^{n/m}(p_l + \theta q_l).
    \end{align}
     
     Note that by \eqref{eq: temp9}, we have 
     \begin{align}
            \label{eq: temp10}
            \lim_{T \rightarrow \infty} \frac{1}{T} \sum_{t \in N^*(\Tilde{\Lambda}_\theta, T, \Tilde{A})} \delta_{a_t\Tilde{\Lambda}_\theta} = \mu_{\seda}|_{\Tilde{A}}.
        \end{align}
    Using Theorem \ref{muJM}, $\mu_{\seda}$-almost everywhere continuity of $\tilde \psi$ and Lemma \ref{lem: Continuity of tau T}, we pushforward both sides of \eqref{eq: temp10} under map $f$ to get
     \begin{align}
            \label{eq: temp12}
            \lim_{T \rightarrow \infty} \frac{1}{T} \sum_{t \in N^*(\Tilde{\Lambda}_\theta, T, \Tilde{A})} f_*\delta_{a_t\Tilde{\Lambda}_\theta} = f_*\mu_{\seda}|_{\Tilde{A}}.
        \end{align}
        Now applying Lemma \ref{lem: Important Commutativity}, \eqref{temp 13} and the fact that $\Y_T^A(\theta)=  N^*(\Tilde{\Lambda}_\theta, T, \Tilde{A})$ up to finitely many elements, to evaluate the LHS of \eqref{eq: temp12}, we get that \eqref{eq: cor to Main thm time visits} holds for $\theta$, with $\nu_{A,s}$ being the probability measure obtained by normalising $ f_*\mu_{\seda}|_{\Tilde{A}}$. Thus, the theorem holds.
\end{proof}

\appendix

\section{Convergence of Measures on lcsc spaces}
\label{Tight Convergence}
In this section, we briefly and rapidly collect some definitions and facts from measure theory that have been used in this paper. Our presentation here is heavily influenced by Shapira and Weiss \cite{SW22}, and in particular, we follow their notation. Let $X$ be a locally compact, second countable, Hausdorff topological space. Let $\mathcal{B}_X$ denote the Borel $\sigma$-algebra for the underlying topology and let $\mathcal{M}(X)$ denote the collection of finite regular Borel measures on $X$. {Let $C_b(X)$ denote the space of all real-valued, bounded, continuous functions on $X$.} For $\nu \in \mathcal{M}(X)$ and $f \in C_b(X)$, $\nu(f)$ will denote the integral $\int_X f \, d\nu$. We will use the \textit{tight topology} on $\mathcal{M}(X)$ for which convergence $\nu_l \rightarrow \nu$ is defined by either of the following equivalent requirements (see \cite[$\S 5$, Prop.~$9$]{Bou04b} for the equivalence):
\begin{itemize}
    \item[(i)] for all $f \in C_b(X)$, $\nu_l(f) \rightarrow \nu(f)$,
    \item[(ii)] for any compactly supported continuous function $f: X\rightarrow \R$, $\nu_l(f) \rightarrow \nu(f)$ and $\nu_l(X) \rightarrow \nu(X)$.
\end{itemize}
While this convergence is not equivalent to weak$-^*$ convergence, the two notions coincide when all measures involved are probability measures.

\begin{thm}[See {\cite[Thms.~2.1 $\&$ 2.7]{Bil68}} or {\cite[Chap.~4]{Bou04a}}]
\label{muJM}
    {If} $\nu_l, \nu \in \mathcal{M}(X)$, then $\nu_l \rightarrow \nu$ tightly if and only if for any $\nu$-JM set $E$ one has $\nu_l(E) \rightarrow \nu(E)$.

    Moreover, if $Y$ is also a locally compact second countable space and $\psi: X \rightarrow Y$ is a measurable function, then the push-forward map $\psi_*: \mathcal{M}(X)\rightarrow \mathcal{M}(Y)$ is continuous at a measure $\mu \in \mathcal{M}(X)$ (with respect to the tight topology) provided that $\psi$ is continuous $\nu$-almost everywhere.
\end{thm}

\section{Convergence in Compact extension}
\label{sec: Convergence in Compact extension}
This section aims to prove the following theorem, which provides sufficient conditions under which equidistribution in a measurable factor of a measure-preserving ergodic system lifts to equidistribution in the original system. The key example to keep in mind is when \( Y = \XA_d \), \( K = K_f \), and \( H = \{ {a}_t : t \in \R^{k+r-1} \} \), so that \( K \backslash Y = \X_d \).

\begin{thm}
\label{thm: Generecity in compact extension}
    {Let $Y$ be a locally compact second countable space carrying a continuous action of a compact Lie group $K$. Let $\pi: Y \rightarrow K \backslash Y$ denote the natural projection map. Suppose $m_Y$ is a Radon measure on $Y$, preserved under the action of $K$.} Let $H$ be {a} semi-group acting ergodically on the measure space $(Y,m_Y)$ such that the action of $H$ commutes with the action of $K$. Then, a sequence $(\mu_{l})_l$ of probability measures on $Y$ which satisfies 
    \begin{enumerate}
        \item The weak limit of $\mu_l$ along any subsequence is $H$-invariant,
        \item $(\pi)_*(\mu_l) \rightarrow (\pi)_*(m_{Y})$ as $l \rightarrow \infty$,
    \end{enumerate}
    must converge to $m_Y$.
\end{thm}
\begin{proof}
    We will prove the theorem by showing that for any positive monotone real sequence $(l_i)_i$ diverging to $\infty$, there is a subsequence {$(l_{i_j})_j$}, such that $\mu_{l_{i_j}} \rightarrow m_Y$. To see this, note that by the Banach-Alaoglu theorem, the sequence of measures in $\mu_{l_{i}}$ has a convergent subsequence (with respect to the weak topology). Assume that the subsequence is parametrised by {$l_{i_j}$} and that the limit measure is $\mu$. Claim that
    \begin{align}
        \label{eq: temp3}
        \int_{K} (R_\gamma)_* \mu \, d \mu_{K}= m_Y,
    \end{align}
    where $\mu_{K}$ denotes the Haar probability measure on $K$ and $R_\gamma : Y \rightarrow Y$ denotes the left action of $\gamma \in K$.

    To see this, fix $f \in C_c(Y)$. Define $F : K \backslash Y \rightarrow \mathbb{C}$ as 
    $$
    F(Ky)= \int_{K} f( \gamma y) \, d\mu_{K}(\gamma).
    $$   
    Clearly $F$ is well defined and belongs to $C_c(K \backslash Y)$ satisfying $\int_{K \backslash Y} F \, d(\pi)_*(m_Y) = \int_{Y} f \, dm_Y.$
   
    {Now, we have 
    \begin{align*}
        &\int_{K} \int_{Y} f(x) \, d((R_\gamma)_* \mu) \, d\mu_{K}(\gamma) \\
        &= \int_{K} \left( \lim_{j \rightarrow \infty}  \int_{Y} f( \gamma y) \, d\mu_{l_{i_j}}(y) \right) \, d\mu(\gamma)  \\
        &= \lim_{j \rightarrow \infty}  \int_{Y} \int_{K} f( \gamma y)\, d\mu_{K}(\gamma)  \, d\mu_{l_{i_j}}(y)
        \text{ (by the Dominated Convergence Theorem and Fubini's Theorem) } \\
        &= \lim_{j \rightarrow \infty}  \int_{K \backslash Y}  F(Ky)  \, d(\pi)_*(\mu_{l_{i_j}})(Ky)   \ \ \text{ (by the definition of $F$)} \\
        &= \int_{K \backslash Y} F \, d(\pi)_*(m_Y) = \int_{Y}f \, dm_Y,
    \end{align*}}
    where in penultimate equality, we used the fact that $(\pi)_*(\mu_l) \rightarrow (\pi)_*(m_Y)$. Since this holds for all $f \in C_c(Y)$, the claim is verified.

{Note that $m_Y$ is an $H$-ergodic and $ (R_\gamma)_* \mu$ is $H$-invariant for all $\gamma \in K$. Thus by \eqref{eq: temp3} and ergodic decomposition theorem,} we get that for $\mu_{K}$ almost every $\gamma \in K$, we have $(R_\gamma)_* \mu = m_Y= (R_\gamma)_* m_Y$, which implies $\mu= m_Y$. This proves the theorem.
\end{proof}

\section{Effective Equidistribution under Multiparameter flows}
\label{Appendix: Measure Theory}
{This section is devoted to proving the following theorem, which provides sufficient conditions under which \(\nu\)-almost every point becomes equidistributed under a  multi-parameter diagonal flow in a larger ambient space. A central motivating example is when \(Y = \Mat([0,1])\), \(\nu\) is the restriction of Lebesgue measure to \(Y\), and
\[
F(y, t) = f\left(a_t \Lambda_y \Gamma\right),
\]
for some smooth function \(f\) on \(\X_d\). The theorem below then implies that for Lebesgue-almost every \(\theta\), we have
\[
\lim_{T \rightarrow \infty} \frac{1}{m_{\R^{k+r-1}}(J^T)} \int_{J^T} f(a_t \Lambda_\theta \Gamma) \, dt = \mu_{\X_d}(f),
\]
as stated in Theorem~\ref{thm: Effective Equidistribution}.}

\begin{thm}
\label{Pointwise Main Lemma}
{Let} $m_1, \ldots, m_k, n_1, \ldots, n_r$ be positive integers. Let \[\cR = \left\{ (t_1, \ldots, t_k, s_1, \ldots, s_r) \in (\R_{\geq 0})^{k+r} : \sum_{i=1}^k m_i t_i = \sum_{j=1}^r n_j s_j \right\}.\] Let \((Y,\nu)\) be a probability space, and suppose \(F: Y \times (\R_{\geq 0})^{k+r} \to \R\) is a bounded measurable function. Assume there {exist} constants \(\delta > 0\) and \(C > 0\) such that for all \(h_1, h_2 \in \cR\),
\begin{align}
\label{eq: con pointwise main lemma}
    \left| \int_Y F(x,h_1) F(x,h_2) \, d\nu(x) \right| \leq C \exp\left( - \delta \min \left\{ \lfloor h_1 \rfloor, \lfloor h_2 \rfloor, \|h_1 - h_2\|_\infty \right\} \right),
\end{align}
{ where, for \(h = (h_1, \ldots, h_{k+r}) \in \cR\), the quantity $\lfloor h \rfloor$ defined as in Theorem~\ref{KM1 Thm} equals
\begin{align}
    \label{eq: def lfloor}
    \lfloor h \rfloor := \min\{ |h_1|, \ldots, |h_{k+r}| \}.
\end{align}}
Then, for any \(\e > 0\), for \(\nu\)-almost every { \(y \in Y\), we have
\begin{align}
\label{eq: thm effective multi}
    \frac{1}{T^{k+r-1}} \int_{\J^T} F(y,h) \, dm_{\cR}(h) =
\begin{cases}
o\left(T^{-1} (\log T)^{\frac{3}{2}+r+\e}\right), & \text{if } k+r > 2, \\
o\left(T^{-1/2} (\log T)^{\frac{3}{2}+\e}\right), & \text{if } k+r = 2,
\end{cases}
\end{align}}
where 
$$
\J^T = \{(t_1, \ldots, t_k, s_1, \ldots, s_r) \in \cR:\  \forall \  i, \  s_i \leq T\}.
$$
\end{thm}
{ \begin{rem}
  Note that the set $\cR$ is contained in the affine hyperplane of $\R^d$ given by
\begin{align}
\label{eq: def plane last}
    \left\{ (t_1, \ldots, t_k, s_1, \ldots, s_r) \in \R^{k+r} : \sum_{i=1}^k m_i t_i = \sum_{j=1}^r n_j s_j \right\}.
\end{align}
Thus, $\cR$ carries a natural measure \( m_{\cR} \), defined up to a scaling factor, obtained by restricting the Lebesgue measure on the hyperplane~\eqref{eq: def plane last} to $\cR$. This measure satisfies the scaling property
\[
m_{\cR}(\J^T) = c \cdot T^{k+r-1},
\]
where the implied constant \( c \) depends on the choice of normalization for \( m_{\cR} \). This constant also affects the implied constant in~\eqref{eq: thm effective multi}.

For concreteness, throughout the proof we fix the measure on the hyperplane~\eqref{eq: def plane last} to be the pushforward of the Lebesgue measure \( m_{\R^{k+r-1}} \) under the map
\begin{align}
    \label{eq: nnn map def}
    (t_1, \ldots, t_{k+r-1}) \mapsto \left(t_1, \ldots, t_{k+r-1}, \frac{ \sum_{i=1}^k m_i t_i - \sum_{j=1}^{r-1} n_j s_j }{n_r} \right),
\end{align}
and define \( m_{\cR} \) to be its restriction to the subset \( \cR \).
\end{rem}}
\begin{rem}
    The theorem for $k+r=2$, i.e, $r=k=1$ is proved in \cite[Thm.~3.1]{KSW}. Hence, we will assume that $k+r>2$ for the proof. Philosophically, the arguments in this appendix can be traced back to \cite{KSW}, and in fact, earlier works of Cassels \cite{Cassels} and Schmidt \cite{Schmidt}.
\end{rem}

To prove Theorem~\ref{Pointwise Main Lemma}, we introduce the {following notation and auxiliary lemmas.} Define, for {\( \alpha = (\alpha_1, \ldots, \alpha_r), \beta = (\beta_1, \ldots, \beta_r) \in (\R_{\geq 0})^r \)} satisfying \( \alpha_i < \beta_i \) for all \( i \), 
\begin{align}
\label{def: J(a,b)}
    \J(\alpha,\beta) = \{(t_1, \ldots, t_k, s_1, \ldots, s_r) \in \cR : \forall i,\ \alpha_i \leq s_i \leq \beta_i \}.
\end{align}
Also, for \( c > 0 \), define
\begin{align}
    \label{def:J^+(a,b,c)}
    \J^+(\alpha,\beta,c) &= \{(h_1, \ldots, h_{k+r}) \in \J(\alpha,\beta) : \lfloor h \rfloor \geq c \}, \\
    \label{def:J^-(a,b,c)}
    \J^-(\alpha,\beta,c) &= \{(h_1, \ldots, h_{k+r}) \in \J(\alpha,\beta) : \lfloor h \rfloor \leq c \}.
\end{align}
{It is clear from definitions that 
\begin{align}
\label{eq: nn cap 0}
    m_{\cR}\left( \J^+(\alpha,\beta,c) \cap \J^-(\alpha,\beta,c)  \right) =0,
\end{align}
since the set $\J^+(\alpha,\beta,c) \cap \J^-(\alpha,\beta,c)$ is contained in union of finitely many $k+r-2$-dimensional hyperplanes in $\cR$.} \vspace{0.2in}

We begin with the following lemma.
\begin{lem}
\label{lem: L^2 estimate weak}
{ Let $c>0$ be given. Let \( \alpha = (\alpha_1, \ldots, \alpha_r), \beta = (\beta_1, \ldots, \beta_r) \in (\R_{\geq 0})^r \) be such that \( \alpha_i < \beta_i \) for all \( i \).} Then
\begin{align}
    \int_Y \left( \int_{\J(\alpha,\beta)} F(x,h)\, dm_{\cR}(h) \right)^2 d\nu(x)
    &\leq 2\, m_{\cR}(\J^-(\alpha,\beta,c))^2 \|F\|_\infty^2 + 2 (2c)^{k+r-1} m_{\cR}(\J(\alpha,\beta)) \|F\|_\infty^2 \nonumber \\
    &\quad + 2C\, m_{\cR}(\J(\alpha,\beta))^2 e^{-\delta c}. \label{eq:L^2 estimate weaker}
\end{align}
\end{lem}

\begin{proof}
{ We begin by decomposing the inner integral over \( \J(\alpha, \beta) \) as follows:
\begin{align*}
    &\int_Y \left( \int_{\J(\alpha,\beta)} F(x,h)\, dm_{\cR}(h) \right)^2 d\nu(x) \\
    &= \int_Y \left( \int_{\J^-(\alpha,\beta,c)} F(x,h)\, dm_{\cR}(h) + \int_{\J^+(\alpha,\beta,c)} F(x,h)\, dm_{\cR}(h) \right)^2 d\nu(x) \\
    &\leq 2 \int_Y \left( \int_{\J^-(\alpha,\beta,c)} F(x,h)\, dm_{\cR}(h) \right)^2 d\nu(x)
    + 2 \int_Y \left( \int_{\J^+(\alpha,\beta,c)} F(x,h)\, dm_{\cR}(h) \right)^2 d\nu(x),
\end{align*}
where we used the equation \eqref{eq: nn cap 0} and the inequality \( (a + b)^2 \leq 2a^2 + 2b^2 \).

For the first term, we estimate using the bound \( |F(x,h)| \leq \|F\|_\infty \):
\begin{align}
    \int_Y \left( \int_{\J^-(\alpha,\beta,c)} F(x,h)\, dm_{\cR}(h) \right)^2 d\nu(x)
    \leq \left( m_{\cR}(\J^-(\alpha,\beta,c)) \|F\|_\infty \right)^2. \label{eq: J- estimate}
\end{align}

We now estimate the second term:
\begin{align}
    \int_Y \left( \int_{\J^+(\alpha,\beta,c)} F(x,h)\, dm_{\cR}(h) \right)^2 d\nu(x)
    = \int_{\J^+(\alpha,\beta,c) \times \J^+(\alpha,\beta,c)} \left( \int_Y F(x,h_1) F(x,h_2)\, d\nu(x) \right) dm_{\cR}(h_1) dm_{\cR}(h_2). \label{eq: double integral}
\end{align}

We split this double integral based on whether \( \|h_1 - h_2\|_\infty < c \) or \( \|h_1 - h_2\|_\infty \geq c \).

\medskip
\noindent
\textit{Estimate for near-diagonal pairs.} We have
\begin{align}
    &\int_{\substack{h_1, h_2 \in \J^+(\alpha,\beta,c) \\ \|h_1 - h_2\|_\infty < c}} \left( \int_Y F(x,h_1) F(x,h_2)\, d\nu(x) \right) dm_{\cR}(h_1) dm_{\cR}(h_2) \nonumber \\
    &\leq \|F\|_\infty^2 \cdot m_{\cR}(\J^+(\alpha,\beta,c)) \cdot \sup_{h_1 \in \J^+(\alpha,\beta,c)} m_{\cR}\left( \left\{ h_2 \in \J^+(\alpha,\beta,c) : \|h_1 - h_2\|_\infty < c \right\} \right) \nonumber \\
    &\leq (2 c)^{k+r-1} m_{\cR}(\J(\alpha,\beta)) \|F\|_\infty^2. \label{eq: near-diagonal estimate}
\end{align}
The final bound follows from a geometric estimate: for each \( h_1 \in \J^+(\alpha,\beta,c) \), the preimage of the set of \( h_2 \) satisfying \( \|h_1 - h_2\|_\infty < c \) under the map \eqref{eq: nnn map def} lies in a box of volume \( (2c)^{k+r-1} \).

\medskip
\noindent
\textit{Estimate for off-diagonal pairs.} For \( \|h_1 - h_2\|_\infty \geq c \), we use the decay assumption \eqref{eq: con pointwise main lemma} to get
\begin{align}
    \int_{\substack{h_1, h_2 \in \J^+(\alpha,\beta,c) \\ \|h_1 - h_2\|_\infty \geq c}} \left( \int_Y F(x,h_1) F(x,h_2)\, d\nu(x) \right) dm_{\cR}(h_1) dm_{\cR}(h_2)
    \leq m_{\cR}(\J(\alpha,\beta))^2 C e^{-\delta c}. \label{eq: off-diagonal estimate}
\end{align}

\medskip
\noindent
Combining the bounds \eqref{eq: J- estimate}, \eqref{eq: near-diagonal estimate}, and \eqref{eq: off-diagonal estimate} completes the proof of the lemma.}
\end{proof}

For a positive integer $s$ we let $L_s$ be the set of intervals of form $[2^ij,2^i(j+1)]$, where $i,j$ are non-negative integers and $2^i(j+1) < 2^s.$

\begin{lem}
\label{lem: nnnn 1}
    For a positive integer $s$, we have that 
    $$
    \sum_{[\alpha_1,\beta_1] \in L_s} \cdots \sum_{[\alpha_r, \beta_r] \in L_s} \int_{Y}\left( \int_{\J((\alpha_1, \ldots, \alpha_r),(\beta_1, \ldots, \beta_r))} F(x,h) \, dm_{\cR}(h) \right)^2\, d\nu(x) \ll s^{r+2} 2^{2s(k+r-2)}.
    $$
\end{lem}
\begin{proof}
    Using Lemma \ref{lem: L^2 estimate weak} with $c=  (\log (2^{s(k+r-1)}))/\delta $, we get that
    { \begin{align*}
        &\sum_{[\alpha_1, \beta_1] \in L_s} \cdots \sum_{[\alpha_r,\beta_r] \in L_s} \int_{Y}\left( \int_{\J((\alpha_1, \ldots, \alpha_r),(\beta_1, \ldots, \beta_r))} F(x,h) \, dm_{\cR}(h) \right)^2\, d\nu(x)  \\
        &\leq \sum_{[\alpha_1, \beta_1] \in L_s} \cdots \sum_{[\alpha_r,\beta_r] \in L_s} \left(2 m_{\cR}(\J^-((\alpha_1, \ldots, \alpha_r),(\beta_1, \ldots, \beta_r),c))^2 \|F\|_\infty^2 \right. \\
        & \left. + 2 (2c)^{k+r-1} m_{\cR}(\J((\alpha_1, \ldots, \alpha_r),(\beta_1, \ldots, \beta_r))) \|F\|_\infty^2 + 2C m_{\cR}(\J((\alpha_1, \ldots, \alpha_r),(\beta_1, \ldots, \beta_r)))^2  e^{-\delta c} \right) \\
        &= \sum_{i_1=0}^{s-1} \cdots \sum_{i_r=0}^{s-1} \sum_{j_1=0}^{2^{s-i_1}-1} \cdots \sum_{j_r=0}^{2^{s-i_r}-1} \left(2 m_{\cR}(\J^-((2^{i_1}j_1, \ldots, 2^{i_r}j_r),(2^{i_1}(j_1+1), \ldots, 2^{i_r}(j_r+1)),c))^2 \|F\|_\infty^2 \right.\\
        &+ 2(2c)^{k+r-1} m_{\cR}(\J((2^{i_1}j_1, \ldots, 2^{i_r}j_r),(2^{i_1}(j_1+1), \ldots, 2^{i_r}(j_r+1))) \|F\|_\infty^2 \\ 
        & \left. + 2C m_{\cR}(\J((2^{i_1}j_1, \ldots, 2^{i_r}j_r),(2^{i_1}(j_1+1), \ldots, 2^{i_r}(j_r+1))))^2  e^{-\delta c} \right) \\ 
        &\leq \sum_{i_1=0}^{s-1} \cdots \sum_{i_r=0}^{s-1} \left( 2 m_{\cR}(\J^-((0, \ldots, 0),(2^s, \ldots, 2^s),c))^2 \|F\|_\infty^2 \right. \\
         & \left. +  2 (2c)^{k+r-1} m_{\cR}(\J((0, \ldots, 0),(2^s, \ldots, 2^s))) \|F\|_\infty^2 + 2C m_{\cR}(\J((0, \ldots, 0),(2^s, \ldots, 2^s)))^2  e^{-\delta c} \right).
        \end{align*}}
       { Note that
        \begin{align*}
            m_{\cR}(\J^-((0, \ldots, 0),(2^s, \ldots, 2^s),c))^2 &\ll c^2 2^{2s(k+r-2)}  \ll  s^2 2^{2s(k+r-2)}, \\
            m_{\cR}(\J((0, \ldots, 0),(2^s, \ldots, 2^s))) &\ll 2^{s(k+r-1)}.
        \end{align*}
        Therefore we have
        \begin{align*}
            &\sum_{[\alpha_1, \beta_1] \in L_s} \cdots \sum_{[\alpha_r,\beta_r] \in L_s} \int_{Y}\left( \int_{\J((\alpha_1, \ldots, \alpha_r),(\beta_1, \ldots, \beta_r))} F(x,h) \, dm_{\cR}(h) \right)^2\, d\nu(x)  \\
            &\ll s^r \|F\|_\infty^2 \left( s^2 2^{2s(k+r-2)} + s^{k+r-1} 2^{s(k+r-1)} + 2^{s(k+r-1)} \right) \\
            &\ll s^{r+2} 2^{2s(k+r-2)} \text{ since } k+r>2. 
        \end{align*}
       This proves the lemma.}
\end{proof}

\begin{lem}
\label{lem: nnnn 2}
    Let $l,s$ be positive integers with $l< 2^s$. Then the interval $[0,l]$ can be covered by at most $s$ intervals in $L_s$.
\end{lem}
\begin{proof}
    These intervals can be easily constructed using the binary expansion of $k$. 
\end{proof}
\begin{lem}
\label{lem: L^2 Sets}
    For every $\e > 0$, there exists a sequence of measurable subsets $\{Y_s\}_{s \in \N}$ of $Y$ such that 
    \begin{enumerate}
        \item $\nu(Y_s) \ll  s^{-(1+2\e)} $ 
        \item For every positive integer { $l< 2^s$ and every $y \notin Y_s$, one has $$\left| \int_{\J^l} F(y,h) dm_{\cR}(h) \right| \leq 2^{s(k+r-2)} s^{ \tfrac{3}{2} + r+\e} .$$}
   \end{enumerate}
\end{lem}
\begin{proof}
{ Define
$$Y_s = \left\{ y \in Y: \sum_{[\alpha_1,\beta_1] \in L_s} \cdots \sum_{[\alpha_r, \beta_r] \in L_s} \int_{Y}\left( \int_{\J((\alpha_1, \ldots, \alpha_r),(\beta_1, \ldots, \beta_r))} F(x,h) \, dm_{\cR}(h) \right)^2\, d\nu(x) > s^{r+3+ 2\e} 2^{2s(k+r-2)} \right\}.$$

Assertion (1) follows from Lemma~\ref{lem: L^2 estimate weak} and Markov’s inequality.

For (2), fix \( y \notin Y_s \). By Lemma~\ref{lem: nnnn 2}, there exists a subset \( L(l) \subset L_s \) with \( \# L(l) \leq s \) such that
\[
[0, l] = \bigcup_{[\alpha, \beta] \in L(l)} [\alpha, \beta],
\]
and therefore,
\begin{align*}
    \left( \int_{\J^l} F(y,h)\, dm_{\cR}(h) \right)^2
    &= \left( \int_{\J((0, \ldots, 0), (l, \ldots, l))} F(y,h)\, dm_{\cR}(h) \right)^2 \\
    &= \left( \sum_{[\alpha_1, \beta_1] \in L(l)} \cdots \sum_{[\alpha_r, \beta_r] \in L(l)} \int_{\J((\alpha_1, \ldots, \alpha_r), (\beta_1, \ldots, \beta_r))} F(y,h)\, dm_{\cR}(h) \right)^2.
\end{align*}

Using the inequality
\[
\left( \sum_{i=1}^p z_i \right)^2 \leq p \left( \sum_{i=1}^p z_i^2 \right),
\]
with \( p = \# L(l)^r \leq s^r \), we obtain
\begin{align*}
    &\left( \sum_{[\alpha_1, \beta_1] \in L(l)} \cdots \sum_{[\alpha_r, \beta_r] \in L(l)} \int_{\J((\alpha_1, \ldots, \alpha_r), (\beta_1, \ldots, \beta_r))} F(y,h)\, dm_{\cR}(h) \right)^2 \\
    &\leq s^r \sum_{[\alpha_1, \beta_1] \in L(l)} \cdots \sum_{[\alpha_r, \beta_r] \in L(l)} \left( \int_{\J((\alpha_1, \ldots, \alpha_r), (\beta_1, \ldots, \beta_r))} F(y,h)\, dm_{\cR}(h) \right)^2.
\end{align*}

Since \( L(l) \subset L_s \), we further bound this by
\[
s^r \sum_{[\alpha_1, \beta_1] \in L_s} \cdots \sum_{[\alpha_r, \beta_r] \in L_s} \left( \int_{\J((\alpha_1, \ldots, \alpha_r), (\beta_1, \ldots, \beta_r))} F(y,h)\, dm_{\cR}(h) \right)^2.
\]

By definition of \( Y_s \), the right-hand side is bounded for \( y \notin Y_s \) by
\[
 s^r \cdot s^{r+3+2\varepsilon} \cdot 2^{2s(k+r-2)} = 2^{2s(k+r-2)} \cdot s^{2r + 3 + 2\varepsilon}.
\]

Taking square roots gives
\[
\left| \int_{\J^l} F(y,h)\, dm_{\cR}(h) \right| \leq 2^{s(k+r-2)} \cdot s^{\tfrac{3}{2} + r + \varepsilon},
\]
which proves assertion (2). Hence, the lemma follows. }
\end{proof}

\begin{proof}[Proof of Theorem \ref{Pointwise Main Lemma}]
   { We fix $\e > 0$ and choose a sequence of measurable subsets $\{Y_s\}_{s \in \N}$ as given by Lemma~\ref{lem: L^2 Sets}.  
Since
\[
\sum_{s=1}^{\infty} \nu(Y_s) \ll \sum_{s=1}^{\infty} s^{-(1+2\e)} < \infty,
\]
the Borel--Cantelli lemma implies that there exists a measurable subset $Y(\e) \subset Y$ of full measure such that, for every $y \in Y(\e)$, there exists $s_y \in \N$ with the property that $y \notin Y_s$ whenever $s \geq s_y$.

We claim that for every $y \in Y(\e)$ one has
\[
\frac{1}{T^{k+r-1}} \left| \int_{\J^T} F(y,h) \, dm_{\cR}(h) \right| \ll T^{-1} (\log T)^{\frac{3}{2} + r + \e},
\]
provided that $T$ is sufficiently large. Here the implied constant depends only on $F$ and on the integers $n_1, \ldots, n_r, m_1, \ldots, m_k$.

To prove the claim, fix $y \in Y(\e)$ and take {$T \geq 2^{s_y}$.  
Let $l$ be the greatest integer less than or equal to $T$, and let $s \in \N$ be the unique integer such that
\[
2^{s-1} \leq l < 2^{s}.
\]
Then
\[
2^{s-1} \leq l \leq T < l+1 \leq 2^{s}.
\]
Since $T \geq 2^{s_y }$, we have $s \geq s_y$,} and therefore $y \notin Y_s$.  
By Lemma~\ref{lem: L^2 Sets}(2), it follows that
\begin{align*}
\left| \int_{\J^T} F(y,h) \, dm_{\cR}(h) \right|
&\ll T^{k+r-2} \|F\|_\infty + \left| \int_{\J^l} F(y,h) \, dm_{\cR}(h) \right| \\
&\ll T^{k+r-2} + 2^{s(k+r-2)} s^{\frac{3}{2} + r + \e} \\
&\ll T^{k+r-2} + T^{k+r-2} (\log T)^{\frac{3}{2} + r + \e}.
\end{align*}
Dividing through by $T^{k+r-1}$ yields the desired bound, proving the claim.

Since the above bound holds for every $y \in \bigcap_{i \in \N} Y(1/i)$, a set of full measure, and $\e>0$ is arbitrary, the estimate in \eqref{eq: thm effective multi} follows. This completes the proof.}
\end{proof}

\section{Measure of $J^T$}
\label{sec: measure of JT}

{In this section, we compute the Lebesgue measure of the set $J^T$ defined in \eqref{eq: def J^T}. To do so, we begin with a preliminary calculation involving a standard simplex.}

{\subsection*{Volume of a Truncated Simplex}}

{For any $a > 0$, define the set
\[
\widetilde{\Delta}_l(a) = \left\{ (x_1, \ldots, x_l) \in (\R_{\geq 0})^l : x_1 + \cdots + x_l \leq a \right\}.
\]
We compute its Lebesgue measure using Fubini's theorem. The integration proceeds iteratively as follows:
\begin{align*}
    m_{\R^l}(\widetilde{\Delta}_l(a)) &= \int_0^a \int_0^{a - x_1} \cdots \int_0^{a - (x_1 + \cdots + x_{l-1})} dx_l \cdots dx_1 \\
    &\quad \vdots \\
    &= \int_0^a \int_0^{a - x_1} \cdots \int_0^{a - (x_1 + \cdots + x_{l-i-1})} \frac{(a - x_1 - \cdots - x_{l-i-1})^i}{i!} \, dx_{l-i} \cdots dx_1 \\
    &\quad \vdots \\
    &= \frac{a^l}{l!}.
\end{align*}}

{\subsection*{Computing the Measure of $J^T$}}

{We now evaluate the volume of the set
\[
J^T = \{(t_1, \ldots, t_{k}, s_1, \ldots, s_{r-1}) \in (\R_{\geq 0})^{k+r-1} : \substack{ s_i \leq T \text{ for all } i \\ 0 \leq m_1t_1 + \cdots+ m_kt_k- n_1s_1 -\cdots- n_{r-1}s_{r-1} \leq n_rT} \}.
\]
We begin by applying a linear change of variables in the $t_i$ coordinates to normalise the coefficients $m_1, \ldots, m_k$:
\begin{align*}
    &m_{\R^{k+r-1}}(J^T) \\
    &= \frac{1}{m_1 \cdots m_k} \cdot m_{\R^{k+r-1}} \Bigg( \left\{(t_1, \ldots, t_{k}, s_1, \ldots, s_{r-1}) \in (\R_{\geq 0})^{k+r-1} : \substack{ s_i \leq T \text{ for all } i \\ 0 \leq t_1 + \cdots+ t_k- n_1s_1 -\cdots- n_{r-1}s_{r-1} \leq n_rT} \right\} \Bigg) \\
    &= \frac{1}{m_1 \cdots m_k} \int_{[0, T]^{r-1}} m_{\R^k} \left( \widetilde{\Delta}_k\left(n_1 s_1 + \cdots + n_{r-1} s_{r-1} + n_r T \right) \setminus \widetilde{\Delta}_k\left(n_1 s_1 + \cdots + n_{r-1} s_{r-1} \right) \right) ds_1 \cdots ds_{r-1}.
\end{align*}
Using the earlier computation for $\widetilde{\Delta}_k(a)$, we obtain:
\begin{align*}
    m_{\R^{k+r-1}}(J^T) &= \frac{1}{m_1 \cdots m_k} \int_{[0, T]^{r-1}} \frac{ \left(n_1 s_1 + \cdots + n_{r-1} s_{r-1} + n_r T \right)^k - \left(n_1 s_1 + \cdots + n_{r-1} s_{r-1} \right)^k }{k!} \, ds_1 \cdots ds_{r-1}\\
    &= \frac{n_r}{m_1 \cdots m_k} \int_{[0, T]^{r}} \frac{ \left(n_1 s_1 + \cdots + n_{r} s_{r} \right)^{k-1} }{(k-1)!} \, ds_1 \cdots ds_{r}.
\end{align*}
Now, make the substitution $s_i \mapsto T s_i$ for each $i$ and use homogeneity to factor out powers of $T$:
\begin{align*}
    m_{\R^{k+r-1}}(J^T) &= \frac{T^{k + r - 1} \cdot n_r}{m_1 \cdots m_k} \int_{[0, 1]^r} \frac{(n_1 s_1 + \cdots + n_{r-1} s_{r-1} + n_r s_r)^{k-1}}{(k-1)!} \, ds_1 \cdots ds_r.
\end{align*}
Evaluating the final integral, we obtain
\[
m_{\R^{k+r-1}}(J^T) = \frac{T^{k+r-1} \cdot n_r}{m_1 \cdots m_k } \left( \frac{1}{n_1 \cdots n_r} \cdot \frac{ c_{k+r-1}(n_1, \ldots, n_r)}{(k+r-1)!} \right) = \frac{T^{k+r-1} \cdot c_{k+r-1}(n_1, \ldots, n_r)}{m_1 \cdots n_{r-1} \cdot (k+r-1)!}.
\]}

\bibliography{Biblio}

\section*{Table of Notations}
\label{Table of notations}
\newcommand{\shorttoprule}{\specialrule{\heavyrulewidth}{0pt}{0pt}}
{\renewcommand{\arraystretch}{1.4} 
\begin{longtable}{@{}>{\raggedright\arraybackslash}p{0.21\textwidth} p{0.79\textwidth}@{}}
\shorttoprule
\textbf{Notation} & \textbf{Meaning / Description} \\
\midrule
\endfirsthead
\toprule
\textbf{Notation} & \textbf{Meaning / Description} \\
\midrule
\endhead
\bottomrule
\endfoot

$\Z$ &
The set of integers. \\

$\N$  &
The set of all positive integers. \\

$\Z_{\geq 0}$ &
The set of all non-negative integers. \\

$\R$ &
The set of all real numbers.\\

$\e$ & Fixed positive real number, used throughout the paper. \\

$m, n$ &  Fixed positive integers, used throughout the paper. \\

$d$ & The sum $m + n$, fixed throughout the paper. \\

$m_1, \ldots, m_k$ & Fixed positive integers such that $m_1 + \cdots + m_k = m$; used throughout the paper. \\

$n_1, \ldots, n_r$ & Fixed positive integers such that $n_1 + \cdots + n_r = n$; used throughout the paper. \\ 

$\|\cdot\|$ & With slight abuse of notation, denotes a fixed norm on each of $\R^{m_1}, \ldots, \R^{m_k}$ and $\R^{n_1}, \ldots, \R^{n_r}$, as well as the corresponding derived norms on $\R^m$, $\R^n$, and $\R^d$ defined via \eqref{eq: def norm 1}, \eqref{eq: def norm 2}, and \eqref{eq: def norm 3}. \\

$\varrho_1$, $\varrho_2$ & Projection maps from $\R^d = \R^m \times \R^n$ onto $\R^m$ and $\R^n$, respectively.\\

$\rho_i$ & For $1 \leq i \leq k$, the projection map from $\R^m = \R^{m_1} \times \cdots \times \R^{m_k}$ onto its $i$-th component $\R^{m_i}$. \\

$\rho_j'$ & For $1 \leq j \leq r$, the projection map from $\R^n = \R^{n_1} \times \cdots \times \R^{n_r}$ onto its $j$-th component $\R^{n_j}$. \\

$\Sphere^l$ & The unit sphere in $\R^{l}$ with respect to the chosen norm $\|\cdot\|$, i.e.\ 
$$
\Sphere^l = \{\,x\in\R^{l} : \|x\| = 1\}.
$$
\\

$B_\delta^{l}$ & Denotes the ball of radius $\delta$ around origin in $\R^l$ under the chosen norm, i.e., 
\[
B_\delta^{l} := \{ x \in \R^l : \|x\| \leq \delta \}.
\] \\

$\R_{>0}$ &
The set of all positive real numbers. \\

$\R_{\geq 0}$&
The set of all non-negative real numbers. \\

$\Primes$ &
The set of all prime numbers.
\\

$\Q_p$ & The field of \(p\)-adic numbers: the completion of \(\mathbb{Q}\) with respect to the \(p\)-adic norm \(|\cdot|_p\).
\\

$\Z_p$ & The ring of \(p\)-adic integers: the unit ball in \(\mathbb{Q}_p\), i.e., \(\{x \in \mathbb{Q}_p : |x|_p \leq 1\}\).
\\

$\A_f$ & The ring of finite adeles of \(\mathbb{Q}\): equals \(\prod_{p \in \Primes}' \mathbb{Q}_p\), where \(\prod'\) denotes the restricted product. That is, a sequence \(\underline{\beta} = (\beta_2, \beta_3, \ldots, \beta_p, \ldots)\) belongs to \(\mathbb{A}_f\) if and only if \(\beta_p \in \mathbb{Z}_p\) for all but finitely many primes \(p\). \\

$\A$ & The ring of adeles of $\Q$: $\A= \R \times \A_f$.\\

$\pi_f$ & Natural projection map on adeles:  
$\pi_f: \A \to \A_f$, \; $(g_\infty,g_f)\mapsto g_f$ \\

$\hZ$ & The profinite completion of $\Z$: $\hZ = \prod_{p \in \Primes} \Z_p$, the product of $p$-adic integers over all primes $p$.
\\

$\Lambda_\prim$ & The set of all primitive vectors in $\Lambda \subset \R^d$: that is, vectors $v \in \Lambda \setminus \{0\}$ such that if $v = n w$ for some $w \in \Lambda$ and $n \in \Z$, then $n = \pm1$. \\

$\hZp^d$ & The closure in $\hZ^d$ of the image of the set of primitive vectors $\Z_\prim^d$ under the natural inclusion $\Z^d \hookrightarrow \hZ^d$. \\

$\M_{i \times j}(I)$ & The set of all $i \times j$ matrices with entries in the set $I$. \\

$\GL$ & The general linear group of invertible real matrices. \\

$\SL$ & The special linear group: matrices in $\GL$ with determinant $1$. \\

$A_{\cdot j}$ & For a $d \times d$ matrix $A$ and $1 \leq j \leq d$, the $j$-th column of $A$. \\

$\bfe_j$ & The $j$-th standard basis vector in $\R^d$ (or $\Z^d$), i.e., the vector with $1$ in the $j$-th position and $0$ in all other positions.\\

$J^T$ & The set of parameters $(t_1, \dots, t_k,\, s_1, \dots, s_{r-1}) \in (\R_{\geq 0})^k \times [0,T]^{r-1}$ satisfying a linear constraint depending on $T$; see Equation~\eqref{eq: def J^T}. \\

$c_{k+r-1}(n_1, \ldots, n_r)$ & Alternating sum defined in \eqref{eq: def c k r 1}, related to the volume of $J^T$; see also \eqref{eq: measure of J T}. \\

$\#A$ & The cardinality of the set $A$, i.e., the number of elements in $A$. \\

$A^\circ$ or $\Int_X(A)$ & The interior of the set $A$ in $X$. \\

$\cl_X(A)$ & The closure of the set $A$ in $X$. \\

$\partial(A)$, $\partial_X(A)$ & For \( A \subset X \), this denotes the boundary of \( A \) in \( X \), i.e., \( \partial(A) = \overline{A} \setminus A^\circ \). \\

$\ind_A$ & The characteristic function of the set $A$, defined as
{\renewcommand{\arraystretch}{1}$$
\ind_A(x) =
\begin{cases}
1 & \text{if } x \in A, \\
0 & \text{if } x \notin A.
\end{cases}
$$} \\

$f|_A$ & The restriction of a function $f : X \to Y$ to a subset $A \subset X$, defined by $f|_A(x) = f(x)$ for all $x \in A$. \\

$\mathbf{1}$ & The constant function with value $1$, i.e., for all $x$, we have $\mathbf{1}(x) = 1$. \\

$C_b(X)$ & The space of bounded continuous functions on $X$. \\
$C_c(X)$ & The space of continuous functions on $X$ with compact support. \\
$C_c^\infty(X)$ & The space of infinitely differentiable functions on $X$ with compact support. \\

$\|\cdot\|_\infty$ & For a vector $v \in \R^\ell$, denotes the supremum of the absolute values of its coordinates; for a real-valued function $f$, denotes the supremum of $|f|$. \\

$\delta_x$ & The Dirac measure at point $x$, i.e., the probability measure on a measurable space assigning mass $1$ to $\{x\}$ and $0$ to its complement. \\

$\nu$-JM & For a finite Radon measure $\nu$ on a topological space $X$ and a Borel measurable set $E \subset X$, we say that $E$ is \textit{Jordan measurable with respect to $\nu$} (abbreviated $\nu$-JM) if $\nu(\partial_X(E)) = 0$. \\

$\mu|_A$ & For a measure $\mu$ on $X$ and a measurable subset $A\subset X$, the restricted measure $\mu|_A$ is defined by $\mu|_A(B)=\mu(A\cap B)$ for every measurable $B\subset X$. \\

$f_*\mu$ & The pushforward of the measure $\mu$ on $X$ under a measurable map $f: X \to Y$, defined by $(f_*\mu)(B) := \mu(f^{-1}(B))$ for all Borel sets $B \subset Y$. \\

$m_{\R^l}$ & The Lebesgue measure on $\R^l$. \\

$\mu^{(\Sphere^l)}$ & The pushforward of the Lebesgue measure $m_{\R^l}|_{\{x \in \R^l : \|x\| \leq 1\}}$ under the projection map $x \mapsto x/\|x\|$. Note: $\mu^{(\Sphere^l)}$ is not necessarily a probability measure. \\

$m_{\hat\Z^d_{\mathrm{prim}}}$ & $\SL_d(\hat\Z)$-invariant probability measure supported on $\hat\Z^d_{\mathrm{prim}}$, uniquely defined since $\hat\Z^d_{\mathrm{prim}}$ is an orbit of the natural $\SL_d(\hat\Z)$-action on $\hat\Z^d$. \\

$\zeta$ & The Riemann zeta function, defined for $\Re(s) > 1$ by $\zeta(s) = \sum_{n=1}^\infty \frac{1}{n^s}$, and extended to other values via analytic continuation. \\

$K_f$ & The compact group $\prod_{p \in \Primes } \SL_d(\Z_p)$.\\
$G$ & The group $\SL_d(\R)$.\\

$H$ & The group $ \left\{\begin{pmatrix}
        A & 0 \\ x & 1 
    \end{pmatrix}: A \in \SL_{n}(\R), x \in \R^n \right\}$ \\

$\Gamma$ & The discrete subgroup $\SL_d(\Z)$ of $G$. \\

$\X_d$ & The homogeneous space $G/\Gamma$, identified with the space of all unimodular lattices in $\R^d$ via the map $A\Gamma \mapsto A\Z^d$. \\

$\XA_d$ & The quotient space $\SL_d(\A)/\SL_d(\Q)$. \\

$\pi$ & The natural projection $\pi: \XA_d  \rightarrow \X_d$, see Section~\ref{ sec: Adelic Setup} for the detailed definition.\\

$\mu_{\X_d}$ & Denotes the unique $G$-invariant probability measure on $\X_d$ \\
$\mu_{\XA_d}$ & Denotes the unique $\SL_d(\A)$-invariant probability measure on $\XA_d$. \\

$\X_d(W, l)$ & For a subset $W \subset \R^d$ and integer $l \geq 1$, this denotes the set of lattices $\Lambda$ such that $\# (\Lambda_{\prim} \cap W) \geq l$. \\

$\X_d(W)$ & Shorthand for $\X_d(W, 1)$, i.e., lattices $\Lambda$ with at least one primitive point in $W$. \\

$\X_d^\sharp(W)$ & The set $\X_d(W)$ excluding those lattices with two or more primitive points in $W$, i.e., $\X_d^\sharp(W) = \X_d(W) \setminus \X_d(W, 2)$. \\

$\vector$ & Function from $\X_d^\sharp(W)$ to $W$ assigning to each $\Lambda$ the unique primitive point in $\Lambda \cap W$; i.e., $\{\vector(\Lambda)\} = \Lambda_{\prim} \cap W$. The same symbol $\vector$ is used for all sets $W$ by abuse of notation. \\

$\E_d^j$ & The space of unimodular lattices in $\R^d$ that contain $\bfe_j$ as a primitive vector, for $j = 1, \ldots, d$. \\

$m_{\E_d^j}$ & For each $1 \leq j \leq d$, the unique probability measure on $\E_d^j$ that is invariant under the group $\{h \in \SL_d(\R) : h \cdot \bfe_j = \bfe_j\}$; uniqueness follows since $\E_d^j$ can be identified with a homogeneous space of this group. \\

\((\E_d, m_{\E_d})\) & Shorthand notation for \((\E_d^d, m_{\E_d^d})\). \\

$\ZZ_j$ & The space $\E_d^j \times \Sphere^{m_1} \times \cdots \times \Sphere^{m_k} \times \Sphere^{n_1} \times \cdots \times \Sphere^{n_r} \times [0,\e] \times \hZp^d$, for each $1 \leq j \leq d$.\\

$\tmu^j$ & The natural measure on $\ZZ_j$, for each $1 \leq j \leq d$, defined as the product of the canonical measures on its components: 
$m_{\E_d^j} \times \mu^{(\Sphere^{m_1})} \times \cdots \times \mu^{(\Sphere^{m_k})} \times \mu^{(\Sphere^{n_1})} \times \cdots \times \mu^{(\Sphere^{n_r})} \times m_{\R}|_{(0,\e)} \times  m_{\hZp^d}$. \\

$\mu^j$ & The probability measure on $\ZZ_j$, for each $1 \leq j \leq d$, obtained by normalising $\tmu^j$ \\

$\disp(\cdot)$ & A function from $\M_{m \times n}(\R) \times \Z^m \times \Z^n $ to $\R$, defined by
$$\disp(\theta, p, q) = \left( \prod_{i=1}^k \|\rho_i(p + \theta q)\|^{m_i} \right) \cdot \left( \prod_{j=1}^r \|\rho_j'(q)\|^{n_j} \right).$$ \\

$\proj(\cdot)$ & A function from $\M_{m \times n}(\R) \times \Z^m \times \Z^n \setminus \disp^{-1}(0)$ to $\Sphere^{m_1} \times \cdots \times \Sphere^{m_k} \times \Sphere^{n_1} \times \cdots \times \Sphere^{n_r}$, defined by
$$\proj(\theta, p, q) = \left( \frac{\rho_1(p + \theta q)}{\|\rho_1(p + \theta q)\|}, \ldots, \frac{\rho_k(p + \theta q)}{\|\rho_k(p + \theta q)\|}, \frac{\rho_1'(q)}{\|\rho_1'(q)\|}, \ldots, \frac{\rho_r'(q)}{\|\rho_r'(q)\|} \right).$$ \\

$\lambda_j$ & For each $1 \leq j \leq d$, a map from $\left(\Mat \times (\Z^m \times \Z^n)_\prim\right) \setminus \left(\disp^{-1}(0) \cup \left\{(\theta, p, q) : (p+\theta q, q)_j = 0\right\} \right)$ to $\E^j_d$. See Section~\ref{ sec: The Relative Size} for the precise definition. \\

\(\Theta_j\) & For each $1 \leq j \leq d$, the map from $\left(\Mat \times (\Z^m \times \Z^n)_\prim\right) \setminus \left(\disp^{-1}(0) \cup \left\{(\theta, p, q) : (p+\theta q, q)_j = 0\right\} \right)$ to $\ZZ_j$ defined by  
\[
\Theta_j(\theta, p, q) := \left( \lambda_j(\theta, p, q),\, \proj(\theta, p, q),\, \disp(\theta, p, q),\, (p, q) \right).
\]
\\

$I_l$ & The $l \times l$ identity matrix.\\

${a}_{t}$ & Block-diagonal matrix in $\SL_d(\R)$ with $t \in \R^{k+r-1}$; see equation~\eqref{defatt} for the precise definition \\

$({a}_t, \mu_{\X_d})$-generic  & See Definition~\ref{def: generic flow} \\
 $({a}_t, \mu_{\XA_d})$-generic  & See Definition~\ref{def: generic flow} \\

$\Lambda_\theta$ & For $\theta \in \Mat$, the point in $\X_d$ defined in \eqref{eq:def lambda theta}  \\
$\widetilde{\Lambda}_\theta$ &  For $\theta \in \Mat$, the point in $\XA_d$ defined in \eqref{eq: def tilde theta}  \\

$\Le$ & The set defined in \eqref{eq: def Le} \\
$\Le^+$ & The set defined in \eqref{eq:def Lej} \\

$\sed$ & The set $\X_d(\Le)$. \\

$\seda$ & The preimage of $\sed$ under the projection map $\pi_\A$, i.e., $\seda = \pi_\A^{-1}(\sed)$. \\

$T_{\e}(x)$ & For $x \in \X_d$, the set $\{t \in \R^{k+r-1} : a_t x \in \sed\}$. \\
$\Tilde{T}_{\e}(x)$ & For $x \in \XA_d$, the set $\{t \in \R^{k+r-1} : a_t x \in \seda\}$. \\

$E^I$ & For $E \subset \X_d$ or $E \subset \XA_d$ and $I \subset \R^{k+r-1}$, the set $$E^I := \{a_tx : t \in I,\, x \in E\}.$$ \\

$I_\tau$ & The box $[0, \tau]^{k + r - 1} \subset \R^{k + r - 1}$. \\

$\mu_{\sed}$ & Measure on $\sed$ defined by $$\mu_{\sed}(E) = \sup_{\tau > 0} \left\{ \frac{1}{\tau^{k+r-1}} \mu_{\X_d}(E^{I_\tau}) \right\}$$ for measurable $E \subset \sed$; see Section~\ref{sec: The Cross-section: Measure on the cross-section} for details. \\

$\mu_{\seda}$ & The unique $K_f$-invariant measure on $\seda$ such that $(\pi_\A)_* \mu_{\seda} = \mu_{\sed}$. \\

$(\sed)_{\geq \delta}$ & For any $\delta > 0$, the set
$\{ x \in \sed : \min\{ \|t\|_{\el} : t \in T_\e(x) \setminus \{(0,0)\} \} \geq 2\delta \}$, consisting of points whose first return to $\sed$ under the $a_t$-action is at least $2\delta$ away. \\

$(\sed)_{< \delta}$ & The complement of $(\sed)_{\geq \delta}$ in $\sed$, i.e., $(\sed)_{< \delta} = \sed \setminus (\sed)_{\geq \delta}$. \\

$(\seda)_{\geq \delta}$ & The set $\pi_\A^{-1}((\sed)_{\geq \delta})$. \\

$(\seda)_{< \delta}$ & The set $\pi_\A^{-1}((\sed)_{< \delta})$. \\

$\ssed$ & The set $\X_d^\sharp(\Le^+)$. \\

$\sseda$ & The set $\pi_\A^{-1}(\X_d^\sharp(\Le^+))$. \\

$\ued$ & The set $\ssed \cap \X_d(\Le^\circ)$. \\

$\ueda$ & The set $\pi_\A^{-1}(\ued)$. \\

$ \cL_\e$ & The set $\{y \in \R^{n_r} : y_{n_r} > 0, \ \|y\|^{n_r} \leq \e\}$. \\

$\varphi$ & Map from $\E_d \times \Sphere^{m_1} \times \cdots \times \Sphere^{n_{r-1}} \times \cL_\e$ to $\sed$, defined in~\eqref{eq:def varphi genral}. \\

$u(\cdot)$ & Map from $\Sphere^{m_1}\times\cdots\times\Sphere^{n_{r-1}} \times \R^{n_r-1} \times \R_{+}$ to $\SL_d(\R)$, defined in~\eqref{eq: def u general}. \\

$ \tilde \psi, \tilde \psi^-$ & Maps from $\sseda$ to $\ZZ_d$, defined in~\eqref{defpsi 1} and~\eqref{defpsi 2}. \\

$N(x,T,\beta,E)$ & For a Borel subset $E$ of $\seda$, $\beta \in \R$, $x \in \XA_d$, and $T>0$, defined in~\eqref{def:N(x,T,E)} \\

$\Y_T(\theta)$ & For $\theta \in \Mat$ and $T>0$, defined in Lemma~\ref{Number ineq}  \\

$D^\sharp(T,\theta)$ & Number of $\e$-approximates $(p,q)$ of $\theta \in \Mat$ satisfying $\|q\| \leq e^T$, where $T>0$ \\

$\tau_A$ & For a subset $A \subset \seda$ with $\mu_{\seda}(A) > 0$, the return time function $\tau_A: A \to \R_{>0}$ defined by $\tau_A(x) = \min\{t \in \R_{>0} : a_t x \in A\}$; see Section~\ref{sec: Time visits} for more details \\

$T_A$ & For a subset $A \subset \seda$ with $\mu_{\seda}(A) > 0$, the first return map $T_A: A \to A$ defined by $T_A(x) = a_{\tau_A(x)} x$, defined $\mu_{\seda}|_A$-almost everywhere; see Section~\ref{sec: Time visits} for more details \\

\end{longtable}}

\end{document}